\documentclass[aap]{imsart}

\RequirePackage{amsthm,amsmath,amsfonts,amssymb}
\RequirePackage[numbers]{natbib}

\startlocaldefs

\newcommand{\ac}[1]{\textcolor{red}{add citation}}

\usepackage{enumerate}
\usepackage{color}
\usepackage{multirow}\usepackage{graphicx}
\usepackage{amsfonts}
\usepackage{amsmath}
\usepackage{amssymb}

\usepackage{color}
\usepackage{dsfont}

\usepackage{pgfplots}
\usepackage{subfigure}
\usetikzlibrary{patterns}
\usepackage{amssymb}
\usepackage{graphicx}
\usepackage{amsmath}
\usepackage{bm}
\usepackage{algorithm} 
\usepackage{algorithmic}  
\usepackage[algo2e]{algorithm2e} 
\usepackage{caption}

\usetikzlibrary{decorations.pathreplacing}

\usetikzlibrary{calc,patterns}

\usepackage{url}
\usepackage{fancyhdr}
\usepackage{indentfirst}

\usepackage{enumerate}
\usepackage{dsfont}
\usepackage[british]{babel}
\usepackage[colorlinks=true,citecolor=blue]{hyperref}
\usepackage{amsthm}
\usepackage{color}

\usepackage{comment}
\usepackage{tikz-cd}
\usepackage{graphicx}
\usepackage{amsfonts}
\usepackage{amsmath}
\usepackage{amssymb}
\usepackage{url}
\usepackage{bbm}
\usepackage{fancyhdr}
\usepackage{indentfirst}
\usepackage{enumerate}
\usepackage{dsfont}
\usepackage[british]{babel}
\pgfplotsset{compat=1.7}
\usepackage[colorlinks=true,citecolor=blue]{hyperref}
\usepackage{cleveref}
\usepackage{amsthm}
\usepackage{color}
\renewcommand{\cite}{\citet}
\usepackage{comment}

\usepackage{algorithm}

\renewcommand{\d}{\,\mathrm{d}}
\newcommand{\dd}{\overset{\mathrm{law}}{=}}
\newcommand{\p}{\mathbb{P}}
\newcommand{\q}{\mathbb{Q}}
\newcommand{\qn}{\mathbb{Q}^{(n)}}
\newcommand{\bt}{\mathbf t}

\usepackage{xcolor}

\newcommand{\E}{\mathbb{E}}    
\newcommand{\R}{\mathbb{R}}    
\newcommand{\N}{\mathbb{N}}    

\theoremstyle{plain}
\newtheorem{theorem}{Theorem}
\newtheorem{corollary}[theorem]{Corollary}
\newtheorem{lemma}[theorem]{Lemma}
\newtheorem{proposition}[theorem]{Proposition}

\theoremstyle{definition}

\newtheorem{example}[theorem]{Example}
\newtheorem{conjecture}{Conjecture}

\newtheorem{remark}{Remark}

\usepackage{tikz}
\usetikzlibrary{arrows.meta,positioning}

\newcommand{\ee}{\varepsilon}

\newcommand{\n}[1]{\left\lVert#1\right\rVert}

\newcommand{\bz}{{\mathbf{z}}}
\newcommand{\bxi}{{\boldsymbol{\xi}}}
\newcommand{\bXi}{{\boldsymbol{\Xi}}}
\newcommand{\bzeta}{{\boldsymbol{\zeta}}}
\newcommand{\bth}{{\boldsymbol{\theta}}}

\newcommand{\tnu}{\widetilde{\nu}}

\newcommand{\bv}{\mathbf{v}}
\newcommand{\bu}{\mathbf{u}}
\newcommand{\bU}{\mathbf{U}}
\newcommand{\bV}{\mathbf{V}}
\newcommand{\bS}{\mathbb{S}}

\newcommand{\bx}{\mathbf{x}}
\newcommand{\by}{\mathbf{y}}
\newcommand{\bX}{\mathbf{X}}
\newcommand{\bY}{\mathbf{Y}}
\newcommand{\bJ}{\mathbf{J}}
\newcommand{\bl}{c_2}
\newcommand{\bc}{\mathbf{c}}
\newcommand{\bla}{\boldsymbol\lambda}

\renewcommand{\tnu}{\widetilde{\nu}}
\newcommand{\tl}{\widetilde{\kappa}}

\renewcommand{\bth}{\boldsymbol{\theta}}

\newcommand{\bn}{\mathbf{n}}

\newcommand{\bpsi}{\boldsymbol{\psi}}

\newcommand{\bet}{\boldsymbol{\eta}}

\def\d{\mathrm{d}}

\newcommand{\bone}{ {\mathbbm{1}} }
\renewcommand{\S}{\mathbb{S}}
\renewcommand{\bS}{\mathbf S}

\usepackage{mathrsfs}

\newcommand{\z}{\mathbf{0}}

\newcommand{\K}{\mathbb{K}}

\renewcommand{\le}{\leqslant}
\renewcommand{\geq}{\geqslant}
\renewcommand{\leq}{\leqslant}
\renewcommand{\epsilon}{\varepsilon}

\newcommand{\rrm}{{classical} BRW}

\endlocaldefs

\begin{document}

\begin{frontmatter}
\title{On the First Passage Times of Branching Random Walks \\ in {$\R^{\lowercase{d}}$}}

\runtitle{First Passage Times of Branching Random Walks}

\begin{aug}

\author[A]{\fnms{Jose}~\snm{Blanchet}\ead[label=e1]{jose.blanchet@stanford.edu}},
\author[B]{\fnms{Wei}~\snm{Cai}\ead[label=e2]{caiwei@stanford.edu}}
,
\author[B]{\fnms{Shaswat}~\snm{Mohanty}\ead[label=e3]{shaswatm@stanford.edu}}
\and
\author[C]{\fnms{Zhenyuan}~\snm{Zhang}\ead[label=e4]{zzy@stanford.edu}}
\address[A]{Department of Management Science and Engineering, Stanford University, CA 94305-4040, USA\printead[presep={,\ }]{e1}}

\address[B]{Department of Mechanical Engineering, Stanford University, CA 94305-4040, USA\printead[presep={,\ }]{e2,e3}}
\address[C]{Department of Mathematics, Stanford University, CA 94305-4040, USA\printead[presep={,\ }]{e4}}
\end{aug}

\begin{abstract}
We study the first passage times of discrete-time branching random walks in ${\mathbb R}^d$ where $d\geq 1$. Here, the genealogy of the particles follows a supercritical Galton-Watson process. 
We provide asymptotics of the first passage times to a ball of radius one at a distance $x$ from the origin, conditioned upon survival. 
We explicitly provide the linear dominating term and the logarithmic correction term as functions of $x$. 
The asymptotics are precise up to an order of $o_{\mathbb P}(\log x)$ for general jump distributions and up to $O_{\mathbb P}(\log\log x)$ for spherically symmetric jumps. 
A crucial ingredient of both results is the tightness of first passage times. 
We also discuss an extension of the first passage time analysis to a modified branching random walk model that has been proven to successfully capture shortest path statistics in polymer networks.

\end{abstract}



\end{frontmatter}


\section{Introduction}
We consider a discrete-time branching random walk (BRW) on $\R^d$, where particles in $\R^d$ perform independent random walks until they randomly branch into more particles or die. 
Spatial branching processes, including BRW and its continuous-time sibling branching Brownian motion (BBM), have a long history and a wide range of applications in ecology, population biology, and modeling epidemics (\citep{ermakova2019branching,fisher1937wave,kolmogorov1937etude,konig2020branching,kot2004stochasticity,wang1980convergence}). 
The study of the extremal behavior of spatial branching processes has gained increasing attention over recent years. The mathematical techniques applied in those studies are related to developments in numerous other fields, e.g., 
discrete Gaussian free field \citep{bramson2016convergence2,ding2014extreme},
random unitary matrices \citep{arguin2017maximum},
Riemann zeta function \citep{arguin2017maxima},
random multiplicative functions \citep{harper2020moments},
spin glasses \citep{bovier2017gaussian}, 
holomorphic multiplicative chaos \citep{gu2024universality}, 
among many others that exhibit a log-correlated structure. We refer to \citep{shi2015branching,zeitouni2016branching} for lecture notes on BRW and related topics.

The study of the extrema for one-dimensional BBM traces back to \citep{bramson1978maximal}, which established a precise asymptotic for the maximum as a function of time. For general branching random walks, \citep{addario2009minima} provided a precise asymptotic in terms of the jump distribution of the process. Finer behavior near the frontier is also well understood since the works of \citep{aidekon2013convergence,bramson2016convergence,lalley1987conditional,madaule2017convergence}. 
Nevertheless, less is known in dimensions greater than one. To name a few examples, there are studies of the maximum norm for BBM (\citep{berestycki2024extremal,kim2023maximum,kim2024shape,mallein2015maximal,stasinski2021derivative}) and very recently of BRW (\citep{bezborodov2023maximal}), as well as the range of spatial branching processes (\citep{kyprianou2002asymptotic,zhang2024large}). These works in higher dimensions typically assumed that the increment distribution is spherically symmetric. A notable exception is \citep{uchiyama1982spatial}, which provided estimates on the number of particles in a linearly shifted ball in time for a continuous-time version of the BRW. 

Our primary focus is on the first passage times (FPT) for branching random walks in $\R^d$. By definition, the FPT is the first time for particles to reach a prescribed region. For BBM, \citep{zhang2024modeling} established a precise asymptotic formula (up to $O_\p(1)$) for the FPT; see \eqref{eq:bbm} below. 
In this paper, we provide FPT asymptotics for BRW that may not have a spherically symmetric jump distribution (while the case of spherically symmetric jumps is of special interest for our applications; see Section \ref{sec:appl}). The main results are explained in Section \ref{sec:main results} and their proofs are presented in Section \ref{app:proof_thm1}.
 Applications to polymer physics will be briefly discussed in Section \ref{sec:appl} and further detailed in Section \ref{sec:newbrwproof}.

 We take slightly different routes depending on whether the jump distribution is spherically symmetric or not. In the symmetric case, a key part is establishing asymptotic upper and lower bounds on the number of particles that are near the frontier for a one-dimensional BRW (Proposition \ref{prop:number of particles}), where our approach follows the modified second moment method of \citep{bramson2016convergence}. In the non-spherically symmetric case, our argument relies heavily on the theory of random walks in cones \citep{denisov2015random}. In both settings, we prove the tightness around the median of the FPT (Theorem \ref{theorem:concentration}) using independent techniques, building upon ideas from \citep{mcdiarmid1995minimal}. A more detailed proof sketch can be found in Section \ref{sec:proof stratgey}.

\subsection{Notation and assumptions} \label{sec:fpt_for_brw}

Let us start with a formal definition of the branching random walk (BRW). We consider a discrete-time BRW model with offspring distribution $\{p_i\}_{i\geq 0}$. This means that at each time step, each particle is independently replaced by $i$ particles at the same location with probability $p_i$. The particles in each generation have i.i.d.~displacements. The case of $i=0$ corresponds to the particle being terminated. Let $\rho=\sum_i i\,p_i$ be the mean of the offspring distribution. We will always assume the supercritical case $\rho>1$. 
The  $d$-dimensional random walk increment is given by $\bxi$, and the first particle starts from the origin $\z\in\R^d$. Typically, in this paper, a bold symbol refers to a vector. 

Our goal is to study the asymptotics of the FPTs of the BRW. For $x\in\R$, we let $B_x$ denote a ball of radius one\footnote{Replacing the radius one by any fixed positive constant does not change the analysis; the same proof goes through.} centered at $(x,0,\dots,0)$ in $\R^d$. We let $V_n$ denote the collection (i.e., set) of particle locations at time step $n$, and  $\{\bet_{v,n}(k)\}_{0\leq k\leq n}$ denotes the $d$-dimensional random walk that leads to $v\in V_n$. We define the FPT  $\tau_x$  of the BRW to $B_x$ as
$$\tau_x=\min\{n\geq 0 : \exists\, v\in V_n,~ \bet_{v,n}(n)\in B_x\}.$$

In what follows, we will separately analyze the cases depending on whether the law of $\bxi$ is spherically symmetric or not.\footnote{A probability distribution on $\R^d$ is \emph{spherically symmetric} if it is invariant under any orthonormal transformation. Equivalent definitions can be found in Theorem 2.5 of \citep{fang2018symmetric}.} In both cases, we impose the following assumptions on the BRW:
 \begin{itemize}
     \item [(A1)] the offspring distribution has a finite second moment, i.e., $\sum_i i^2\,p_i<\infty$;\label{A1}
     \item [(A2)] the law of $\bxi$ is integrable and centered, i.e., $\E[\bxi]=\z$.
 \end{itemize}
For a stochastic process $\{T_x\}_{x\geq 0}$ and a positive deterministic map $\{t_x\}_{x\geq 0}$, we write $T_x=O_\p(t_x)$ if the collection $\{T_x/t_x\}_{x\geq 0}$ is tight, and   $T_x=o_\p(t_x)$ if $T_x/t_x\to 0$ in probability.

 \subsubsection{Spherically symmetric jumps}\label{sec:notations-sphsym}
Denote the first coordinate of $\bxi$ by $\xi^{(1)}$, which is a real-valued random variable. We  introduce the large deviation rate function
\begin{align}
    I(x):=\sup_{\lambda>0}\Big(\lambda x-\log\phi_{\xi^{(1)}}(\lambda)\Big),\label{eq:ratef}
\end{align} 
where $\phi_{\xi^{(1)}}(\lambda):=\E[e^{\lambda{\xi^{(1)}}}]$ is the moment generating function for ${\xi^{(1)}}$.
In this case, we make the following assumptions on $\bxi$:
\begin{itemize}
    \item [(A3)] the law of $\bxi$ is spherically symmetric in $\R^d$, and its radial component $R=|\bxi|$ satisfies $\p(R=0)<1$;
    \item [(A4)] $\log\rho\in(\mathrm{ran}I)^\circ$, where $(\mathrm{ran}I)^\circ$ is the interior of the range of $I$.  In other words, there exists $c_1 > 0$ such that $I(c_1) = \log \rho$.  It can be shown that $c_1\in (\mathrm{ran}(\log\phi_{\xi^{(1)}})')^\circ$.
   \end{itemize}
    Note that under assumptions (A2) and (A3), $I$ is $C^1$ and convex in its domain, and attains a unique global minimum at $x=0$.\footnote{See e.g., Theorem 26.3 of \citep{rockafellar1970convex}.}  We will also see from Lemma \ref{lemma:edgeworthok} below that (A3) implies that both $\bxi$ and ${\xi^{(1)}}$ are non-lattice if $d\geq 2$.\footnote{We say the law of $\bxi$ is \emph{non-lattice} if there is no $\by\in\R^d$ such that the support of $\bxi+\by$ is contained in a discrete subgroup of $\R^d$.} 
 A consequence of (A4) is that $\rho>1$, and hence the particles form a supercritical Galton-Watson process. 
  Denote by $\bl:=I'(c_1)$  the value of $\lambda$ at which the supremum in \eqref{eq:ratef} is taken for $x=c_1$. In particular, (A4) implies that $\phi_{\xi^{(1)}}$ is well-defined in a neighborhood of $c_2$. The constants $c_1$ and $c_2$ here depend only on the offspring distribution (through $\rho$) and the increment distribution (through $\phi_{\xi^{(1)}}(\lambda)$) and will be fixed throughout this paper.

\subsubsection{Non-spherically symmetric jumps}

Denote by $\phi_\bxi(\bla)=\E[e^{\bla\cdot\bxi}]$ the moment generating function for $\bxi$. 
Let 
\begin{align}
    I(\bx):=\sup_{\bla\in\R^d}\Big(\bla\cdot\bx-\log\phi_\bxi(\bla)\Big)=\sup_{\bla\in\R^d}\Big(\bla\cdot\bx-\log\E[e^{\bla\cdot\bxi}]\Big)\label{eq:I long}
\end{align}
denote the large deviation rate function for $\bxi$. We impose the following assumptions:
\begin{itemize}
    \item [(A5)] $\bxi$ satisfies the strong non-lattice condition $\limsup_{|\bt|\to\infty}|\E[e^{i\bt\cdot \bxi}]|<1$;
    \item [(A6)] $\log\rho\in(\mathrm{ran}I(\cdot,\z))^\circ$, where $I(\cdot,\z)$ refers to the function $I$ with the last $d-1$ variables fixed at zero; let $\widehat{c}_1>0$ satisfy $I(\widehat{c}_1,\z)=\log\rho$, then $(\widehat{c}_1,\z)\in (\mathrm{ran}\nabla\log\phi_\bxi)^\circ$.
   \end{itemize}
   Denote by $\bc_2=\nabla I(\widehat{c}_1,\z)$, which is the value of $\bla$ where the supremum \eqref{eq:I long} is attained at $\bx=(\widehat{c}_1,\z)$. The intuition behind the vector $\bc_2$ is explained by Figure \ref{fig:nonsphere} below. 
The assumption (A5) is crucial for the non-degeneracy of the jumps. More precisely, (A5) implies that the support of $\bxi$ cannot be contained in a proper subspace of $\R^d$. This is necessary for the BRW to be able to reach the target ball $B_x$ and shows that $d$ is the `real' dimension of the increments. A consequence of (A6) is that $\phi_\bxi$ is well-defined in a neighborhood of $\bc_2$. We will see in Proposition \ref{prop:c1} that (A3) and (A4) together imply (A6) with $\widehat{c}_1=c_1$.

\begin{figure}
    \centering
    \includegraphics[width=0.7\textwidth]{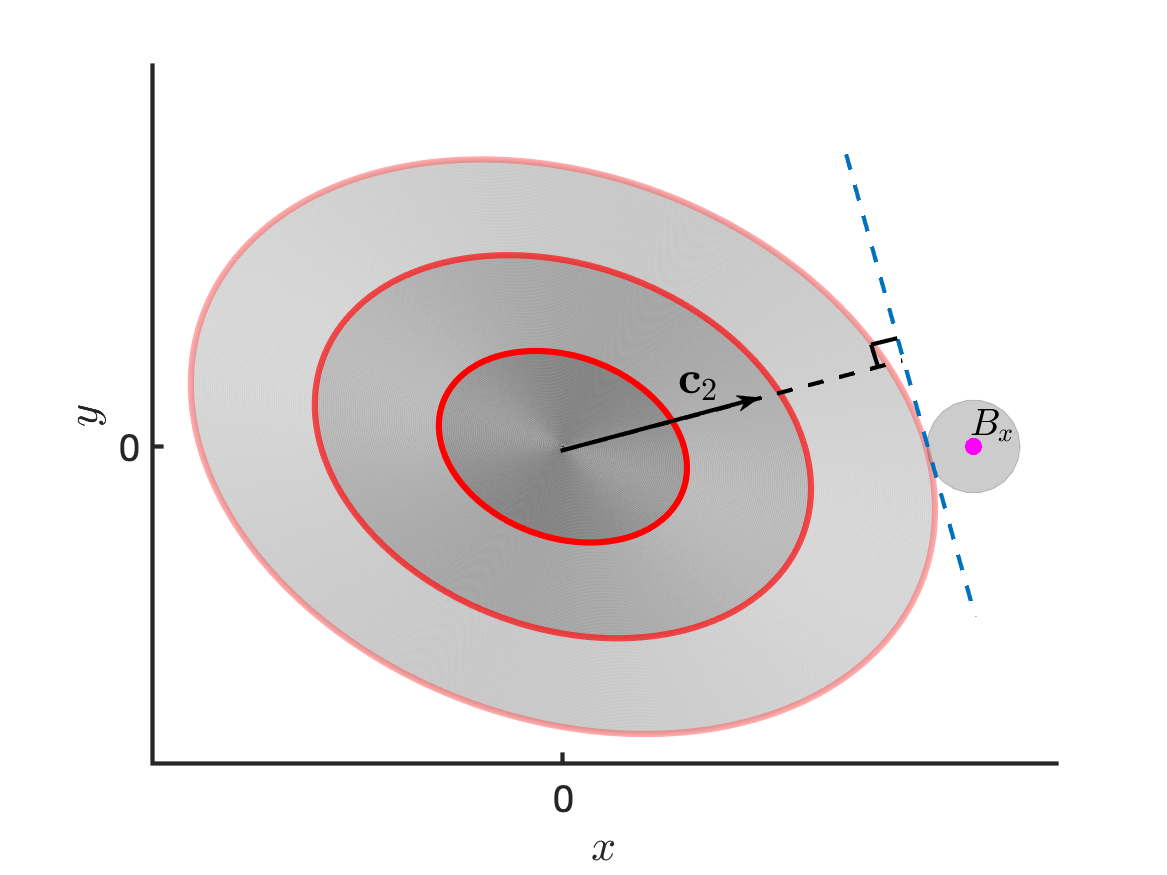}
    \caption{A two-dimensional visualization of the vector $\bc_2$. The range of a (non-spherically symmetric) BRW roughly grows linearly in time (illustrated with the shaded ellipses); at time $\tau_x$, the ellipse is nearly tangent to the target ball $B_x$ (shaded disk). Roughly speaking, the vector $\bc_2$ is normal to the tangent line. }
    \label{fig:nonsphere}
\end{figure}

\subsection{Main results}
\label{sec:main results}

In the following, we fix a dimension $d\geq 1$. Consider first the case of spherically symmetric jumps. 
Let us define
\begin{align}
    A(x) := \frac{x}{c_1}+\frac{d+2}{2\bl c_1}\,\log x.
    \label{eq:A}
\end{align}
For a finite set $G$, we use $\#G$ to denote its cardinality. Denote by $S:=\{\forall n\geq 1,~\#V_n>0\}$ the \emph{survival event} that the underlying branching process survives at all times. Since $\rho>1$, we have $p:=\p(S)>0$.
\begin{theorem}\label{thm:main}
Assume (A1)--(A4). Conditional upon survival, it holds that
\begin{align}
    \tau_x=A(x)+O_\p(\log\log x)=\frac{x}{c_1}+\frac{d+2}{2\bl c_1}\,\log x+O_\p(\log\log x).\label{eq:tauxasymp}
\end{align}
In other words, the collection $\{(\tau_x-A(x))/\log\log x\}_{x>0}$ is tight.
\end{theorem}

A natural question is whether the $O_\p(\log\log x)$ can be replaced by a $O_\p(1)$, which we later formulate as Conjecture \ref{conjecture}.\footnote{We remark that the conjecture has been resolved by \citep{blanchet2024tightness} after the first version of this paper was posted; see Remark \ref{rem:conj}.} We discuss a few pieces of evidence that support this conjecture, and further aspects will be discussed in Section \ref{sec:conclusion}. First, this is the case for $d=1$ (Proposition \ref{prop:1dfpt} below) and for the branching Brownian motion in $\R^d$ (Theorem 1 of \citep{zhang2024modeling}). A second piece of evidence is the following tightness result on first passage times that is also central to our proof of Theorem \ref{thm:main}. 

\begin{theorem}
    \label{theorem:concentration}

    For $r\in(0,1)$ and $x>0$, let $t^{(r)}_x$ be the $r$-th quantile of the first passage time $\tau_x$, i.e., $t^{(r)}_x=\inf\{t: \p(\tau_x\leq t\mid S)\geq r\}$. Assume (A1), (A2), and (A6). There exist constants $C,c>0$ independent of $x$ such that for each $y\in[0,x]$,
    \begin{align}
        \p\left(|\tau_x-t^{(1/2)}_x|>y\mid S\right)\leq Ce^{-cy}.\label{o}
    \end{align}
    In particular, conditioned upon survival, the collection of laws $\{\tau_x-t^{(r)}_x\}_{x>0}$ is tight for each $r\in(0,1)$.
\end{theorem}

The analog of Theorem \ref{theorem:concentration} for the maximum of BRW traces back to \citep{dekking1991limit} and \citep{mcdiarmid1995minimal}, which assumed uniformly bounded jumps. The case of unbounded jumps was later resolved by \citep{bramson2007tightness} and \citep{bramson2009tightness} using a different approach. Theorem \ref{theorem:concentration} generalizes these results to the first passage times of BRW in arbitrary dimensions. In this setting, for the jumps, we only require mean zero and assumption (A6). Since (A3) and (A4) together imply (A6), Theorem \ref{theorem:concentration} applies in the spherically symmetric case as well.

Extending these results to the case of non-spherically symmetric jumps, we define
\begin{align}
    \widehat{A}(x):=\frac{x}{\widehat{c}_1}+\frac{d+2}{2\,\widehat{c}_1\partial_{x_1} I(\widehat{c}_1,\z)}\,\log x,\label{eq:hatA}
\end{align}
where we recall the definition of $\widehat{c}_1$ from assumption (A6).
\begin{theorem}\label{thm:main long}
 Assume (A1), (A2), (A5), and (A6). Conditional upon survival,  
\begin{align}
    \tau_x=\widehat{A}(x)+o_\p(\log x)=\frac{x}{\widehat{c}_1}+\frac{d+2}{2\widehat{c}_1\partial_{x_1} I(\widehat{c}_1,\z)}\,\log x+o_\p(\log x).\label{eq:tauxasymp long}
\end{align}
In other words, the collection $(\tau_x-\widehat{A}(x))/\log x$ converges to $0$ in probability as $x\to\infty$.
\end{theorem}
There are certain instances for non-spherically symmetric jumps where the asymptotic \eqref{eq:tauxasymp long} can be improved. Such a class includes product measures (meaning that $\bxi$ has independent coordinates) and non-degenerate centered elliptical distributions on $\R^d$.\footnote{We say that a centered random variable $\bxi$ follows a \emph{non-degenerate elliptical distribution} if $\bxi\dd TU$ where $U$ is spherically symmetric and $T$ is an invertible linear transformation.} These will be addressed in detail in Section \ref{sec:non-spherical}.

An immediate corollary of Theorems \ref{theorem:concentration} and \ref{thm:main long} is the following strong law of large numbers for the first passage times.
\begin{corollary}\label{coro:slln}
    Assume (A1), (A2),  (A5), and (A6). Conditioned upon survival,
     $$\frac{\tau_x}{x}\to\frac{1}{\widehat{c}_1}\qquad \text{a.s.}
     $$
\end{corollary}

As yet another consequence of Theorems \ref{theorem:concentration} and \ref{thm:main long} (while slightly modifying their proofs), the full range of the BRW forms a dense subset of $\R^d$, complementing Corollary 2.5 of \citep{oz2023} that the full range of the BBM is dense.  

\begin{corollary}\label{coro:density}
Assume (A1), (A2),  (A5), and (A6), and that the MGF $\phi_{\bxi}$ is well-defined in a neighborhood of $\z$. Conditioned upon survival, the set 
$$R:=\bigcup_{n\geq 0}\bigcup_{v\in V_n}\{\bet_{v,n}(n)\}$$ is dense in $\R^d$ almost surely. 
    
\end{corollary}

In Section \ref{sec:numerical}, we test Theorems \ref{thm:main} and \ref{thm:main long} numerically with a \emph{path purging} algorithm that significantly improves the computational efficiency for capturing the extremal behavior of a BRW.

\subsection{Proof strategy}\label{sec:proof stratgey}
In this subsection, we discuss the high-level intuition of the strategy we will follow, building from the case of $d=1$. 
Throughout, we condition on the survival event. We first assume that the jump distribution is spherically symmetric. 
For one-dimensional BRW, we expect the following parity relation between maximum and FPT. Recall that the asymptotic for the maximum $M_n$ at time $n$ is 
\begin{align}
 M_n=m_n+O_\p(1),\quad\text{ where }\quad m_n:=c_1n-\frac{3}{2\bl}\,\log n.\label{eq:one-dim asymp}
\end{align}
Inverting \eqref{eq:one-dim asymp} leads to the asymptotic (e.g., using  the Lambert $W$ function)
$$n=\frac{M_n}{c_1}+\frac{3}{2\bl c_1}\,\log M_n+O_\p(1).$$
It is then natural to guess that for one-dimensional BRW,
\begin{align}
    \tau_x=t_1(x)+O_\p(1),\quad\text{ where }\quad t_1(x):=\frac{x}{c_1}+\frac{3}{2\bl c_1}\,\log x,\label{eq:1dasymp}
\end{align}
which we will prove in Proposition \ref{prop:1dfpt} below in Section \ref{sec:d=1}. Recall \eqref{eq:A}. Let us decompose $A(x)=t_1(x)+t_2(x)$, where
\begin{align}
t_2(x):=A(x)-t_1(x)=\frac{d-1}{2\bl c_1}\,\log x.\label{eq:m1m2}
\end{align} 
For a general dimension $d\geq 1$, we shall explain the extra term $t_2(x)$.

The main difficulty in analyzing the first passage times in higher dimensions is that the second moment method does not immediately apply if one counts the number of particles in the target ball $B_x$ at a certain time. 
To bypass this difficulty, we prove separately the upper and lower bounds for $\tau_x$. 

A key step in these bounds is Proposition \ref{prop:number of particles}, which keeps track of the number of particles that are near the frontier (say, a distance of $\log n$ away) for a one-dimensional BRW. Roughly speaking, there are around $xe^{c_2x}$ particles that are of distance $x$ behind the frontier, where $x\ll\log n$.  This amounts to generalizing the proof of the asymptotic of the maximum \eqref{eq:one-dim asymp}; see \citep{bramson2016convergence}. A notable distinction is that unfortunately, one cannot conclude from the second moment method the lower bound for the number of particles near the frontier with overwhelming probability. Instead, one only concludes the lower bound with a uniformly positive probability. This gap can be filled using Theorem \ref{theorem:concentration}, which shows that a uniformly positive probability would already be sufficient.

A particularly nice feature of the spherical symmetry of the jumps is that the displacements in different directions are approximately independent (likewise, the same result Theorem \ref{thm:main} holds if we instead assume independent centered jumps in the $d$ directions). In particular, at time $n$, one roughly expects a proportion $O(n^{-(d-1)/2})$ of the particles that lie within unit distance from the origin in the last $d-1$ dimensions. In other words, one expects that the first passage event is realized around the time when $x^{(d-1)/2}$ particles reach a distance of $x$ for the $\R^d$-valued BRW projected onto the first dimension; see Figure \ref{fig:schematic_paper2} below. On the other hand, by \eqref{eq:one-dim asymp}, the extra time $t_2(x)$ from \eqref{eq:m1m2} pushes the frontier ahead by a distance of approximately $c_1t_2(x)=(d-1)(\log x)/(2c_2)$, which amounts to having approximately $ x^{(d-1)/2}$ extra particles in view of Proposition \ref{prop:number of particles}.


\begin{figure}[ht!] 
\centering
\includegraphics[width=\textwidth]{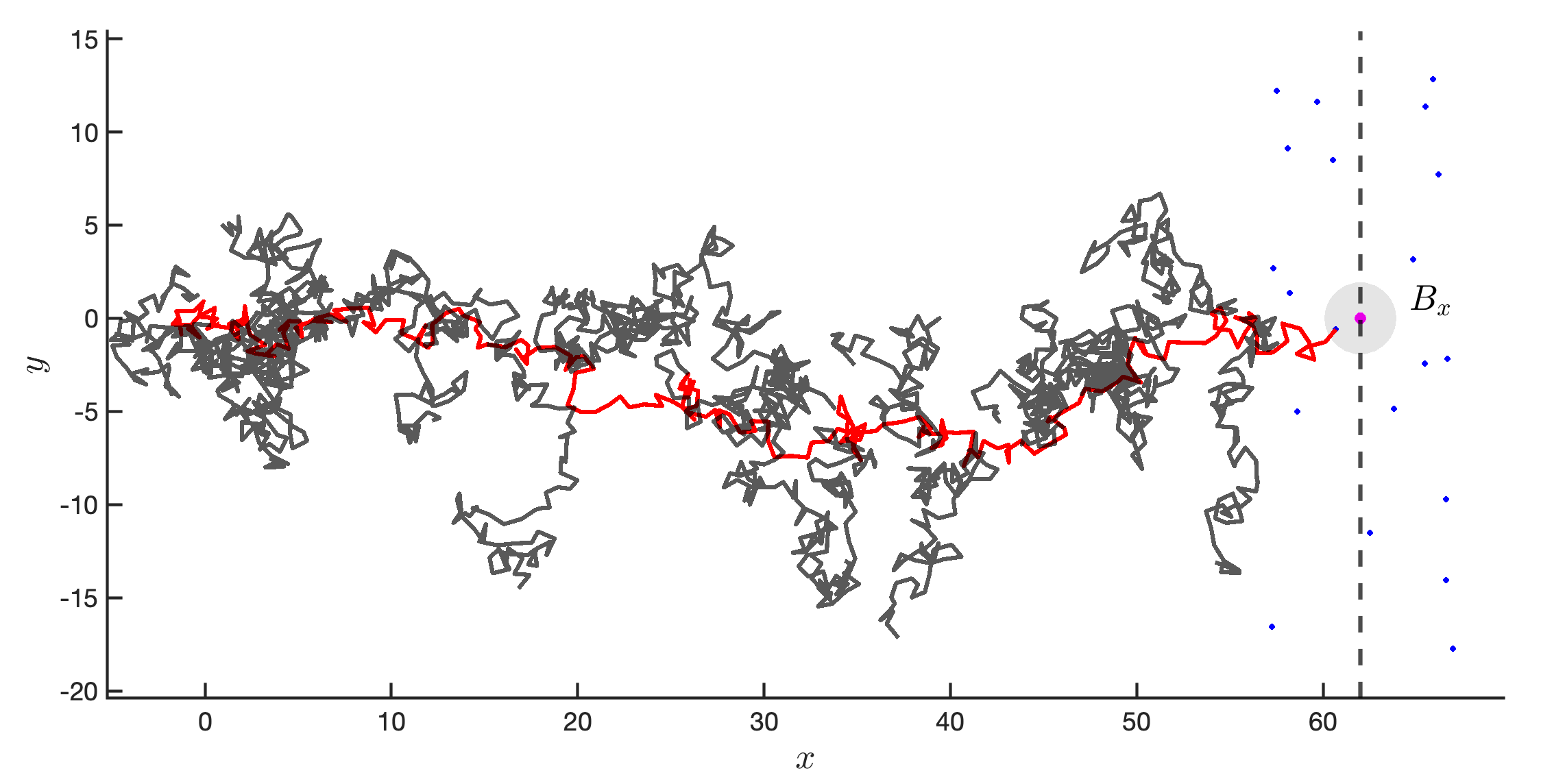}
    \caption{Schematic description of the heuristic for spherically symmetric jumps: around a proportion $x^{-(d-1)/2}$ of the particles that reach roughly $x$ far in the first coordinate (represented by blue dots) lie in the ball $B_x$. The path in red represents (part of) the trajectory that realizes the FPT, whereas the other branches are shown in black.}
    \label{fig:schematic_paper2}
\end{figure}

If the jump distribution is not spherically symmetric, applying similar arguments above (with Proposition \ref{prop:number of particles} replaced by Theorem B of \citep{biggins1979growth} and Proposition \ref{prop:ulb gaussian approx} by Proposition \ref{prop:ulb gaussian approx long}; we omit the details here) will lead to 
\begin{align}
    \tau_x=\frac{x}{\widehat{c}_1}+O_\p(\log x),\label{eq:weak taux}
\end{align}where the corresponding upper and lower bounds for $\tau_x$ are separated by an order of $\log x$.  
To pin down the coefficient of the logarithmic correction term, we extend the arguments for the maximum of one-dimensional BRW (\citep{bramson2016convergence})  with a modified barrier. 
The modified barrier event lies in $\R^d$ and hence results on random walks in cones (\citep{denisov2015random}) are applicable. Note that the same proof also applies for the spherically symmetric case, but would yield less precise asymptotics compared to \eqref{eq:tauxasymp} (an error of $o_\p(\log x)$ instead of $O_\p(\log\log x)$).

\subsection{Applications to polymer physics} \label{sec:appl} 
Our interest in estimating the FPT of spatial branching processes stems from their relevance to polymer physics. 
The mechanical behavior of polymer networks appears to be controlled by the shortest paths between distant nodes in the polymer network, as shown by recent works \citep{yin2020topological,yin2024network} using coarse-grained molecular dynamics (CGMD) simulations.
The shortest path length (SP) serves as an indicator of how stretched the average load-bearing chains are in the polymer network.
Theories based on SP go beyond traditional polymer models focused on elastic behaviors and aim to explain the inelastic behaviors exhibited by such materials.

The SP between (fixed) distant nodes is equivalent to the FPT of the BRW beginning from the origin (source node) and terminating at a desired point in space (destination node). Moreover, the occurrence of the cross-links along the polymer backbone mimics a Poisson process. As a result, a cross-link between two chains is equivalent to a branching event observed in the BRW genealogy tree. The agreement of the mean FPT from the BRW theory with the mean SP from the CGMD simulation results has been shown in our recent work~\citep{zhang2024modeling}.
Furthermore, the numerical implementation of the BRW model shows an excellent agreement between its FPT distribution and the SP distribution from the CGMD simulation. Our earlier work initially focused on constant length jumps with the BRW model where there existed a correlation between the components of a jump step in the $d$ dimensions by virtue of being sampled from the unit sphere $\S^{d-1}=\{\bx\in\R^d\mid\n{\bx}=1\}$. This was later shown to be represented equally well by the Gaussian BRW and BBM in which the components of each jump step are independent in the $d$ dimensions. However, the analysis was restricted to spherically symmetric jumps to accurately model the as-prepared polymer network prior to deformation. 

\section{Proofs of main results} \label{app:proof_thm1}
This section is devoted to the proofs of the main results. We will first analyze the simple case of $d=1$ in Section \ref{sec:d=1}, where a precise asymptotic of the FPT follows directly from inverting the maxima asymptotic \eqref{eq:one-dim asymp}. We then collect the necessary tools in Sections \ref{sec:number of particles} and \ref{sec:rw}, and prove Theorems \ref{thm:main} and \ref{theorem:concentration} in Sections \ref{sec:proof} and  \ref{sec:concentration}  respectively.  
Most of the arguments before and including Section \ref{sec:number of particles} closely follow the approach of \citep{bramson2016convergence}.\footnote{The results in \citep{bramson2016convergence} are stated for the case $p_0=0$, but the general case follows in a similar way.} 
The case of non-spherically symmetric jumps (Theorem \ref{thm:main long}) will be addressed in Section \ref{sec:proof long}, where we also summarize preliminary results on random walks that are constrained to stay in a cone. Section \ref{sec:coroproof} presents the proofs of Corollaries \ref{coro:slln} and \ref{coro:density}.

\sloppy We begin by introducing some necessary notations. Following \citep{bramson2016convergence}, we denote by $V_n$ the offspring on the $n$-th generation. Let $\{\bet_{v,n}(k)\}_{k=0,\dots,n}$ denote the $d$-dimensional random walk where $\bet_{v,n}(k)$ is the position of the $k$-th generation for the path in the tree leading to a vertex $v\in V_n$. We also let $\eta^{(1)}_{v,n}(k)$ denote the first coordinate of $\bet_{v,n}(k)$. When $d=1$, these notions coincide.
 For deterministic quantities or functions $A,B$, We use Vinogradov's symbol $A\ll B$ (or $A=O(B)$) to denote $|A|\le CB$ with some constant $C>0$ that depends only on the distribution of $\bxi$ and the dimension $d$. We write $A\asymp B$ if $A\ll B\ll A$. The notation $\n{\cdot}$ always denotes the Euclidean norm. 
We use $B_\bx(r)$ to denote the shifted ball of radius $r$ centered at $\bx\in\R^d$, $B_\bx=B_\bx(1)$, and $B_x=B_{(x,\z)}$ for $x\in\R$.
From here on, without loss of generality, we will assume that quantities that tend to infinity are integers. The arguments can be made rigorous by writing floor functions everywhere, but we have chosen not to do so for the simplicity of the notation.

\subsection{Warming up in dimension one}\label{sec:d=1}
For $\beta\geq 1$ and $n\in\N$ we define the event
$$G_{n,\beta}:=\bigcup_{v\in V_n}\bigcup_{0\leq k\leq n}\left\{\eta^{(1)}_{v,n}(k)\geq \frac{km_n}{n}+\beta+\frac{4}{\bl}(\log(\min(k,n-k)))_+\right\}.$$
\begin{lemma}[Lemma 2.4 of \citep{bramson2016convergence}]\label{lemma:beyondcurve}
  Assume conditions (A1), (A2), and (A4), and assume that the jump distribution is non-lattice. It holds that for all $n\in\N$, $\p(G_{n,\beta})\ll \beta e^{-\bl\beta}$ for $\beta\geq 1$.
\end{lemma}
\begin{proposition}\label{prop:1dfpt}
    Suppose that a one-dimensional BRW satisfies conditions (A1), (A2), and (A4), $p_0=0$ (no extinction) and that the jump distribution is non-lattice. We have 
    $$\tau_x=t_1(x)+O_\p(1)=\frac{x}{c_1}+\frac{3}{2\bl c_1}\,\log x+O_\p(1).$$
\end{proposition}

\begin{proof}    
    For the lower bound, we fix $\ee>0$ and pick $\beta$ large enough such that $\p(G_{n,\beta})<\ee$ for all $n\in\N$, by means of Lemma \ref{lemma:beyondcurve}. In other words, with a probability larger than $1-\ee$, the whole BRW stays below the curve 
    $$k\mapsto \frac{km_n}{n}+\beta+\frac{4}{\bl}(\log(\min(k,n-k)))_+.$$
    It is clear that with $C_1=C_1(\beta)>0$ chosen large enough, 
    $$\max_{1\leq k\leq t_1(x)-C_1}\Big(\frac{km_{t_1(x)-C_1}}{t_1(x)-C_1}+\beta+\frac{4}{\bl}(\log(\min(k,t_1(x)-C_1-k)))_+\Big)\leq x.$$
    This yields $\p(\tau_x\leq t_1(x)-C_1)\leq \p(G_{n,\beta})<\ee$, proving the lower bound.  


    For the upper bound, we fix $\ee>0$. 
Pick $C_2>0$ such that $\p(M_{t_1(x)+C_2}>x)>1-\ee/3$ uniformly in $x$ by \eqref{eq:one-dim asymp} and the convergence in law of the recentered maximum (Theorem 1.1 of \citep{bramson2016convergence}). Using the same analysis for the lower bound, there exists $C_3>0$ such that $\p(M_{t_1(x)+C_2}>x+C_3)<\ee/3$. Therefore, with probability $>1-2\ee/3$, there is $v\in V_{t_1(x)+C_2}$ such that $ x\leq \eta^{(1)}_{v,t_1(x)+C_2}(t_1(x)+C_2)\leq x+C_3$.
 Since the distribution of ${\xi}$ is non-lattice, by the Chung--Fuchs theorem (Theorem 5.4.8 of \citep{durrett2019probability}), the random walk with step distribution $\xi$ initiated from $v$ is recurrent, and hence there exists $C_4>0$ such that the descendants of $v$ reaches $[x-1/2,x+1/2]$ with probability $>1-\ee/3$ within time $C_4$. Combining the above leads to $\p(\tau_x>t_1(x)+C_2+C_4)<\ee$, as desired.
\end{proof}

\subsection{Proof of Theorem \texorpdfstring{\ref{thm:main}}{}}\label{sec:thm1 proof}
We now present the proof of Theorem \ref{thm:main} for $d\geq 2$, provided that Theorem \ref{theorem:concentration} holds true. The proof of Theorem \ref{theorem:concentration} will be the focus of the next section. The following auxiliary results are necessary as explained in our proof strategy.

\subsubsection{The number of particles  near maximum}\label{sec:number of particles}

The goal of this section is to prove the following result on the number of particles near the frontier of a one-dimensional BRW, which might be of independent interest.

\begin{proposition}\label{prop:number of particles}Assume (A1)--(A4). There exists $L>1$ depending only on the law of the BRW such that the following holds conditioned upon survival. Given any $\ee>0$, there exists $C>0$ independent from $n$ and $x$   such that  uniformly for $n$ large enough and for $x\in[2,\sqrt{n}]$, 
    \begin{align}
        \p\left(\#\{v\in V_n : \eta^{(1)}_{v,n}(n)\geq m_n-x\}>Cxe^{\bl x}\mid S\right)<\ee\label{eq:ub number}
    \end{align}
    and 
    \begin{align}
        \p\left(\#\{v\in V_n : \eta^{(1)}_{v,n}(n)\geq m_n-x\}>\frac{1}{C}xe^{\bl x}\mid S\right)>\frac{1}{L}.\label{eq:lb number}
    \end{align}
\end{proposition}

\begin{remark}\label{remark:number of particles}
Proposition \ref{prop:number of particles} may be compared against \citep{biggins1979growth,rouault1993precise}, where it is proven that for a location that is $\Omega(n)$ away from $m_n$, the number of particles therein has the same order as its expectation. Proposition \ref{prop:number of particles} shows that this is not the case for locations that are $O(\sqrt{n})$ near $m_n$. 
    On the other hand, the number of particles that are $O(\sqrt{n})$ away from the origin can be described by the BRW central limit theorem \citep{biggins1990central}. An interesting question would be to analyze the phase transitions between those regimes. We also expect that \eqref{eq:lb number} holds if we replace its right-hand side by $1-\ee$, but the current version suffices for our purpose. 
\end{remark}

In the rest of this section, we assume that $\p$ is the probability measure conditioned upon the survival event $S$. 
Following the setup of Section 2.2 of \citep{bramson2016convergence}, we introduce $\lambda_n$ which is the value of $\lambda$ where the supremum of \eqref{eq:ratef} is attained with $x=m_n/n$. By (15) of \citep{bramson2016convergence}, it holds $0\leq \bl-\lambda_n\ll (\log n)/n$. 
For a given $v\in V_n$, let $\qn$ be a probability measure on length-$n$ random walk paths defined by
\begin{align}
    \frac{\d\p}{\d\qn}:=e^{-\lambda_n(\eta^{(1)}_{v,n}(n)-m_n)-nI(m_n/n)}\asymp n^{3/2}\rho^{-n}e^{-\lambda_n(\eta^{(1)}_{v,n}(n)-m_n)}.\label{eq:dpdqn}
\end{align}
  It follows that under $\qn$, $\{\eta^{(1)}_{v,n}(k)-km_n/n\}_{k=0,\dots,n}$
is a mean zero random walk. 
Let 
$$Q_{n,\beta}:=\left\{v\in V_n : \text{ for any }0\leq k\leq n,~\eta^{(1)}_{v,n}(k)<\frac{km_n}{n}+\beta+\frac{4}{\bl}(\log(\min(k,n-k)))_+\right\}$$
and define 
\begin{align}
   g(n,\beta,x):={\E[\#\{v\in Q_{n,\beta} : \eta^{(1)}_{v,n}(n)\geq m_n-x\}]}.\label{eq:gndef}
\end{align}
We first reduce the proof of the upper bound \eqref{eq:ub number} into proving the following estimate on $g(n,\beta,x)$.
\begin{proposition}\label{prop:gn}Assume (A1)--(A4). 
Uniformly in $x\in[2,\sqrt{n}]$, $$g(n,\beta,x)\ll \beta (x+\beta)e^{\bl x}.$$
\end{proposition}
\begin{proof}[Proof of \eqref{eq:ub number} given Proposition \ref{prop:gn}] 
    For fixed $C,\ee>0$, we pick $\beta$ large enough so that $\p(G_{n,\beta})<\ee/2$ by Lemma \ref{lemma:beyondcurve}. On the event $G_{n,\beta}^c$, $Q_{n,\beta}=V_n$. By Proposition \ref{prop:gn}, uniformly for $n$ and $x\in[2,\sqrt{n}]$,
    $$\E\left[\#\{v\in V_n : \eta^{(1)}_{v,n}(n)\geq m_n-x\}\bone_{G_{n,\beta}^c}\right]\ll \beta (x+\beta)e^{\bl x}.$$
    Markov's inequality then yields that for some $L>0$,
    $$\p\left(G_{n,\beta}^c\cap\left\{\#\{v\in V_n : \eta^{(1)}_{v,n}(n)\geq m_n-x\}>L\beta (x+\beta)e^{\bl x}\right\}\right)<\frac{\ee}{2}.$$This yields \eqref{eq:ub number}.
\end{proof}



\begin{proof}[Proof of Proposition \ref{prop:gn}] 
 Let $\{S_k\}_{1\leq k\leq n}$ be a random walk with increment $\xi^{(1)}$. By the many-to-one formula, we have
    \begin{align*}
        &\hspace{0.5cm}\E[\#\{v\in Q_{n,\beta} : \eta^{(1)}_{v,n}(n)\geq m_n-x\}]\\
        &=\rho^n\p\Big(S_k<\frac{km_n}{n}+\beta+\frac{4}{\bl}(\log(\min(k,n-k)))_+\text{ for }1\leq k\leq n;\,S_n\geq m_n-x\Big)\\
        &=\sum_{j=0}^{x+\beta-1}\rho^n\p\Big(S_k<\frac{km_n}{n}+\beta+\frac{4}{\bl}(\log(\min(k,n-k)))_+\text{ for }1\leq k\leq n;\\
        &\hspace{4cm}S_n\in[m_n+\beta-j-1,m_n+\beta-j)\Big).
    \end{align*}
By substituting $i=\beta-j$, the latter probability is precisely $\chi_{n,n}^{\p}(i)$ defined below (15) in \citep{bramson2016convergence}. It is proved in (18) therein the estimate
$$\chi_{n,n}^{\p}(i)\ll \beta(\beta-i+2)\rho^{-n}e^{-\bl i},$$ where our assumption $x\leq\sqrt{n}$ is used. 
This leads to 
$$\E[\#\{v\in Q_{n,\beta} : \eta^{(1)}_{v,n}(n)\geq m_n-x\}]\ll \sum_{j=0}^{x+\beta}\beta(j+2)e^{-\bl(\beta-j)}\ll  \beta (x+\beta)e^{\bl x},$$
completing the proof.    
\end{proof}

\begin{proof}[Proof of \eqref{eq:lb number}]
    We apply the second moment method. Define$$P_n:=\left\{v\in V_n : \text{ for any }0\leq k\leq n,~\eta^{(1)}_{v,n}(k)<\frac{km_n}{n}\right\}.$$
For $x\in[2,\sqrt{n}]$ and $v\in V_n$, define the event
$$H_{v,n}(x):=\{v\in P_n,~\eta^{(1)}_{v,n}(n)\in[m_n-x,m_n-x+1)\}$$
and $\Delta_{n,x}:=\sum_{v\in V_n}\bone_{H_{v,n}(x)}$. It follows that \begin{align}
    \#\{v\in V_n:  \eta^{(1)}_{v,n}(n)\geq m_n-x\}\geq \Delta_{n,x}.\label{eq:Un lb}
\end{align} 

Let us first compute the first moment $\E[\Delta_{n,x}]$. Using \eqref{eq:dpdqn}, we have
\begin{align*}
    \p(H_{v,n}(x))&\gg n^{3/2}\rho^{-n}\E_{\qn}\left[e^{-\lambda_n(\eta^{(1)}_{v,n}(n)-m_n)}\bone_{H_{v,n}(x)}\right]\\
    &\gg n^{3/2}\rho^{-n}\E_{\qn}\left[e^{\bl x}\bone_{H_{v,n}(x)}\right]=n^{3/2}\rho^{-n}e^{\bl x}\qn(H_{v,n}(x)).
\end{align*}
By the ballot theorem, in the form of Lemma 2.1 of \citep{bramson2016convergence}, We have
$$\qn(H_{v,n}(x))\gg n^{-3/2}\max(x-1, 1)\gg n^{-3/2}x.$$
Combining the above yields 
\begin{align}
    \E[\Delta_{n,x}]= \rho^n\p(H_{v,n}(x))\gg xe^{\bl x}.\label{eq:1stmoment}
\end{align}

Next we estimate the second moment $\E[\Delta_{n,x}^2]$. Observe that by (A1) (see (29) of \citep{bramson2016convergence}),
\begin{align}
    \E[\Delta_{n,x}^2]\ll \E[\Delta_{n,x}]+\rho^n\sum_{s=1}^n\rho^s \p(H_{v,n}(x)\cap H_{w,n}(x)\text{ for }v\sim_s w),\label{eq:rhopower}
\end{align}
where $v\sim_s w$ means that  the distance of $v$ and $w$ in the genealogical tree is equal to $2s$. To bound the latter probability we condition on the location of the common ancestor of $v$ and $w$, $\eta^{(1)}_{v,n}(n-s)$. We have $\text{for }v\sim_s w$,
\begin{align*}
&\hspace{0.5cm}\p(H_{v,n}(x)\cap H_{w,n}(x))\\
&\leq \sum_{j=0}^\infty \p\left(\eta^{(1)}_{v,n}(k)\leq \frac{km_n}{n}\text{ for }k\leq n-s,~\eta^{(1)}_{v,n}(n-s)-\frac{(n-s)m_n}{n}\in [-(j+1),-j)\right)\\
    &\hspace{4cm}\times \left(\sup_{y\in [-(j+1),-j)}P(n,s,x,y)\right)^2,
\end{align*}
where \begin{align*}
    P(n,s,x,y)&:=\p\Bigg(\eta^{(1)}_{v,n}(n-s+\ell)\leq \frac{(n-s+\ell)m_n}{n}\text{ for }0\leq\ell\leq s,\\
    &\hspace{1.2cm}\eta^{(1)}_{v,n}(n)\in[m_n-(x+1),m_n-x)\mid \eta^{(1)}_{v,n}(n-s)-\frac{(n-s)m_n}{n}=y\Bigg).
\end{align*}
With the same change of measure and ballot theorem argument as above, we get
\begin{align*}
    &\hspace{0.5cm}\p\left(\eta^{(1)}_{v,n}(k)\leq \frac{km_n}{n}\text{ for }k\leq n-s,~\eta^{(1)}_{v,n}(n-s)-\frac{(n-s)m_n}{n}\in [-(j+1),-j)\right)\\
    &\ll e^{3(n-s)(\log n)/(2n)}\rho^{-(n-s)}e^{\lambda_n j}\\
    &\hspace{2cm}\times (\max(1,n-s))^{-3/2}\max\left((m_{n-s}-(n-s)m_n/n+j),1\right)\\
    &\ll \rho^{-(n-s)}e^{\lambda_n((3/(2\bl))((n-s)(\log n)/n-\log(n-s))+j)}\\
    &\hspace{2cm}\times\max\left(\big(j+\frac{(n-s)\log n}{n}-\log(n-s)\big),1\right)\\
    &\ll \rho^{-(n-s)}je^{\bl j-3\log(n-s)/2+3(n-s)\log n/(2n)},
\end{align*}where in the last step we use $0\leq \bl-\lambda_n\ll (\log n)/n$. 
In a similar manner, (while using (6) of \citep{bramson2016convergence}), we obtain the estimate
\begin{align*}
    P(n,s,x,y)\ll \rho^{-s}jxe^{\bl(x-j)+3((s\log n)/n-\log s)/2}.
\end{align*}
In conclusion, we arrive at
\begin{align}
    \E[\Delta_{n,x}^2]&\ll xe^{\bl x}+\sum_{s=1}^n\sum_{j=0}^\infty je^{\bl j-3\log(n-s)/2+3(n-s)\log n/(2n)}\nonumber\\
    &\hspace{4cm}\times\left(jxe^{\bl(x-j)+3((s\log n)/n-\log s)/2}\right)^2\nonumber\\
    &\ll xe^{\bl x}+x^2e^{2\bl x}\sum_{s=1}^n \frac{n^{3/2}e^{(3s\log n)/(2n)}}{(\max(1,n-s))^{3/2}s^3}\ll x^2e^{2\bl x},\label{eq:second moment last step}
\end{align}
where in the last step we argue similarly as in the proof of Lemma 2.7 of \citep{bramson2016convergence}.
 Combined with \eqref{eq:Un lb}, \eqref{eq:1stmoment}, and the Paley-Zygmund inequality leads to \eqref{eq:lb number}. 
\end{proof}

\subsubsection{Large deviation computations}\label{sec:rw}

In the following, we let $\{(X_j,\bY_j)\}_{j\in\N}$ be a sequence of i.i.d.~random variables with the same distribution as $\bxi$, and $(F_n,\bJ_n)=\sum_{j=1}^n(X_j,\bY_j)$. Recall the definition of $m_n$ from \eqref{eq:one-dim asymp}. We may slightly abuse notation and use $B_\bx(r)$ to denote the ball centered at $\bx\in\R^{d-1}$ with radius $r$ in dimension $d-1$ (instead of dimension $d$).

\begin{proposition}\label{prop:ulb gaussian approx}
   Fix a large constant $C>0$ and a dimension $d\geq 1$. Suppose that the  $\R^d$-valued  random variable $\bxi$ satisfies conditions (A2)--(A4) from Section \ref{sec:fpt_for_brw}. 
   \begin{enumerate}[(i)]
      
   \item \sloppy Uniformly for  any positive sequence $u(n)\leq C\sqrt{n}$ and any $c=c(n)\in [-C\log n,C\log n]$, 
    \begin{align}
        \p(\bJ_n\in B_\z(u(n))\mid F_n\geq m_n+c)\ll u(n)^{d-1}n^{-(d-1)/2}.\label{eq:ub}
    \end{align}
    \item Fix $a(n)\gg 1$ and $a(n)=o(\log n)$.  Uniformly for  $\by$ with $\n{\by}\leq C\sqrt{n}$ and  $c\in [-C\log n,C\log n]$,  \begin{align}
        \p(\bJ_n\in B_\by(1)\mid F_n\in[ m_n+c,m_n+c+a(n)])\gg n^{-(d-1)/2}.\label{eq:lb}
    \end{align}

 \item For any $\ee>0$, there exists $K>0$ such that uniformly for $c\in [-C\log n,C\log n]$,
 \begin{align}
     \p(\bJ_n\not \in B_\z(K\sqrt{n})\mid F_n\geq m_n+c)<\ee.\label{eq:ub2}
 \end{align}
   \end{enumerate}
 
\end{proposition}

The proof of Proposition \ref{prop:ulb gaussian approx} is deferred to the appendix.

\subsubsection{Proof of Theorem \texorpdfstring{\ref{thm:main}}{} conditioned on Theorem \texorpdfstring{\ref{theorem:concentration}}{}}
\label{sec:proof}
Recall \eqref{eq:A}. Let us define 
$$A_1(x):=A(x)+\frac{2}{\bl c_1}\,\log\log x,\quad A_2(x):=A(x)-\frac{1}{\bl c_1}\,\log\log x.$$
In view of Theorem \ref{theorem:concentration}, it suffices to prove that there exists $L>0$ such that for any $a(x)$ that tends to infinity, and any $\ee>0$, there exists $N>0$ such that for any $x\geq N$,     \begin{align}
    \p(\tau_x>A_1(x)+a(x))\leq 1-\frac{1}{L}\label{eq:toprove1}
\end{align}
and 
\begin{align}
    \p(\tau_x<A_2(x)-a(x))<\ee.\label{eq:toprove2}
\end{align}
    In the following, we fix $\ee>0$ and $a(x)$. \\

\emph{Proof of the upper bound.} Define $$\alpha(x):=\frac{d-1}{2\bl}\,\log x-\frac{1}{\bl}\,\log\log x+c_1a(x)$$and 
$$W_n:=\left\{v\in V_{n}: \bet_{v,n}(n)\in\Big( m_n-\alpha(x),m_n\Big)\times B_\z(\sqrt{x})\right\}.$$
We first prove that for $x$ large enough and some large constant $L_1$ to be determined,
\begin{align}
    \p\big(\#W_{(\log x)^2}<L_1x^{({d-1})/{2}}\big)<1-\frac{1}{L},\label{eq:Wnlb}
\end{align}
for some constant $L>0$ depending on $L_1$ only. 
Indeed, let
$$\widetilde{W}_n:=\left\{v\in V_{n}: \eta^{(1)}_{v,n}(n)\in \Big(m_n-\alpha(x),m_n\Big) \right\}.$$
It follows from Proposition \ref{prop:number of particles} that for $x$ large enough, with $n=(\log x)^2$, $\p(\#\widetilde{W}_n<L_1x^{(d-1)/2})<1-1/L$. 
Comparing the definitions of $W_n$ and $\widetilde{W}_n$, it remains to show that at generation $n=(\log x)^2$, the maximum norm of the BRW is smaller than $\sqrt{x}$ with probability $1-o(1)$. But this is a direct consequence of Theorem 1.1 of \citep{zhang2024large}.


The plan is to evolve (on the event of \eqref{eq:Wnlb})  the  $L_1x^{(d-1)/2}$ particles in $W_{(\log x)^2}$ (independently) so that at least one of those families has a descendant landing in $B_x$ at time $A_1(x)+a(x)$ with high probability, where we recall that $B_x$ is the ball of radius one centered at $(x,\z)$. Let us now fix  $w\in W_{(\log x)^2}$. Define $\tau'(x):=A_1(x)+a(x)-(\log x)^2+K_1$ for some large constant $K_1>0$ (to be determined) and 
$$p(w):=\p\left(  \bet_{v,\tau'(x)}(\tau'(x))\in B_x-\bet_{w,(\log x)^2}((\log x)^2)\text{ for some }v\in V_{\tau'(x)}\right),$$
which is the probability that at least one descendant of $w$ lands in $B_x$ at time $A_1(x)+a(x)+K_1$, or equivalently, the probability that a BRW reaching the set $B_x-\bet_{w,(\log x)^2}((\log x)^2)$ at generation $\tau'(x)$. By independence and by picking the  constant $L_1$ in \eqref{eq:Wnlb} large enough, it suffices to prove $p(w)\gg x^{-(d-1)/2}$ for $w\in W_{(\log x)^2}$. Note that $\eta^{(1)}_{w,(\log x)^2}((\log x)^2)\in (m_{(\log x)^2}-\alpha(x),m_{(\log x)^2})$ and that for $K_1$ chosen large enough,
\begin{align*}
    m_{\tau'(x)}&=x + \frac{d-1}{2c_2} \log x + \frac{2}{c_2} \log \log x - c_1 (\log x)^2 + c_1 a(x)+c_1K_1+\frac{3\log c_1}{2c_2}+o(1)\\
    &\geq x + \frac{d-1}{2c_2} \log x + \frac{2}{c_2} \log \log x - c_1 (\log x)^2 + c_1 a(x)+1\\
    &= x-(m_{(\log x)^2}-\alpha(x))+1.
\end{align*}
Therefore, the asymptotic of the maximum at time $\tau'(x)$, $m_{\tau'(x)}$, is at least $x+1-\eta^{(1)}_{w,(\log x)^2}((\log x)^2)$. Using the lower bound on $\Delta_{n,x}$ derived in the proof of \eqref{eq:lb number}, there is some constant $L>0$ such that
$$\p\left(  \eta^{(1)}_{v,\tau'(x)}(\tau'(x))\in \big(x-\frac{1}{2},x+\frac{1}{2}\big)-\eta^{(1)}_{w,(\log x)^2}((\log x)^2)\text{ for some }v\in V_{\tau'(x)}\right)\geq\frac{1}{L}.$$
Conditioned on the above event, Proposition \ref{prop:ulb gaussian approx}(ii) shows that the particle $v$ that realizes this event has a chance of $\gg x^{-(d-1)/2}$ to land in $B_x-\bet_{w,(\log x)^2}((\log x)^2)$ at time $\tau'(x)$, leading to $p(w)\geq x^{-(d-1)/2}/L$. This establishes \eqref{eq:toprove1}. \\

\emph{Proof of the lower bound.} Our goal is to bound $\p(\tau_x\leq A_2(x)-a(x) )$ from above. 
According to Propositions \ref{prop:number of particles} and  \ref{prop:ulb gaussian approx}, at time $A_2(x) $ there will be less than $O(1)$ number of particles that lie in $[x,\infty)\times B_\z(1)$ in expectation. We need to argue that  other particles that stayed in $B_x$ earlier  but currently do not stay in $[x,\infty)\times B_\z(1)$ do not contribute asymptotically. This is because most of the $O(1)$ particles will arrive not much earlier than $t_1(x)+t_2(x)$, and will not move quickly.

 Consider $n=A_2(x)$. Recall that $\ee>0$ and $a(x)\to\infty$ are fixed. 
We define $T_{v,n}(x):=\min\{0\leq k\leq n: \eta^{(1)}_{v,n}(k)\geq x\}$ and for an integer $j\in[a(x),t_2(x)]$, 
$$V_{n,j}:=\big\{v\in V_n:  T_{v,n}(x)=A_2(x)-j\big\}.$$
It follows from Lemma \ref{lemma:beyondcurve} that for any $\ee>0$, there exists  $L_2$ large enough such that
\begin{align}
    &\hspace{0.5cm}\p(\tau_x\leq A_2(x)-a(x))\nonumber\\
    &\leq \p\left(\bigcup_{j=a(x)}^{t_2(x)+L_2}\bigcup_{v\in V_{n,j}}\bigcup_{T_{v,n}(x)\leq k\leq n}\{\bet_{v,n}(k)\in B_x\}\right)+\p\left(M_{A_2(x)-t_2(x)-L_2}\geq x\right)\nonumber\\
    &\leq \sum_{j=a(x)}^{t_2(x)+L_2}\p\left(\bigcup_{v\in V_{n,j}}\bigcup_{T_{v,n}(x)\leq k\leq n}\{\bet_{v,n}(k)\in B_x\}\right)+\frac{\ee}{4}.\label{eq:1}
\end{align}
Let us further decompose the event $\{\bet_{v,n}(k)\in B_x\}$ depending on the location at the first time the random walk $\{\bet_{v,n}(k)\}_{0\leq k\leq n}$ reaches $[x,\infty)\times\R^{d-1}$,  whether $\bet_{v,n}(T_{v,n}(x))\in [x,\infty)\times B_\z(c_3j)$ with some $c_3>0$ to be determined. 

First, with $j$ fixed,
\begin{align*}
    &\hspace{0.5cm}\p\left(\bigcup_{v\in V_{n,j}}\bigcup_{T_{v,n}(x)\leq k\leq n}\{\bet_{v,n}(k)\in B_x,~\bet_{v,n}(T_{v,n}(x))\in [x,\infty)\times B_\z(c_3j)\}\right)\\
    &\leq \p\left(\bigcup_{v\in V_{n,j}}\{\bet_{v,n}(A_2(x)-j)\in [x,\infty)\times B_\z(c_3j)\}\right)\\
    &= \p\left(\bigcup_{v\in V_{A_2(x)-j}}\{\bet_{v,A_2(x)-j}(A_2(x)-j)\in [x,\infty)\times B_\z(c_3j)\}\right).
\end{align*}
Using Proposition \ref{prop:number of particles} (while adjusting the proof suitably), we see that for some $L_3>0$,
\begin{align*}
    &\p\left(\#\{v\in V_{A_2(x)-j}: \eta^{(1)}_{v,A_2(x)-j}(A_2(x)-j)\geq x\}\geq L_3j^2e^{-\bl c_1 j}n^{(d-1)/2}\right) <\frac{\ee}{10j^2}.
\end{align*}
Applying Proposition \ref{prop:ulb gaussian approx}(iii) and the union bound then yields
\begin{align*}
    &\p\left(\bigcup_{v\in V_{A_2(x)-j}}\{\eta^{(1)}_{v,A_2(x)-j}(A_2(x)-j)\in [x,\infty)\times B_\z(c_3j)\}\right)\\
    &\hspace{6cm}\leq \frac{\ee}{10j^2}+L_3j^2e^{-\bl c_1 j}(c_3j)^{d-1}.
\end{align*}
 Second, using that $j\geq a(x)$ and $a(x)\to\infty$ as $x\to\infty$, as well as Theorem 3.2 of \citep{gantert2018large}, we have for some $c_3,c_4>0$ large and independent of $x$, such that for all $x$ large enough,
\begin{align*}&\hspace{0.5cm}\p\left(\bigcup_{v\in V_{n,j}}\bigcup_{T_{v,n}(x)\leq k\leq n}\{\bet_{v,n}(k)\in B_x,~\bet_{v,n}(T_{v,n}(x))\in [x,\infty)\times B_\z(c_3j)^c\}\right)\\
&\leq \p\left(\max_{v\in V_j}\n{\bet_{v,j}(j)}\geq (c_3-1)j\right)\\
&\leq d\,\p\left(\max_{v\in V_j}|\eta^{(1)}_{v,j}(j)|\geq \frac{(c_3-1)j}{\sqrt{d}}\right)\\
&\leq e^{-j/c_4}.
\end{align*}
Inserting the above estimates to \eqref{eq:1} leads to
\begin{align*}
     \p(\tau_x\leq A_2(x)-a(x))&\leq \frac{\ee}{4}+\sum_{j=a(x)}^{\infty}\left(\frac{\ee}{10j^2}+L_3j^2e^{-\bl c_1 j}(c_3j)^{d-1}+e^{-j/c_4}\right).
\end{align*}
For any $\ee>0$ and $a(x)\to\infty$, the right-hand side is $<\ee$ for $x$ large enough. This completes the proof of the lower bound \eqref{eq:toprove2}.

\subsection{Tightness of first passage times: proof of Theorem \texorpdfstring{\ref{theorem:concentration}}{}}\label{sec:concentration}

 We build upon ideas from  \citep{mcdiarmid1995minimal}, while in contrast, we do not assume uniform boundedness of the increments. 
  In the following, we write $C_1,C_2,C_3,\dots$ as large constants and $c_3,c_4,\dots$ as small constants that may depend on the law of $\bxi$ and the branching rate.\footnote{Not to be confused with the fixed constants $c_1,c_2$ defined in Section \ref{sec:notations-sphsym}.} Here, we do not assume that $\bxi$ is spherically symmetric.
We start with a few preparatory lemmas. 

 \begin{lemma}\label{lemma:weak uniform LD}
  Assume (A6).   For any $C_1>0$ and $c_3>0$, there exists  $C_2>0$ such that the following holds for $k$ large enough:      let $(X_j,\bY_j)$ be a sequence of i.i.d.~random variables with the same distribution as $\bxi$, $C>C_2$, and $(F_{Ck},\bJ_{Ck})=\sum_{j=1}^{Ck}(X_j,\bY_j)$. Then uniformly for $k$ large enough and $\bz\in\R^d$ such that $\n{\bz}\leq C_1k$,
     $$\p((F_{Ck},\bJ_{Ck})\in B_{\bz})\gg e^{-c_3k}.$$
 \end{lemma}

 The proof of Lemma \ref{lemma:weak uniform LD} is deferred to the appendix.

\begin{lemma}\label{lemma:expgrowth}
    There exist some $C_3,c_4>0$  such that for $k\geq 1$ large enough,
    $$\p(\#\{v\in V_{C_3k}:\eta^{(1)}_{v,C_3k}(C_3k)\geq k\}<e^{c_4k}\mid S)\ll e^{-c_4k}.$$
\end{lemma}

\begin{proof}
We first claim that for any $c_5>0$, there exist $c_6,c_7>0$ such that uniformly for $n$ large enough, 
\begin{align}
    \p(\#\{v\in V_n:\n{\bet_{v,n}(n)}\leq c_5n\}< e^{c_6n}\mid S)\ll e^{-c_7n}.\label{eq:c5c6}
\end{align}
Indeed, given $c_5$, let $c_8=c_5/(6d\bar{c}_1)$ where here $\bar{c}_1$ stands for the maximum linear speed of the BRW across all $d$ directions, so that by Theorem 3.2 of \citep{gantert2018large},\footnote{Theorem 3.2 of \citep{gantert2018large} is stated in the Schr\"{o}der case, where $p_0+p_1>0$. However, the proof for the upper bound of the large deviation probability does not rely on this assumption; see Remark 3.3(b) therein.}
\begin{align}
    \p\left(\exists v\in V_{c_8n},~\n{\bet_{v,c_8n}(c_8n)}> \frac{c_5n}{2}\right)\ll e^{-c_7n}.\label{eq:ld}
\end{align}
Conditioning upon the survival event $S$, with exponentially small probability we have that $\#V_{c_8n}\leq e^{c_9n}$. Evidently, each $v\in V_{c_8n}$ will lead independently with probability $>1/2$ to a particle $w\in V_n$ where $\n{\bet_{w,n}(n)-\bet_{w,n}(c_8n)}\leq c_5n/2$.
This along with standard binomial estimates proves \eqref{eq:c5c6}.

We return to the proof.  We apply first \eqref{eq:c5c6} with $n=k$ and $c_5=1$. Outside a set of exponentially small probability, we may consider a subcollection of $e^{c_6k}$ particles in $\{v\in V_k:\n{\bet_{v,k}(k)}\leq k\}$, which we label by $v_j,~1\leq j\leq e^{c_6k}$. We next evolve the particles $v_j$ independently and check if they reach $k$ in the $x$-direction. It suffices to consider the presence of the particles evolved from $v_j,~1\leq j\leq e^{c_6k}$ in a ball of radius one and at most $k+k+1\leq 3k$ apart (from the locations of $v_j$). By Lemma \ref{lemma:weak uniform LD} applied with $C_1=3$ and $c_3=c_6/2$, for each $j$ we have for some $C_2>0$, with $w_j$ denoting the location of (a fixed descendant of) $v_j$ after time $C_2k$, that 
\begin{align}
    \p(w_j\in [k,\infty)\times\R^{d-1})\geq \p(w_j\in B_{(k+1,\z)})\gg e^{-c_6k/2}.\label{eq:vjafterk}
\end{align} Note that such a fixed path exists for at least half of $v_j$'s with overwhelming probability.
In conclusion, we evolved at least $e^{c_6k}/2$ many particles independently at time $k$ for time $C_2k$, each resulting in a probability of $\gg e^{-c_6k/2}$ to land in $[k,\infty)\times\R^{d-1}$. A standard binomial estimate yields that with overwhelming probability, there exist $e^{c_6k/4}$ particles at time $(C_2+1)k$ present in $[k,\infty)\times\R^{d-1}$ (conditional on survival). This completes the proof.
\end{proof}

\begin{lemma}\label{lemma:LD for taux}
    There exist constants $c_{10},C_4>0$ such that for $x$ large enough,
    $$\p(\tau_{x}\geq C_4x\mid S)\ll e^{-c_{10}x}.$$
\end{lemma}

\begin{proof}
The proof is almost the same as Lemma \ref{lemma:expgrowth}, by replacing $k$ by $x$ and using the second inequality of \eqref{eq:vjafterk} only. 
\end{proof}

In the proof of Theorem \ref{theorem:concentration},  we wish to use Lemma \ref{lemma:LD for taux} to bound probabilities of the form $\p(\tau_x\geq \alpha,~\tau_y\leq\beta)$, where $\beta<\alpha$. This is intuitively possible since we may run the particle $v_y$ that reaches $B_y$ and bound the probability that it never reaches $B_x$ in a time longer than $\alpha-\beta$, where we have used the strong Markov property of the random walk.  There is, however, a subtle issue for the survival of the branching process initiated by the particle $v_y$, as this event is almost independent of the survival event $S$ we are conditioning on (in the proof of Theorem \ref{theorem:concentration}). We will bypass this difficulty by excluding this termination event a priori, or by first having an exponential number $e^{cn}$ of particles in $B_y$ instead of a single particle $v_y$. 

Recall our notations beginning Section  \ref{app:proof_thm1} and denote by $q=1-p=1-\p(S)$. 

\begin{proof}[Proof of Theorem \ref{theorem:concentration}]  Pick $C_3$ and $c_4$ as in Lemma \ref{lemma:expgrowth} and let $\delta\in(e^{-c_{13}/(2C_3)},1)$ for some $c_{13}>0$ to be determined.  For $x,y>0$, let $$s=s(x,y):=\inf\{t\geq 0:\p(\tau_x\leq t\mid S)\geq\delta^y\}.$$
It follows that $\p(\tau_x<s-1\mid S)\leq\delta^y$ and
\begin{align}
    \p(\tau_x>s+1)\leq 1-p\delta^y\leq \exp(-p\delta^y).\label{eq:ub at s}
\end{align}
 Let $k=y/C_3$ and set $C_5\gg 1+C_3+C_6C_7$, all to be determined. 
Denote by $\tau^{(c_{11})}_{x+k}$ the waiting time until $e^{c_{11}k}$ many particles reach the ball $B_{x+k}$. 
Now note that
\begin{equation}\label{eq:stepone}\begin{split}
    \p(\tau_x>s-1)&\leq \p(s-1<\tau_x<s+C_5k)+\p\left(\tau^{(c_{11})}_{x+k}>s+C_3k+1+C_6C_7k\right)\\
    &\hspace{2.5cm}+\p\left(\tau_x\geq s+C_5k,~\tau^{(c_{11})}_{x+k}\leq s+C_3k+1+C_6C_7k\right).
\end{split}
\end{equation}
We first decompose the second event $\tau^{(c_{11})}_{x+k}>s+C_3k+1+C_6C_7k$ of \eqref{eq:stepone} depending on the number of particles at time $C_3k$ that lie beyond $k$ in the $x$-coordinate. On one hand, by Lemma \ref{lemma:expgrowth},
$$\p(\#\{v\in V_{C_3k}:\eta^{(1)}_{v,C_3k}(C_3k)\geq k\}<e^{c_4k})\leq q+o(e^{-c_4k}).$$
On the other hand, on the event that $\#\{v\in V_{C_3k}:\eta^{(1)}_{v,C_3k}(C_3k)\geq k\}\geq e^{c_4k}$, we may label $e^{c_4k}$ many particles by $\{v_i\}_{1\leq i\leq e^{c_4k}}$. Our goal is to show that the probability that fewer than $e^{c_{11}k}$ of $\{v_i\}_{1\leq i\leq e^{c_4k}}$ reach $B_x$ in the period of time $[C_3k,C_3k+s+1+C_6C_7k]$ is small. Note that the trajectories for each $i$ in the period $[C_3k,C_3k+s+1+C_6C_7k]$ are independent and equivalent in law to a BRW trajectory initiated at $\bet_{v,C_3k}(C_3k)$. A technical difficulty arises, however, as the events of reaching $B_x$ are not monotone in $x$ in dimensions $d\geq 2$. It would be nice if we had stochastic monotonicity for $\tau_x$ in $x$ (as in the one-dimensional case), but this appears out of our reach. Instead, we use the following argument to bypass this technical difficulty. 

Recall that each $\eta^{(1)}_{v_i,C_3k}(C_3k)\geq k$, and the notation $\bx=(x,0\dots,0)\in\R^d$. By a standard large deviation estimate on the last $d-1$ coordinates of $\bet_{v_i,C_3k}(C_3k)$ and using the triangle inequality, we may assume that $\n{\bx-\bet_{v_i,C_3k}(C_3k)}\leq x+C_6k$ for all $1\leq i\leq e^{c_4k}$ by losing an event of probability $O(e^{-c_{12}k})$ with $c_4$ chosen small enough.
Similarly, we may assume $\eta^{(1)}_{v_i,C_3k}(C_3k)\leq C_8k$ with an overwhelming probability $1-O(e^{-c_{12}k})$ (since there are $e^{c_4k}$ many particles in total and we can apply a union bound). As a consequence,
$$\p\left(\forall 1\leq i\leq e^{c_4k},~\max\{0,x-C_8k\}\leq \n{\bx-\bet_{v_i,C_3k}(C_3k)}\leq x+C_6k\right)\geq 1-O(e^{-c_{12}k}).$$
For each $1\leq i\leq e^{c_4k}$, consider the set of $\bx_i$'s such that $\n{\bx_i-\bet_{v_i,C_3k}(C_3k)}=x$. Selecting $C_8>C_6$ and on a set of overwhelming probability of at least $1-O(e^{-c_{12}k})$, we can select $\bx_i$ such that $\n{\bx-\bx_i}\leq C_8k$.
For $\bz\in\R^d$, we denote by $\tau_{\bz}$ the FPT to the ball of radius one centered at $\bz$. By Lemma \ref{lemma:LD for taux} applied with $C_4=C_7/2$ and picking $C_7$ large (and without loss of generality, a $p/2$ proportion of the $e^{c_4k}$ many particles do not terminate, by losing an exponentially small probability),
\begin{align*}
    &\hspace{0.5cm}\p\left(\#\{ 1\leq i\leq \frac{p}{2}e^{c_4k},~\tau_{\bx-\bet_{v_i,C_3k}(C_3k)}\geq s+1+C_6C_7k\}\geq e^{c_{11}k}\right)\\
     &\leq \p\left(\#\{ 1\leq i\leq \frac{p}{2}e^{c_4k},~\tau_{\bx_i-\bet_{v_i,C_3k}(C_3k)}\geq s+1\}\geq e^{c_{11}k}\right)+O(e^{-c_{12}k})\\
     &\hspace{1cm}+\p\left(\exists 1\leq i\leq \frac{p}{2}e^{c_4k},~\tau_{\bx-\bet_{v_i,C_3k}(C_3k)}\geq s+1+C_6C_7k,~\tau_{\bx_i-\bet_{v_i,C_3k}(C_3k)}< s+1\right)\\
    &\leq \p\left(\#\{1\leq i\leq \frac{p}{2}e^{c_4k},~\tau_{\bx_i-\bet_{v_i,C_3k}(C_3k)}\geq s+1\}\geq e^{c_{11}k}\right)\\
    &\hspace{4cm}+\sum_{ 1\leq i\leq e^{c_4k}}\p\left(\tau^{(i)}_{C_6k+1}>C_6C_7k\right)+O(e^{-c_{12}k})\\
    &\leq \p\left(\mathrm{Bin}(\frac{p}{2}e^{c_4k},\p(\tau_x>s+1))\geq e^{c_{11}k}\right)+e^{c_4k}e^{-c_{10}k} +O(e^{-c_{12}k}).\nonumber
\end{align*}
Here and later, we use $\mathrm{Bin}(n,p)$ to denote a generic binomial random variable with $n$ trials and success probability $p$. 
By choosing $c_4$ and then $c_{11}$ small enough and using \eqref{eq:ub at s}, we obtain that the above satisfies 
\begin{align*}
&\p\left(\mathrm{Bin}\left(\frac{p}{2}e^{c_4k},\p(\tau_x>s+1)\right)\geq e^{c_{11}k}\right)+e^{(c_4-c_{10})k} +O(e^{-c_{12}k})\\
&\hspace{7cm}\leq \exp(-p\delta^ye^{c_{13}k})+O(e^{-c_{14}k}).\nonumber
\end{align*}
Inserting back to \eqref{eq:stepone} and using Lemma \ref{lemma:LD for taux} on the third term on the right-hand side of \eqref{eq:stepone} (with our choice of $C_5\gg 1+C_3+C_6C_7$), we have 
\begin{align}
    \p(\tau_x>s-1)\leq q+\exp(-p\delta^ye^{c_{13}k})+O(e^{-c_{14}k})+\p(s-1<\tau_x\leq s+C_5k). \label{3}
\end{align}
Since $\delta\in(e^{-c_{13}/(2C_3)},1)$ and $k=y/C_3$, for $y$ large enough we have $\exp(-p\delta^ye^{c_{13}k})\leq \delta^y$. Thus, we conclude from \eqref{3} that
\begin{align}
    \p(\tau_x>s+C_5k)\leq q+O( e^{-c_{15}y}).\label{eq:tau_x recursion}
\end{align}

For $q>0$, using the fact that  $\p(\#V_n>0\mid S^c)=o(e^{-c_{16}n})$ (this is a consequence of Theorem 13.3 of \citep{athreya2004branching}), we obtain 
\begin{equation}\begin{split}
    \p(\tau_x>s+C_5k)&\geq p\,\p(\tau_x>s+C_5k\mid S)+q\,\p(\tau_x>y,~\#V_y=0\mid S^c)\\
    &\geq p\,\p(\tau_x>s+C_5k\mid S)+q(1-o(e^{-c_{16}y}))-C_9\p(\tau_x\leq y,~S^c).
\end{split}
    \label{eq:steptwo}
\end{equation}
Again by using $\p(\#V_n>0\mid S^c)=o(e^{-c_{16}n})$, we have for some $c_{17}$ to be determined later,
$$\p(\tau_x\leq y,~S^c)\leq o(e^{-c_{17}c_{16}y})+\p(\tau_x\leq c_{17}y).$$
Recall that $M_n$ denotes the maximum in the first direction of the BRW at time step $n$. 
Using Lemma 5.2 of \citep{gantert2018large}, it holds that for $\widetilde{M}_j$ denoting the maxima of $\#V_{j}$ many independent random walks,
$$\p(\tau_x\leq c_{17}y)\leq \sum_{j=1}^{c_{17}y}\p(M_j\geq x)\leq \sum_{j=1}^{c_{17}y}\p(\widetilde{M}_j\geq x).$$
It is straightforward to show using a Chebyshev argument that uniformly for $y\leq x$, with $c_{17}$ small enough, the right-hand side is bounded by $e^{-c_{18}y}$. Inserting into \eqref{eq:steptwo} yields
$$\p(\tau_x>s+C_5k)\geq p\,\p(\tau_x>s+C_5k\mid S)+q-C_{10}e^{-c_{19}y}.$$
 Using \eqref{eq:tau_x recursion}, it then holds that $$\p(\tau_x>s+C_5k\mid S)\ll e^{-c_{19}y}.$$ We conclude that for $y\leq x$ large enough,
$$\p(s-1\leq \tau_x\leq s+C_5k\mid S)\geq 1-\delta^y-O(e^{-c_{19}y})\geq 1-O(e^{-c_{20}y}).$$
Therefore, for some $y_0>0$ large enough,
$\p(s-1\leq \tau_x\leq s+C_5y_0/C_3\mid S)>4/5$, so that $s-1\leq t^{(1/2)}_x\leq s+C_5y_0/C_3$. For $y>2C_5y_0/C_3$, we then have
\begin{align*}
    \p\left(|\tau_x-t^{(1/2)}_x|>y\mid S\right)&\leq 1- \p\left(s-1\leq\tau_x\leq s+\frac{y}{2}+\frac{C_5y_0}{C_3}\mid S\right)\leq O(e^{-c_{20}C_3y/(2C_5)}).
\end{align*}
For the case $y\leq 2C_5y_0/C_3$, we may increase the constant $C$ in \eqref{o} such that for $y$ in this range, $Ce^{-cy}\geq 1$ holds, and hence \eqref{o} trivially holds. 
This establishes \eqref{o}. 
\end{proof} 

We remark that the above proof works for a general $r\in(0,1)$ (where $r$ is the quantile in $t_x^{(r)}$), in addition to the case $r=1/2$. 



\subsection{Proof of Theorem \texorpdfstring{\ref{thm:main long}}{}}\label{sec:proof long}

First, let us recall our notations. Let $\Lambda(\bla)=\log\phi_\bxi(\bla)=\log\E[e^{\bla\cdot\bxi}]$ and $I(\bx)$ be its Fenchel--Legendre transform, which coincides with the large deviation rate function for the jump distribution $\bxi$. The unique positive solution to $I(\widehat{c}_1,\z)=\log\rho$ is given by $\widehat{c}_1$. Our goal is to prove that $\tau_x=\widehat{A}(x)+o_\p(\log x)$, where
$$\widehat{A}(x):=\frac{x}{\widehat{c}_1}+\left(\frac{d+2}{2\widehat{c}_1\partial_{x_1} I(\widehat{c}_1,\z)}\right)\log x.$$

\begin{figure}
    \centering
    \includegraphics[width=0.7\textwidth]{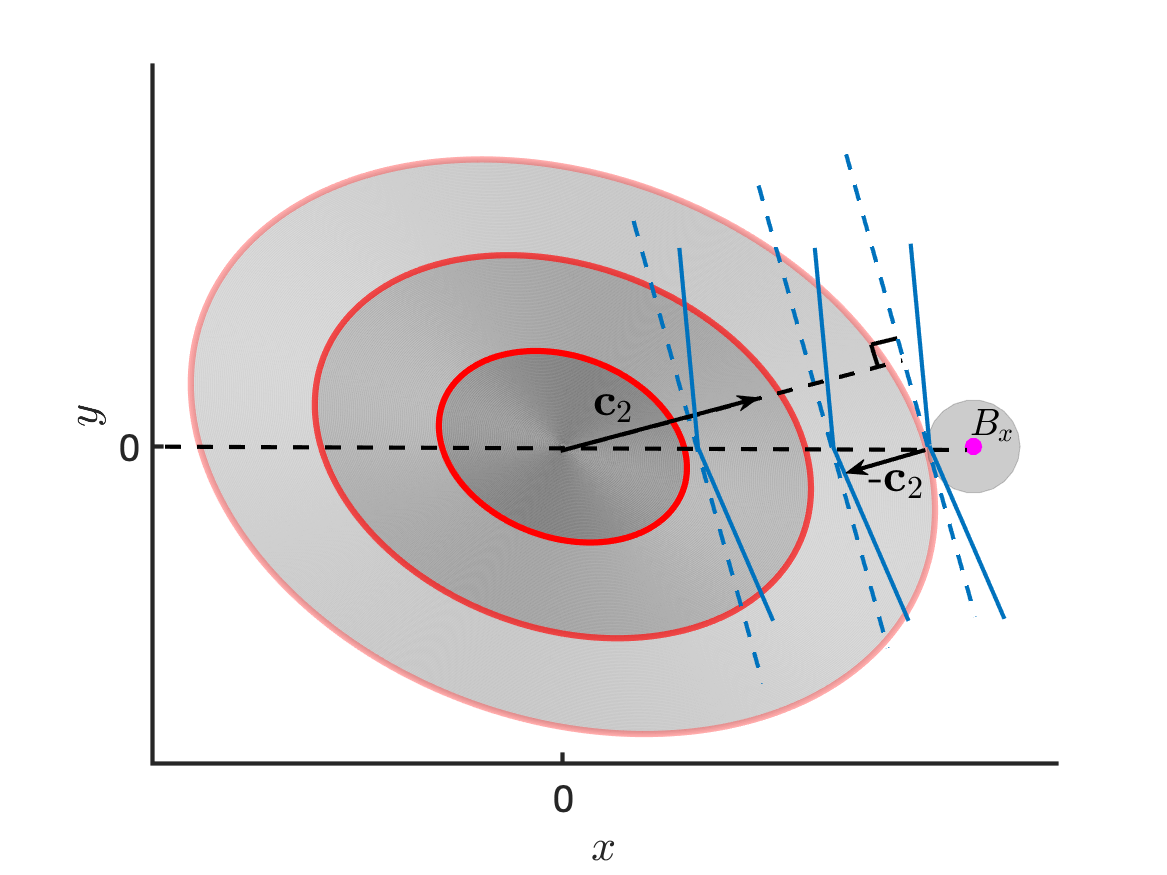}
    \caption{Visualization of how random walks in cones are related to estimating first passage times. The first passage time $\tau_x$ is roughly the time when the growing range (illustrated with the shaded ellipses) of the BRW becomes tangent with the target ball $B_x$. In our proof, we will construct moving cones (that are \emph{circular} centered in the direction $-\bc_2$ in the sense of \eqref{eq:circcone}; the boundaries are indicated by the blue lines) and show that with high probability, the range of the BRW always lies within the moving cones. }
    \label{fig:nonsphere2}
\end{figure}

\subsubsection{Random walks in cones}
For estimating $\tau_x$, we need a few preliminary results on random walks and Brownian motion in cones. The intuitive reason why random walks in cones come into play is explained in Figure \ref{fig:nonsphere2}. In this section, we try to be self-contained and summarize only the necessary results. We refer to \citep{banuelos1997brownian,burkholder1977exit,deblassie1987exit,denisov2015random} for more background on this topic.

We first fix some notations. Let $\{\bU_i\}_{i\geq 1}$ be i.i.d.~copies of $\bU=(U_1,\dots,U_d)$ in $\R^d$ and $\bS_n=\bU_1+\dots+\bU_n$ denote its partial sum. Consider an open connected subset $\mathfrak{S}\subseteq \R^d$ and let $\K$ be the cone generated by the rays emanating from $\z\in\R^d$ that go through $\mathfrak{S}$. For $\bz\in\K$, let $$\tau^\K_\bz=\min\{n\geq 1:\bz+\bS_n\not\in \K\}$$ be the first passage time of the random walk $\{\bS_n\}_{n\geq 1}$ to $\K^{\mathrm{c}}$, starting from $\bz\in\K$. Similarly, let $\tau_\bz^{\K,\mathrm{BM}}$ be the first passage time of the standard $d$-dimensional Brownian motion to $\K^{\mathrm{c}}$, starting from $\bz\in\K$. We will frequently consider cones that are \emph{circular} (in the sense of \citep{deblassie1987exit}), that is, of the form
\begin{align}
    \K_\alpha(\bv):=\{\bz\in\R^d:0\leq\theta<\alpha\},\label{eq:circcone}
\end{align}
where $\alpha\in(0,\pi),\z\neq\bv\in\R^d$, and $\theta=\theta(\bz,\bv)$ is the angle between nonzero vectors $\bz$ and $\bv$ in $\R^d$. In the case $\bv=(1,\z)$, we may write $\K_\alpha=\K_\alpha((1,\z))$.

The following result combines several results of \citep{denisov2015random}. Recall from assumption (A5) that the law of $\bxi$ is non-lattice.

\begin{lemma}\label{lemma:rw in cones}Fix a circular cone $\mathbb{K}$, i.e., $\K=\K_\alpha$ for some $\alpha\in(0,\pi)$. Suppose that $\bU$ is centered, non-lattice, and $\E[U_iU_j]=\bone_{\{i=j\}}$.
Let $\bx,\by\in\K$. There exist positive functions $V,\widetilde{V}$ satisfying
$\max(V(\bx),\widetilde{V}(\bx))\leq C(1+\n{\bx})^p$ and
    \begin{align}
        \p(\bx+\bS_n\in B_\by,\tau^\K_{\bx}>n)\ll \frac{V(\bx)\widetilde{V}(\by)}{n^{p+d/2}},\label{eq:dewa1}
    \end{align}
    where $p=p(\K)$ is such that $\p(\tau_\bx^{\K,\mathrm{BM}}>t)\asymp t^{-p/2}$. Moreover, for fixed $\by\in\K$,
    \begin{align}
        \p(\bx+\bS_n\in B_\by,\tau^\K_{\bx}>n)\asymp \frac{V(\bx)}{n^{p+d/2}},\label{eq:dewa2}
    \end{align}
    where the asymptotic constant in $\asymp$ may depend on $y$. 
\end{lemma}

\begin{proof}
The asymptotic \eqref{eq:dewa2} is a direct consequence of (1), (10) of Theorem 6 (which holds also for a fixed $\by\in\K$), and Lemma 13(b) of \citep{denisov2015random}. On the other hand, \eqref{eq:dewa1} follows from the derivation of Lemmas 27 and 28 therein, along with their Theorem 1.    The results in \citep{denisov2015random} are stated in the lattice case but non-lattice walks can be proved in the same way.
\end{proof}

\begin{remark}\label{remark:transform}
    As commented in \citep{denisov2015random}, if $\bU$ does not satisfy $\E[U_iU_j]=\bone_{\{i=j\}}$ but has a positive definite covariance matrix, then the same results hold with $\K$ replaced by $M\K=\{M\bx:\bx\in\K\}$, where $\bV=M\bU$ satisfies $\E[V_iV_j]=\bone_{\{i=j\}}$.
\end{remark}

The exponent $p(\K)$ is in general not explicit and is related to the smallest eigenvalue of the Laplace-Beltrami operator; see \citep{denisov2015random}. However, it suffices for our purpose to understand the behavior of $p(\K)$ when $\K$ is a {circular cone}  that is close to a half-space. The following result shows that $p\approx 1$ if the circular cone $\K$ is close enough to a half-space.

\begin{lemma}[(3.11) of \citep{burkholder1977exit}]\label{lemma:angle choice}
 Let $\alpha\in(0,\pi)$ and $\bz\in\K_\alpha$ be arbitrary. There exists a decreasing and continuous function $p_{\alpha,d}$ such that $p_{\pi/2,d}=1$ and\footnote{The formal definition is given in \citep{burkholder1977exit} by $p_{\alpha,d}=\sup\{p:\alpha<\theta_{p,d}\}$, where $\theta_{p,d}$ is the smallest zero in $(0,\pi)$ of $h(\theta)=F(-p,p+d-2;(d-1)/2;(1-\cos\theta)/2)$. The continuity is not explicitly stated in \citep{burkholder1977exit}, but follows from the continuity of the hypergeometric function.}
 $$\E[(\tau^{\K_\alpha}_\bz)^{q/2}]<\infty\quad\Longleftrightarrow \quad q<p_{\alpha,d}.$$
\end{lemma}

The following ballot-type lemma provides bounds when the cone is shifted by a function of time. We do not attempt the sharpest bound here, but state the minimal result needed for our purpose.

\begin{lemma}\label{lemma:newballot}Assume the same setting on the law of $\bU$ as in Lemma \ref{lemma:rw in cones}. 
    Let $\bv\in \R^d,~\bv\neq\z$, and $\alpha\in(0,\pi)$. 
    Let $p=p(\K_\alpha(-\bv))$ be the exponent defined in Lemma \ref{lemma:rw in cones}, $p'<p$, and $C>0$. Then uniformly in $\bz,\by\in\R^d$ such that $\bz-\by\in\K_\alpha(-\bv)$ and $n\in\N$,
    \begin{align}
    \begin{split}
        \label{eq:ballot1}
        &\p(\bS_n\in B_\bz;\,\forall 1\leq k\leq n,\,\bS_k-\by-C\log(\min(k,n-k))_+\bv\in\K_\alpha(-\bv))\\
        &\hspace{6cm}\ll \frac{(1+\n{\bz})^p(1+\n{\by})^{p}}{n^{p'+d/2}},
    \end{split}
    \end{align}
    where the asymptotic constant in $\ll$ depends on $d,p'$, and the law of $\bxi$. 
    Moreover, uniformly for $n\in\N$, $1\leq\ell\leq n$, and $\bz,\by\in\R^d$ with $\bz- \by\in C\log(\min(\ell,n-\ell))_+\bv+\K_\alpha(-\bv)$, \begin{align}
\begin{split}
    \label{eq:ballot2}
        &\p(\bS_\ell\in B_\bz;\,\forall 1\leq k\leq \ell,\,\bS_k-\by-C\log(\min(k,n-k))_+\bv\in\K_\alpha(-\bv))\\
        &\hspace{6cm}\ll \frac{(1+\n{\bz})^p(1+\n{\by})^{p}}{n^{p'+d/2}},
\end{split}
    \end{align}
    where the asymptotic constant in $\ll$ depends on $d,p'$, and the law of $\bxi$. 
\end{lemma}
\begin{proof}
Let $\tilde{\by}=-\by-(C\log n)\bv$ and $\tilde{\bz}=-\bz-(C\log n)\bv$. We have by using \eqref{eq:dewa1} of Lemma \ref{lemma:rw in cones},
\begin{align*}
    &\hspace{0.5cm}\p(\bS_n\in B_\bz;\,\forall 1\leq k\leq n,\,\bS_k-\by-C\log(\min(k,n-k))_+\bv\in\K_\alpha(-\bv))\\
        &\leq \p(\bS_n\in B_\bz;\,\forall 1\leq k\leq n,\,\bS_k-\by-(C\log n)\bv\in\K_\alpha(-\bv))\\
        &=\p(\tilde{\by}+\bS_n\in B_{\tilde{\bz}};\,\forall 1\leq k\leq n,\,\tilde{\by}+\bS_k\in\K_\alpha(-\bv))\\
        &\ll \frac{(1+\n{\tilde{\by}})^p(1+\n{\tilde{\bz}})^p}{n^{p+d/2}}\\
        &\ll \frac{(1+\n{{\by}})^p(1+\n{{\bz}})^p}{n^{p'+d/2}}.
\end{align*}
    This completes the proof of \eqref{eq:ballot1}. The proof of \eqref{eq:ballot2} is similar and is omitted here for brevity.
\end{proof}

\subsubsection{A large deviation computation}\label{242}

In the following, we work with a sequence $\{\bxi_j=(X_j,\bY_j)\}_{j\in\N}$ of i.i.d.~centered random variables in $\R^d$ that are not necessarily spherically symmetric but have densities belonging to $L^r(\R^d)$ for some $r>1$ and a positive definite covariance matrix $V$. Let $\Lambda(\bla)=\log\phi_\bxi(\bla)=\log\E[e^{\bla\cdot\bxi}]$ and $I(\bx)$ be its Fenchel--Legendre transform.\footnote{Not to be confused with the functions $\Lambda$ and $I$ for the first coordinate which is defined on $\R$, since we may assume $d\geq 2$.} Let also $\bc_2=\nabla I(x/n,\z)$. By definition,
\begin{align}
    \Lambda(\bc_2)=\bc_2\cdot\left(\frac{x}{n},\z\right)-I\left(\frac{x}{n},\z\right).\label{eq:FL identity long}
\end{align}
We have the following counterpart of Proposition \ref{prop:ulb gaussian approx}. The proof is similar and deferred to the appendix.

\begin{proposition}\label{prop:ulb gaussian approx long}
   Fix a large constant $C>0$ and a dimension $d\geq 1$. Suppose that the  $\R^d$-valued  random variable $\bxi$ satisfies conditions (A2), (A5), and (A6) from Section \ref{sec:fpt_for_brw}. Let $(X_j,\bY_j)$ be a sequence of i.i.d.~random variables with the same distribution as $\bxi$ and $(F_n,\bJ_n)=\sum_{j=1}^{n}(X_j,\bY_j)$. Fix $\overline{c}>\underline{c}>0$.
   \begin{enumerate}[(i)]
      
   \item \sloppy Uniformly for  any  $c=c(n)\in [-C\log n,C\log n]$ and any $x\in[ \underline{c}\,n,\overline{c}\,n]$,
    \begin{align}
       \p(\bJ_n\in B_\z(1)\mid F_n\geq x+c)\ll e^{-nI(x/n,\z)-\bc_2\cdot (c,\z)}n^{-(d-1)/2}e^{nI(x/n)+cI'(x/n)}.\label{eq:ub long}
    \end{align}
    \item Fix $a(n)\to\infty$ and $a(n)=o(\log n)$.  Uniformly for  $\by$ with $\n{\by}\leq C\sqrt{n}$ and  $c\in [-C\log n,C\log n]$, and  any $x\in[ \underline{c}\,n,\overline{c}\,n]$, \begin{equation}
        \begin{split}
        &\p(\bJ_n\in B_\by(1)\mid F_n\in[ x+c,x+c+a(n)])\\
        &\hspace{2cm}\gg e^{-nI(x/n,\z)-\bc_2\cdot (c,\by)}n^{-(d-1)/2}e^{nI(x/n)+cI'(x/n)}.
        \end{split}\label{eq:lb long}  
    \end{equation}
   \end{enumerate}
 
\end{proposition}

\subsubsection{Upper bound of Theorem \texorpdfstring{\ref{thm:main long}}{}}\label{sec:ub of thm main long}
 Let $\ee>0$ be arbitrary and \begin{align}
    n=\widehat{A}_1(x)=\frac{x}{\widehat{c}_1}+\left(\frac{d+2+2\ee}{2\widehat{c}_1\partial_{x_1} I(\widehat{c}_1,\z)}\right)\,\log x.\label{eq:nx long}
\end{align}
Our goal is to show that the event $\{\tau_x>\widehat{A}_1(x)\}$ has a small probability. 
Define $\bc_2:=\nabla I(x/n,\z)$. By definition,
\begin{align}
    \Lambda(\bc_2)=\bc_2\cdot\left(\frac{x}{n},\z\right)-I\left(\frac{x}{n},\z\right).\label{eq:FL identity long2}
\end{align}

 We set up a barrier event. 
 Pick $\alpha\in(0,\pi/2)$ such that $p_{\alpha,d}=1+\ee$; cf., Lemma \ref{lemma:angle choice}. Recall \eqref{eq:circcone}. 
 Let
$$\widehat{P}_n:=\left\{v\in V_n:~\forall 0\leq k\leq n,~ \bet_{v,n}(k)-\frac{k\bx}{n}\in\K_\alpha(-\bc_2) \right\}$$
and let 
\begin{align}
    \widehat{\Delta}_{n,x}:=\sum_{v\in \widehat{P}_n}\bone_{\widehat{H}_{v,n}(x)}:=\sum_{v\in \widehat{P}_n}\bone_{\{\bet_{v,n}(n)\in B_x\}}\label{eq:hat H}
\end{align}
be the number of particles lying in the target $B_x$ at time $n=\widehat{A}_1(x)$ that is controlled by the barrier given by $\widehat{P}_n$. The quantity $\widehat{\Delta}_{n,x}$ serves as a lower bound for the number of particles lying in the target $B_x$ at time $n$. As a consequence, it suffices to show that 
\begin{align}
    \p(\widehat{\Delta}_{n,x}\geq 1)>1/C\label{eq:wts long}
\end{align} for some $C>0$ independent of $x$, which in turn gives the upper bound for $\tau_x$ by Theorem \ref{theorem:concentration}. 
We compute first and second moments for $\widehat{\Delta}_{n,x}$.
Applying the many-to-one formula, the change of measure \eqref{eq:dqdp long}, and using a similar change of measure argument as in the proof of Proposition \ref{prop:ulb gaussian approx long}, we have
\begin{align*}
    \E[\widehat{\Delta}_{n,x}]&=\rho^n\p\left(\forall 0\leq k\leq n,~ \bS_k-\frac{k\bx}{n}\in\K_\alpha(-\bc_2);\,\bS_n\in B_x\right)\\
    &\gg \rho^n e^{-nI(x/n,\z)}\q\left(\forall 0\leq k\leq n, \bS_k\in\K_\alpha(-\bc_2);\,\n{\bS_n}\leq 1\right).
\end{align*}
By \eqref{eq:dewa2} of Lemma \ref{lemma:rw in cones}, 
$$\q\left(\forall 0\leq k\leq n, \bS_k\in\K_\alpha(-\bc_2);\,\n{\bS_n}\leq 1\right)\asymp n^{-p_{\alpha,d}-d/2}\asymp x^{-p_{\alpha,d}-d/2}.$$
Here, we may without loss of generality replace $\bc_2$ by $M\bc_2$ where $\E_\q[(M\bxi)_i(M\bxi)_j]=\bone_{\{i=j\}}$. This does not affect the analysis since an invertible linear transformation of a half-space containing $\z$ is a half-space containing $\z$ and we are in a regime where $p_{\alpha,d}\approx 1$; see Remark \ref{remark:transform}. 
Using \eqref{eq:nx long}, we have the approximation
$${nI\Big(\frac{x}{n},\z\Big)}= \frac{\log\rho}{\widehat{c}_1}\,x+\left(\frac{(d+2+2\ee)\log\rho}{2\widehat{c}_1\partial_{x_1} I(\widehat{c}_1,\z)}-(d+2+2\ee)\right)\log x+O(1).$$
These altogether lead to 
\begin{align}
    \E[\widehat{\Delta}_{n,x}]&\gg \rho^n\rho^{-x/\widehat{c}_1}\rho^{-{(d+2+2\ee)\log x}/({2\widehat{c}_1\partial_{x_1} I(\widehat{c}_1,\z)})}x^{(d+2+2\ee)/2}x^{-p_{\alpha,d}-d/2}= 1,\label{eq:first moment long}
\end{align}
where we used $p_{\alpha,d}=1+\ee$.

Next, we compute the second moment of $\widehat{\Delta}_{n,x}$, following the same type of arguments that lead to \eqref{eq:second moment last step}. Recall the definition of $\widehat{H}_{v,n}(x)$ from \eqref{eq:hat H}. By conditioning on the closest integer point to $\bet_{v,n}(n-s)\in (n-s)\bx/n+\K_\alpha(-\bc_2)$ and applying \eqref{eq:dewa1} of Lemma \ref{lemma:rw in cones},$\text{ for }v\sim_s w$ (recall that this means that  the distance of $v$ and $w$ in the genealogical tree is equal to $2s$),
\begin{align*}
    &\hspace{0.5cm}\p\left(\widehat{H}_{v,n}(x)\cap \widehat{H}_{w,n}(x)\right)\\
    &\ll \sum_{\by\in \K_\alpha(-\bc_2)\cap\mathbb{Z}^d}\left(e^{-(n-s)I({x}/{n},\z)-\bc_2\cdot\by}(\max(1,n-s))^{-p_{\alpha,d}-d/2}(1+\n{\by})^{2p_{\alpha,d}}\right)\\
    &\hspace{2cm}\times \left(e^{-sI({x}/{n},\z)-\bc_2\cdot(-\by)}s^{-p_{\alpha,d}-d/2}(1+\n{\by})^{2p_{\alpha,d}}\right)^2\\
    &\ll e^{-(n+s)I({x}/{n},\z)}(\max(1,n-s))^{-p_{\alpha,d}-d/2}s^{-2p_{\alpha,d}-d}\sum_{\by\in \K_\alpha(-\bc_2)\cap\mathbb{Z}^d}e^{\bc_2\cdot \by}(1+\n{\by})^{6p_{\alpha,d}}.
\end{align*}
By the definition \eqref{eq:circcone},
\begin{align*}
    \sum_{\by\in \K_\alpha(-\bc_2)\cap\mathbb{Z}^d}e^{\bc_2\cdot \by}(1+\n{\by})^{6p_{\alpha,d}}&= \sum_{j=0}^\infty\sum_{\substack{\by\in \K_\alpha(-\bc_2)\cap\mathbb{Z}^d\\ \by\cdot(-\bc_2)\in[j,j+1]}}e^{\bc_2\cdot \by}(1+\n{\by})^{6p_{\alpha,d}}\\
    &\ll \sum_{j=0}^\infty e^{-j}(1+j)^{12dp_{\alpha,d}}\ll 1.
\end{align*}
As a consequence and using $p_{\alpha,d}=1+\ee$,$\text{ for }v\sim_s w$,
\begin{align*}
    \p\left(\widehat{H}_{v,n}(x)\cap \widehat{H}_{w,n}(x)\right)&\ll e^{-(n+s)I({x}/{n},\z)}(\max(1,n-s))^{-p_{\alpha,d}-d/2}s^{-2p_{\alpha,d}-d}\\
    &\ll \rho^{-(n+s)}x^{(p_{\alpha,d}+d/2)(n+s)/n}(\max(1,n-s))^{-p_{\alpha,d}-d/2}s^{-2p_{\alpha,d}-d}.
\end{align*}
Therefore, by assumption (A1),
\begin{align*}
    \E[\widehat{\Delta}_{n,x}^2]&\ll \E[\widehat{\Delta}_{n,x}]+\sum_{s=1}^n\rho^{n+s}\p\left(\widehat{H}_{v,n}(x)\cap \widehat{H}_{w,n}(x)\text{ for }v\sim_s w\right)\\
    &\ll 1+\sum_{s=1}^n x^{(p_{\alpha,d}+d/2)(n+s)/n}(\max(1,n-s))^{-p_{\alpha,d}-d/2}s^{-2p_{\alpha,d}-d}\ll 1,
\end{align*}
where in the last step we argue similarly as in the proof of Lemma 2.7 of \citep{bramson2016convergence}. Combined with \eqref{eq:first moment long} and the Paley-Zygmund inequality leads to \eqref{eq:wts long}. This proves the desired upper bound.

\subsubsection{Lower bound of Theorem \texorpdfstring{\ref{thm:main long}}{}}\label{sec:lb of thm main long}
Fix $\ee>0$. 
Define
$$\widehat{A}_2(x):=\frac{x}{\widehat{c}_1}+\frac{(d+2-2\ee)\log x}{2\,\widehat{c}_1\partial_{x_1} I(\widehat{c}_1,\z)}.$$
Our goal is to show that the event $\{\tau_x\leq \widehat{A}_2(x)\}$ has a small probability. We first consider the number $M_x$ of particles in $B_x$ at time $n=\widehat{A}_2(x)$. 
We find $\alpha\in(0,\pi)$ such that $p_{\alpha,d}=1-\ee/2$. 
Recalling $\bc_2=\nabla I(x/n,\z)\in \R_+\times\R^{d-1}$, we may assume that $\bx/n\not\in\K_\alpha(-\bc_2)$. 
We may decompose $M_x=\sum_{\ell=0}^{n}M_{x,\ell}$, where
$$M_{x,n}:=\#\left\{v\in V_{n}:\bet_{v,n}(n)\in B_x;\,\forall 0\leq k\leq n,\,\bet_{v,n}(k)\in \frac{k\bx}{n}+\bpsi_n(k)+\K_\alpha(-\bc_2)\right\},$$
where $\bpsi_n(k):=(C_1+C_2(\log(\min(k,n-k)))_+)\,\bc_2$ for some $C_1,C_2>0$ to be determined, and for $0\leq \ell\leq n-1$,
\begin{align*}
    M_{x,\ell}:=\#\Bigg\{v\in V_{n}:\bet_{v,n}(n)\in B_x;\,&\forall 0\leq k\leq \ell,\,\bet_{v,n}(k)\in \frac{k\bx}{n}+\bpsi_n(k)+\K_\alpha(-\bc_2),\\
    &\bet_{v,n}(\ell+1)\not\in \frac{(\ell+1)\bx}{n}+\bpsi_n(\ell+1)+\K_\alpha(-\bc_2)\Bigg\}.
\end{align*}
Denote by $\{\bS_n\}_{n\geq 1}$ the random walk with i.i.d.~increments $\{\bxi_i\}_{i\geq 1}$. 
Applying the change of measure \eqref{eq:dqdp long} and \eqref{eq:ballot1} of Lemma \ref{lemma:newballot}, we have
\begin{align*}
    \E[M_{x,n}]&=\rho^n\p\left(\bS_n\in B_x;\,\forall 0\leq k\leq n,\,\bS_k\in \frac{k\bx}{n}+\bpsi_n(k)+\K_\alpha(-\bc_2)\right)\\
    &=\rho^n e^{-nI(x/n,\z)}\q\left(\bS_n\in B_\z;\,\forall 0\leq k\leq n,\,\bS_k\in \bpsi_n(k)+\K_\alpha(-\bc_2)\right)\\
    &\ll \rho^n e^{-nI(x/n,\z)}n^{-p_{\alpha,d}+\ee/4-d/2}(1+C_1)^{p_{\alpha,d}}\\
    &\ll x^{-\ee/4}.
\end{align*}

Next, we bound $\E[M_{x,\ell}]$ for $0\leq\ell\leq n-1$. We condition on the location of $\bet_{v,n}(\ell)$, which may take value in ${\ell\bx}/{n}+\bpsi_n(\ell)+\K_\alpha(-\bc_2)$. This leads to
\begin{align}
\begin{split}
    \label{eq:Mxl}
    &\hspace{0.5cm}\E[M_{x,\ell}]\\
    &\leq\rho^\ell\sum_{\substack{\bu\in\mathbb{Z}^d\\ \bu-({\ell\bx}/{n}+\bpsi_n(\ell))\in\K_\alpha(-\bc_2)}}\p\left(\forall 0\leq k\leq \ell,\,\bS_k\in \frac{k\bx}{n}+\bpsi_n(k)+\K_\alpha(-\bc_2),\,\bS_\ell\in B_\bu\right)\\
    &\hspace{4.5cm}\times\p\left(\bxi\in -\bu+\frac{(\ell+1)\bx}{n}+\bpsi_n(\ell+1)+(\K_\alpha(-\bc_2))^{\mathrm{c}}\right).
\end{split}
\end{align}
We bound separately the above two probabilities for a fixed $\ell$. Applying the change of measure \eqref{eq:dqdp long} and \eqref{eq:ballot2} of Lemma \ref{lemma:newballot},
\begin{align*}
    &\hspace{0.5cm}\p\left(\forall 0\leq k\leq \ell,\,\bS_k\in \frac{k\bx}{n}+\bpsi_n(k)+\K_\alpha(-\bc_2),\,\bS_\ell\in B_\bu\right)\\
    &=e^{-\ell I(\bu/\ell)}\q\left(\forall 0\leq k\leq \ell,\,\bS_k\in \bpsi_n(k)+\K_\alpha(-\bc_2),\,\bS_\ell\in B_{\bu-\ell\bx/n}\right).\\
    &\ll e^{-\ell I(\bu/\ell)}n^{-p_{\alpha,d}+\ee/4-d/2}(1+C_1)^{p_{\alpha,d}}\\
    &\ll \rho^{-\ell}x^{-\ee/4} \n{\bu-\frac{\ell\bx}{n}}e^{- \bc_2\cdot(\bu-\ell\bx/n)}.
\end{align*}
On the other hand, since the moment generating function of $\bxi$ exists in a neighborhood of $\bc_2$, we have for some $\delta>0$,
$$\p(\bxi\in B_x)\ll e^{-(1+\delta)\bc_2\cdot\bx}.$$
It follows that for $\bv\in\K_\alpha(-\bc_2)$,
\begin{align*}
    &\hspace{0.5cm}\p\left(\bxi\in -\bv+\frac{\bx}{n}+\bpsi_n(\ell+1)-\bpsi_n(\ell)+(\K_\alpha(-\bc_2))^{\mathrm{c}}\right)\\
    &\ll \exp\left(-(1+\delta)\inf_{\by\in (\K_\alpha(-\bc_2))^{\mathrm{c}}}\left(\bc_2\cdot \left(\by-\bv+\frac{\bx}{n}+\bpsi_n(\ell+1)-\bpsi_n(\ell)\right)\right)\right)\\
    &\leq \exp\left(-(1+\delta)\inf_{\by\in (\K_\alpha(-\bc_2))^{\mathrm{c}}}\left(\bc_2\cdot \left(\by-\bv\right)\right)\right),
\end{align*}
where we have used that $(x/n,\z)\in (\K_\alpha(-\bc_2))^c$.

Inserting the above estimates back to \eqref{eq:Mxl} with the change of variable $\bv=\bu-({\ell\bx}/{n}+\bpsi_n(\ell))$ leads to
\begin{align*}\E[M_{x,\ell}]&\ll \sum_{\substack{\bv\in\mathbb{Z}^d\cap\K_\alpha(-\bc_2)}}x^{-\ee/4} \n{\bv+\bpsi_n(\ell)}e^{- \bc_2\cdot(\bv+\bpsi_n(\ell))}\\
    &\hspace{3cm}\times\exp\left(-(1+\delta)\inf_{\by\in (\K_\alpha(-\bc_2))^{\mathrm{c}}}\left(\bc_2\cdot \left(\by-\bv\right)\right)\right)\\
    &\ll x^{-\ee/4}\sum_{j=0}^\infty e^{-(1+\delta)j}\sum_{\substack{\bv\in\mathbb{Z}^d\cap\K_\alpha(-\bc_2)\\ \inf_{\by\in (\K_\alpha(-\bc_2))^{\mathrm{c}}}\left(\bc_2\cdot \left(\by-\bv\right)\right)\in[j,j+1]}} \n{\bv+\bpsi_n(\ell)}e^{- \bc_2\cdot(\bv+\bpsi_n(\ell))}\\
    &\ll x^{-\ee/10}\max(1,\min(\ell,n-\ell))^{-C_2\n{\bc_2}^2}\\
    &\hspace{3cm}\times\sum_{j=0}^\infty e^{-(1+\delta)j}\sum_{\substack{\bv\in\mathbb{Z}^d\cap\K_\alpha(-\bc_2)\\ \inf_{\by\in (\K_\alpha(-\bc_2))^{\mathrm{c}}}\left(\bc_2\cdot \left(\by-\bv\right)\right)\in[j,j+1]}} \n{\bv}e^{- \bc_2\cdot\bv},
\end{align*}where the last step follows from the definition of $\bpsi_n(\ell)$ and $\bpsi_n(\ell)\ll \log n\asymp \log x$. 
By elementary geometry and properties of $\K_\alpha(-\bc_2)$, for $j\geq 0$,
$$\sum_{\substack{\bv\in\mathbb{Z}^d\cap\K_\alpha(-\bc_2)\\ \inf_{\by\in (\K_\alpha(-\bc_2))^{\mathrm{c}}}\left(\bc_2\cdot \left(\by-\bv\right)\right)\in[j,j+1]}} \n{\bv}e^{- \bc_2\cdot\bv}\ll e^{-j}.$$
Together we obtain
\begin{align*}
     \E[M_{x,\ell}]&\ll x^{-\ee/10}\max(1,\min(\ell,n-\ell))^{-C_2\n{\bc_2}^2}\sum_{j=0}^\infty e^{-\delta j}\\
     &\ll x^{-\ee/10}\max(1,\min(\ell,n-\ell))^{-C_2\n{\bc_2}^2}.
\end{align*}
We thus conclude with the following bound on $\E[M_x]$:
$$\E[M_x]= \sum_{\ell=0}^n M_{x,\ell}\ll x^{-\ee/4}+x^{-\ee/10}\sum_{\ell=0}^{n-1}\max(1,\min(\ell,n-\ell))^{-C_2\n{\bc_2}^2}.$$
With $C_2$ picked large enough, we have $\E[M_x]\ll x^{-\ee/10}$. 

The rest of the arguments is standard. Recall the weaker bound of \eqref{eq:weak taux}. If a particle reaches $B_x$ at time $n-t$ where $t=O(\log x)$, the probability that one of its descendants is located in $B_x$ at time $n$ is $\gg (\log x)^{-d}$. On the other hand, $\E[M_x]\ll x^{-\ee/10}=o((\log x)^{-d})$. As a consequence, with probability $o(1)$, there is no particle present in $B_x$ before time $n=\widehat{A}_2(x)$. This finishes the proof of the lower bound of Theorem \ref{thm:main long}.

\subsection{Proofs of remaining corollaries}\label{sec:coroproof}
\subsubsection{Proof of Corollary \texorpdfstring{\ref{coro:slln}}{}}
Denote by $\tau_x(r)$ the first passage time of the BRW to a sphere of radius $r>0$ centered at $\bx=(x,\z)$. Note that for a fixed $r>0$, the proofs of Theorems \ref{theorem:concentration} and \ref{thm:main long} go through. 

Using elementary geometry, we have
\begin{align}
    \inf_{x\in[n,n+1]}\tau_x(1)\geq \max(\tau_n(2),\tau_{n+1}(2))\label{eq:geo1}
\end{align}
and
\begin{align}
    \sup_{x\in[n,n+1/2]}\tau_x(1)\leq \min\Big(\tau_{n}\big(\frac{1}{2}\big),\tau_{n+1/2}\big(\frac{1}{2}\big)\Big).\label{eq:geo2}
\end{align}
Applying Theorems \ref{theorem:concentration} and \ref{thm:main long}, we have for any fixed $\ee>0$,
$$\sum_{n\in\N}\p\bigg(\Big|\frac{\tau_n(2)}{n}-\frac{1}{\widehat{c}_1}\Big|>\ee\bigg)+\sum_{n\in\N}\p\bigg(\Big|\frac{\tau_{n/2}(1/2)}{n/2}-\frac{1}{\widehat{c}_1}\Big|>\ee\bigg)<\infty.$$
By the Borel--Cantelli lemma, 
$$\frac{\tau_n(2)}{n}\to\frac{1}{\widehat{c}_1}\quad\text{ and }\quad \frac{\tau_{n/2}(1/2)}{n/2}\to\frac{1}{\widehat{c}_1}\quad\text{a.s.}$$
Applying \eqref{eq:geo1} and \eqref{eq:geo2} then yields that almost surely,
$$\limsup_{x\to\infty}\frac{\tau_x(1)}{x}\leq \frac{1}{\widehat{c}_1}\leq \liminf_{x\to\infty}\frac{\tau_x(1)}{x}.$$
This proves 
Corollary \ref{coro:slln}.

\subsubsection{Proof of Corollary \texorpdfstring{\ref{coro:density}}{}}

Denote by $E_n$ the event that the range 
$$R_n=\{\bet_{v,n}(n):v\in V_n\}$$
is not $(1/n)$-dense in $(n/C_2)B_\z$ (meaning that there exists $\by\in (n/C_2)B_\z$ with $(\by+(1/n)B_\z)\cap R_n=\emptyset$). 
To bound $\p(E_n)$ from above, we need a slight modification of Theorem \ref{thm:main long} that applies for balls of radii $1/n$ instead of a fixed radius. 
By performing the same analysis in Section \ref{sec:ub of thm main long}, we have the upper bound
$$\tau_{x,1/n}\leq \frac{x}{\widehat{c}_1}+C_1\log\log x+O_\p(1),$$
where $\tau_{x,1/n}$ denotes the first passage time to a ball of radius $1/n$ centered at $(x,\z)$ and $C_1>0$. 
We need to show that $\widehat{c}_1$ is large in all directions. For a $d\times d$ orthogonal matrix $\Theta$, denote by $\widehat{c}_1(\Theta)$ the $\widehat{c}_1$ corresponding to the BRW with jump distribution $\Theta \bxi$. By assumption (A5) and our assumption that $\phi_\bxi$ is well-defined in a neighborhood of $\z$, the domain of $I$ contains a neighborhood of $\z$, and hence 
\begin{align}
    \inf_{\Theta}\widehat{c}_1(\Theta)>0.\label{eq:theta condition}
\end{align}
Applying \eqref{eq:theta condition} gives that there exists $C_2>0$, so that $\tau_{x,1/n,\bth}\leq (C_2/2)x+O_\p(1)$ in all directions $\bth\in \S^{d-1}$. 
Therefore, by applying Theorem \ref{theorem:concentration},
\begin{align*}
    \p(E_n)\leq \sum_{j=1}^{n/C_2}Cj^{d-1}\sup_{\bth\in \S^{d-1}}\p(\tau_{j,1/n,\bth}\geq n)\leq \sum_{j=1}^{n/C_2}Cj^{d-1}e^{-n/C_3}\ll e^{-n/C_4}
\end{align*}for some $C_3,C_4>0$. 
Thus $\sum_n\p(E_n)<\infty$. 
The remaining arguments follow in the same way as Corollary 2.5 of \citep{oz2023}.

\begin{remark}
    We expect that using a similar argument, one can prove \emph{uniform} versions of \eqref{eq:tauxasymp} and \eqref{eq:tauxasymp long} under stronger assumptions on $I$. For instance,
    $$\sup_{\Theta\in\S^{d-1}}\frac{1}{\log x}\Big|\tau_{\Theta \bx}-\Big(\frac{x}{\widehat{c}_1(\Theta)}+\frac{d+2}{2\widehat{c}_1(\Theta)\partial_{x_1} I_{\Theta}(\widehat{c}_1(\Theta),\z)}\,\log x\Big)\Big|\to 0\quad\text{in probability},$$
    where $I_{\Theta}$ is the large deviation rate function of $\Theta\bxi$.
\end{remark}

\section{Extensions and applications}

\subsection{The delayed branching model}
\label{sec:newbrwproof}
In this section, we revisit the applications of BRW in modeling polymer networks. 
A polymer network is a three-dimensional disordered structure consisting of long sequences of monomers, called polymer chains, connected via chemical bonds. It turns out that in the study of polymers, random walks provide a decent statistical approximation to polymer chains \citep{flory1960elasticity}. In turn, any polymer chain can be bonded to another chain at a certain site (e.g. a monomer) also by chemical bonds called cross-links. 

The resulting network topology is generally random. A chain segment is a connected segment of monomers inside a polymer chain with the property that each of its two ends is connected, via cross-links, to another polymer chain (not necessarily the same one), as mentioned earlier. Chain segments have varying (random) lengths. 

As it turns out, macroscopic mechanical properties, such as elasticity and fracture toughness, are largely determined by the distribution of shortest paths (SP) between distant monomers \citep{yin2020topological}. Understanding the statistics of these SP distributions is, therefore, central to predicting network behavior under deformation.

The FPT for the BRW to reach a unit ball at a distance of $x$ can approximate the SP between monomers placed at a distance $x$ from each other in a polymer network \citep{zhang2024modeling}. The intuition can be appreciated from the picture in Figure \ref{fig:disc_schematic}. Locally, the polymer network at a given cross-link is approximated by the branching mechanism of the BRW.

In physical experiments, obtaining SP statistics is infeasible. However, extracting SP statistics from computational simulations of polymer networks using coarse-grained molecular dynamics (CGMD) has been shown promising~\citep{mohanty2025strength,yin2020topological,yin2024network}.  Now, even though the BRW intuition is local, as argued earlier, and the SP distribution is a global object (involving long distances relative to the size of a single monomer), it was shown based on comparisons to CGMD simulations in \citep{zhang2024modeling}, that the SP statistics can be reproduced by FPT statistics of a small variation of the classical BRW model (the reason of which will be explained below), which we refer to as the \emph{delayed branching} model. The reason for this small variation is to model explicitly the cross-link representing the segment that joins any two polymer chains. 

Consider the \rrm~model satisfying $p_3>0$ and $p_0+p_1+p_3=1$ (i.e., only branching instances are into three particles) with a spherically symmetric jump distribution (for simplicity). The choice of branching into three particles is to be consistent with the real polymer network being modeled, where branching represents the cross-link (bond) between two polymer chains. The first particle would correspond to the continuation of the first polymer chain beyond the point of cross-link whereas, the other two particles would correspond to the segments of the second polymer chain (see \citep{zhang2024modeling} for a detailed description).

The classical BRW assumes that the three particles emerge from the same point. However, the more realistic picture states that the first and second chains are displaced at the point of branching due to the presence of the cross-link. The \emph{delayed branching} model is a modification of such a BRW process, where at each branching event (in which the particle was supposed to give rise to three offsprings), a particle branches into two, and one of the two descendants branches again into two particles in the next time step.\footnote{A particle may die in each of the two branching sub-events.} That is, the branching event (into three particles) is broken into two pieces, one of which is \emph{delayed} by one step. Thus, each branching event consists of two sub-events, which we call types I and II respectively. Here, type I branching corresponds to the branching of the parent, and type II branching corresponds to the branching of one of the two children born in the type I sub-event.
\begin{figure}[t!]
    \centering
    \includegraphics[
    width=0.7\textwidth,
    trim={0.13cm 0.55cm 0.23cm 0.28cm}, 
    clip
  ]{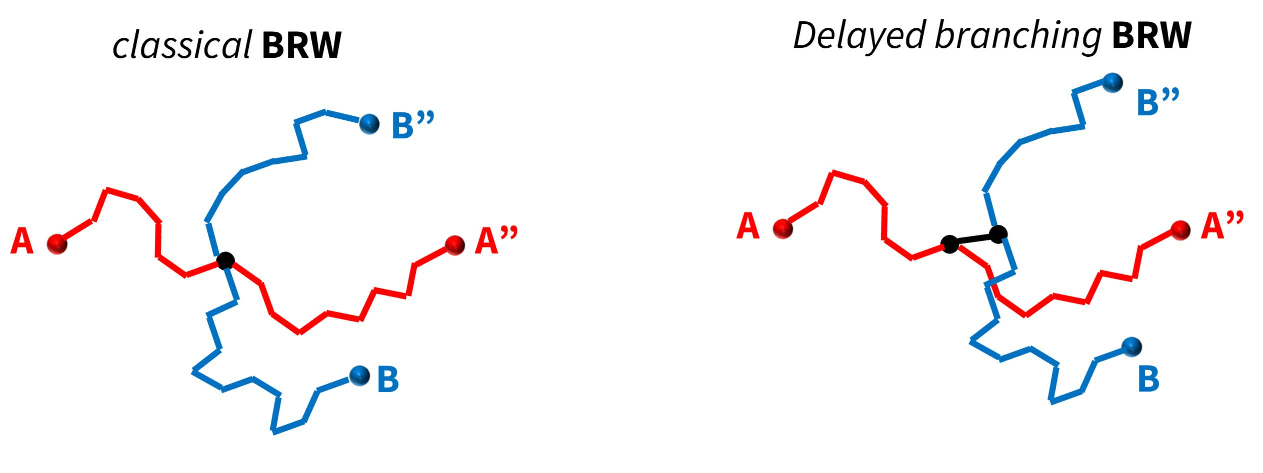}
    \caption{BRW tree representations of polymer networks in the \emph{classical} and the \emph{delayed branching} regimes. The red path from $A$ to $A''$ represents the first polymer chain and the blue path from $B$ to $B''$ represents the second chain that cross-links with the first. In the BRW tree genealogy, the particles $A''$, $B$, and $B''$ are the descendants of the particle starting from $A$.}
    \label{fig:disc_schematic}
\end{figure}

We first derive the analogue for $\rho$ (the parameter governing the growth of the number of particles) in the \rrm. By construction, the expected number $\tilde{N}_n$ of particles in the delayed branching model at time $n$ satisfies
\begin{align}
    \tilde{N}_{n+1}=p_1\tilde{N}_n+p_3 \tilde{N}_n+2p_3\tilde{N}_{n-1}(1-p_0).\label{eq:recur}
\end{align}
The characteristic polynomial of the recursive equation \eqref{eq:recur} is given by
\begin{align}
    x^2-(p_1+p_3)x-2p_3(1-p_0).\label{eq:char poly}
\end{align}
Denote the unique positive zero of \eqref{eq:char poly} by
\begin{align}
    \tilde{\rho}=\frac{1-p_0 }{2}+\sqrt{\frac{(1-p_0 )^2}{4}+2p_3(1-p_0 )}.\label{eq:newrho}
\end{align}
In particular, $\tilde{N}_n\asymp \tilde{\rho}^n$.

Let $\widetilde{c}_1$ satisfy $I(\widetilde{c}_1)=\log\tilde{\rho}$ where the rate function $I$ is  defined in \eqref{eq:ratef}, and $\widetilde{c}_2=I'(\widetilde{c}_1)$.   We define $\widetilde{A}(x)$ similarly as in \eqref{eq:A} with $c_1,\,c_2$ replaced by $\widetilde{c}_1,\,\widetilde{c}_2$.
The parameters $p_3$ and $p_0$ are denoted by $\tl$ and $\tnu$ in \citep{zhang2024modeling} respectively, where $\tl$ represents the \emph{branching rate} of the BRW and $\tnu$ represents the \emph{death rate} of the branches to account for the finite length of the polymer chains in the CGMD simulation. Let $\widetilde{\tau}_x$ be the first passage time to $B_x$.
\begin{theorem}\label{thm:main2}
Assume (A3) and (A4) with $\rho$ replaced by $\tilde{\rho}$. Conditioned upon survival, we have the asymptotic
\begin{align}
    \widetilde{\tau}_x=\widetilde{A}(x)+O_\p(\log\log x)=\frac{x}{\widetilde{c}_1}+\frac{d+2}{2\widetilde{c}_2\widetilde{c}_1}\,\log x+O_\p(\log\log x).\label{eq:tauxasymp2}
\end{align}
In other words, the collection $\{(\widetilde{\tau}_x-\widetilde{A}(x))/\log\log x\}_{x>0}$ is tight.
\end{theorem}
Theorem \ref{thm:main2} partially resolves Conjecture 1 of \citep{zhang2024modeling}. 
 We sketch its proof below, which will be very similar to the proof of Theorem \ref{thm:main}. 
A key ingredient in the proof of Theorem \ref{thm:main} is the analysis of the expected number $N_n$ of particles present at time $n$, conditioned upon survival. If the branching events are independent across layers, then $N_n=\rho^n$. On the other hand, for the delayed branching BRW, due to the two-step branching regime, the branching events are not independent. 
The following two lemmas perform a similar analysis related to the number of particles for the delayed branching model. In particular, the analogue for our expected number of descendants for each particle, $\rho$, becomes $\tilde{\rho}$ for the delayed branching model.

\begin{lemma}[Lemma 5 of \citep{zhang2024modeling}]\label{lemma:number1}
    Conditioned upon survival, the expected number of particles $N_n$ at time $n$ satisfies  $N_n\asymp\tilde{\rho}^n$.
\end{lemma}

\begin{lemma}
    \label{lemma:number2}
    For $1\leq s\leq n$, let $$N_{n,s}:=\E[\#\{(v,w)\in V_n^2 : v\sim_s w\}]$$be the expected number of pairs of particles at level $n$ whose distance in the genealogical  tree is equal to $2s$. 
    Conditioned upon survival,  $N_{n,s}\asymp \tilde{\rho}^{n+s}$ uniformly in $1\leq s\leq n$.
\end{lemma}

\begin{proof}
By Lemma \ref{lemma:number1}, the expected number of choices of the common ancestor $u$ of $v,w$ is $N_{n-s}\asymp \tilde{\rho}^{n-s}$. We now fix $u$. 
Let $S_u$ denote the survival event of the underlying branching process initiated from $u$. Let $N_{u,s}$ and $N'_{u,s}$ denote the number of particles initiated from $u$ in $s$ steps conditioned on $S$ and $S_u$ respectively. By a coupling argument, if $u$ initiates a branching sub-event of type I,
$$\E[N_{u,s}\bone_{S}]\geq \E[N'_{u,s}\bone_{S_u}]\asymp N_s\asymp \tilde{\rho}^s.$$
Moreover, $\E[N_{u,s}\bone_{S}]$ can be controlled by the expected number of particles without conditioning upon survival, which is $\asymp \tilde{\rho}^s$. 
Otherwise, if $u$ initiates a branching sub-event of type II, an extra factor of $2$ exists but will be absorbed into the asymptotics. Therefore we conclude that $\E[N_{u,s}\mid S]\asymp \tilde{\rho}^s$ (that $\p(S)\in(0,1)$ has been proven in Lemma 5 of \citep{zhang2024modeling}), and hence $N_{n,s}\asymp \tilde{\rho}^{n-s}(\tilde{\rho}^s)^2=\tilde{\rho}^{n+s}$.
\end{proof}

We briefly identify the modifications required for the proof of Theorem \ref{thm:main2}. Note that the only difference between the classical BRW and delayed branching BRW is the branching structure, i.e., the underlying tree that describes the genealogy of the particles. Given that Theorem \ref{theorem:concentration} holds, the proof of Theorem \ref{thm:main} depends on such structure only through  
    \eqref{eq:1stmoment} and \eqref{eq:rhopower}. The corresponding modifications can be made by replacing $\rho$ by $\tilde{\rho}$ in view of Lemmas \ref{lemma:number1} and \ref{lemma:number2}. 
    The proof of Theorem \ref{theorem:concentration} depends on large deviation estimates for the maximum of delayed branching BRW in \eqref{eq:ld}. Nevertheless, this applies by performing the same proof as in \citep{gantert2018large} by comparing against the maximum of independent random walks (Theorems 3.1 and 5.2 therein). The other changes needed in Theorem \ref{theorem:concentration} are minor and will be omitted for brevity, as the constants can always be adjusted to take into account the delayed branching property. Otherwise, the proof of Theorem \ref{thm:main2} goes through almost verbatim.

The same arguments apply for non-spherically symmetric jumps, where we instead assume (A1), (A2), (A5), and (A6) and replace Theorem \ref{thm:main} by Theorem \ref{thm:main long}.

\subsection{Finer asymptotics for non-spherically symmetric jumps}\label{sec:non-spherical}
We discuss a few cases with non-spherically symmetric jumps where the asymptotic \eqref{eq:tauxasymp long} can be improved. 
The only instance where we explicitly required spherical symmetry of the law of $\bxi$ for the proof of Theorem \ref{thm:main} lies in the proof of Proposition \ref{prop:ulb gaussian approx}, where we apply the change of measure \eqref{eq:dq/dp}. More precisely, applying the exponential tilt only in the $x$-coordinate in \eqref{eq:dq/dp} does not guarantee that the random variable $\bxi$ is centered in other coordinates under the new measure; cf.~\eqref{eq:dqdp long}. 

In certain nice cases, the non-spherically symmetric case can be analyzed using the same approach. Let $I$ be the large deviation rate function for the jump $\bxi$ and $\widehat{c}_1$ be such that $I(\widehat{c}_1,\z)=\log\rho$. 
One may also define the constant $c_1$ that arose in assumption (A4) for a non-spherically symmetric jump $\bxi$, through the relation $I^{(1)}(c_1)=\log\rho$ where $I^{(1)}$ is the large deviation rate function for the first marginal of $\bxi$. 
If $\widehat{c}_1=c_1$, we can show using essentially the same proof of Theorem \ref{thm:main} that
$$\tau_x=\frac{x}{\widehat{c}_1}+\frac{d+2}{2\widehat{c}_1\partial_{x_1}I(\widehat{c}_1,\z)}\,\log x+O_\p(\log\log x).$$
The following proposition gives a necessary and sufficient condition for this to happen for all $\rho>1$.

\begin{proposition}\label{prop:c1}
    It holds that $\widehat{c}_1=c_1$ for all $\rho>1$ if and only if $\E[{\xi^{(i)}}\mid{\xi^{(1)}}]=0$ for all $2\leq i\leq d$, where ${\xi^{(i)}}$ denotes the $i$-th coordinate of the $\R^d$-valued random variable $\bxi$.
\end{proposition}

\begin{proof}
Since $I^{(1)}$ is strictly increasing on its domain intersecting $(0,\infty)$, an equivalent condition for $\widehat{c}_1=c_1$ for all $\rho$ is that $I^{(1)}(c_1)=I(c_1,\z)$ for all $c_1$. By definition, this amounts to
    $$\sup_{\lambda_1\in\R}\left(\lambda_1c_1-\E[e^{\lambda_1{\xi^{(1)}}}]\right)=\sup_{\lambda_1,\dots,\lambda_d\in\R}\left(\lambda_1c_1-\E[e^{\bla\cdot\bxi}]\right).$$
Since the range of $I'$ is $\R_+$, this is equivalent to
$$\forall \lambda_1\in\R,~\inf_{\lambda_2,\dots,\lambda_d\in\R}\E[e^{\bla\cdot\bxi}]=\E[e^{\lambda_1{\xi^{(1)}}}].$$
By taking partial derivatives with respect to $\lambda_i$ at $\lambda_i=0,~2\leq i\leq d$, we obtain that for each $2\leq i\leq d,~\E[{\xi^{(i)}}e^{\lambda_1{\xi^{(1)}}}]=0$. Since this holds for all $\lambda_1\in\R$, by the uniqueness of the Laplace transform, we must have $\E[{\xi^{(i)}}\mid{\xi^{(1)}}]=0$.  The reverse direction is obvious by Jensen's inequality.  
\end{proof}

There is yet a different method to deal with distributions on $\R^d$ that can be transformed into a spherically symmetric distribution under an invertible linear transformation of space. Suppose that $\bxi\dd T(\bzeta)$ for some spherically symmetric distribution $\bzeta$ and invertible linear map $T$.
The FPT for the BRW with jump $\bxi$ to $B_x$ is then equivalent to the FPT for the BRW with jump $\bzeta$ to $T^{-1}(B_x)$. Since Theorem \ref{thm:main} holds with the radius of the ball replaced by any constant (with the same proof) and $T$ is invertible, we may without loss of generality replace $T^{-1}(B_x)$ by $B_{T^{-1}\bx}$, where $\bx=(x,\z)$. 
This leads to 
\begin{align}
    \tau_x=\frac{\n{T^{-1}\bx}}{\widehat{c}_1(\bzeta)}+\frac{d+2}{2\widehat{c}_1(\bzeta)\partial_{x_1}I_\bzeta(\widehat{c}_1(\bzeta),\z)}\,\log x+O_\p(\log\log x).\label{eq:tauxlong1}
\end{align}
Summarizing the discussions above, we arrive at the following.
\begin{theorem}\label{thm:main_asymmetric_better}Assume assumption (A1). 
Suppose that either of the following holds:
    \begin{enumerate}[(i)]
        \item the law of $\bxi=({\xi^{(1)}},\dots,{\xi^{(d)}})$ satisfies $\E[{\xi^{(i)}}\mid{\xi^{(1)}}]=0$ a.s.~for all $2\leq j\leq d$, and satisfies (A2) and (A4);
        \item there exists an invertible linear transformation $T$ on $\R^d$ with $\bxi\dd T(\bzeta)$ for some spherically symmetric distribution $\bzeta$ satisfying (A2)--(A4).
    \end{enumerate}
    Then the first passage time $\tau_x$ has the asymptotic
    \begin{align}
        \tau_x=\frac{x}{\widehat{c}_1}+\frac{d+2}{2\widehat{c}_1\partial_{x_1}I(\widehat{c}_1,\z)}\,\log x+O_\p(\log\log x).\label{eq:tauxlong2}
    \end{align}
\end{theorem}

\begin{proof}
    By comparing \eqref{eq:tauxlong1} and \eqref{eq:tauxlong2}, it suffices to show for case (ii) that
    \begin{align}
        \frac{x}{\widehat{c}_1(\bxi)}=\frac{\n{T^{-1}\bx}}{\widehat{c}_1(\bzeta)}\label{eq:ts1}
    \end{align}
    and that
    \begin{align}
        \widehat{c}_1(\bxi)\partial_{x_1}I_\bxi(\widehat{c}_1(\bxi),\z)=\widehat{c}_1(\bzeta)\partial_{x_1}I_\bzeta(\widehat{c}_1(\bzeta),\z).\label{eq:ts2}
    \end{align}
Let $\mathbf{e}_1=(1,\z)\in\R^d$.
By definition \eqref{eq:I long} and since $T$ is invertible, $I_\bxi(\bx)=I_\bzeta(T^{-1}\bx)$. 
By spherical symmetry of $\bzeta$,
$$I_\bxi(\widehat{c}_1(\bxi),\z)=I_\bzeta\left(\frac{T^{-1}(\widehat{c}_1(\bzeta),\z)}{\n{T^{-1}\mathbf{e}_1}}\right)=I_\bzeta\left(\frac{\n{T^{-1}(\widehat{c}_1(\bzeta),\z)}}{\n{T^{-1}\mathbf{e}_1}}\right)=I_\bzeta(\widehat{c}_1(\bzeta),\z),$$
proving \eqref{eq:ts1}. Denote the radial component of $I_\bzeta$ by $f$. By applying the multivariate chain rule, we have
\begin{align*}
    \partial_{x_1}I_\bxi(\widehat{c}_1(\bxi),\z)&=\nabla I_\bzeta(T^{-1}(\widehat{c}_1(\bxi),\z)) \cdot \partial_{x_1}T^{-1}(\widehat{c}_1(\bxi),\z)\\
    &=f'\left(\n{T^{-1}(\widehat{c}_1(\bxi),\z)}\right)\frac{T^{-1}(\widehat{c}_1(\bxi),\z)}{\n{T^{-1}(\widehat{c}_1(\bxi),\z)}}\cdot T^{-1}\mathbf{e}_1\\
    &=f'\left(\widehat{c}_1(\bzeta)\right)\frac{\widehat{c}_1(\bxi)\n{T^{-1}\mathbf{e}_1}^2}{\widehat{c}_1(\bzeta)}\\
    &=\partial_{x_1}I_\bzeta(\widehat{c}_1(\bzeta),\z)\n{T^{-1}\mathbf{e}_1}.
\end{align*}
Combined with \eqref{eq:ts1}, this establishes \eqref{eq:ts2}.
\end{proof}

Condition (i) handles the case where the law of $\bxi$ is given by a product measure. 
The condition (ii) above is satisfied, e.g., for centered non-degenerate elliptical distributions. In particular, this applies to non-degenerate Gaussian distributions, which we explicitly compute in the following example.

\begin{example}
    Suppose that $\bxi$ is centered Gaussian with an invertible covariance matrix $\Sigma=AA^{\top}$. In this case, $\bxi=A\bzeta$ where $\bzeta$ is standard Gaussian. In this case, $I_\bxi(\bx)=\bx^{\top}\Sigma^{-1}\bx/2$, so that $\widehat{c}_1=\sqrt{2\log\rho/(\Sigma^{-1})_{11}}$ and $\partial_{x_1}I(\widehat{c}_1,\z)=\sqrt{2(\log\rho)(\Sigma^{-1})_{11}}$. By Theorem \ref{thm:main_asymmetric_better}, we have
    $$\tau_x=\frac{x}{\sqrt{2\log\rho/(\Sigma^{-1})_{11}}}+\frac{d+2}{4\log\rho}\,\log x+O_\p(\log\log x).$$
    One can check indeed that this also follows from \eqref{eq:tauxlong1}, using $\n{A^{-1}\bx}=x\sqrt{(\Sigma^{-1})_{11}}$ where $\bx=(x,\z)$.
\end{example}

\section{Numerical analysis}
\label{sec:numerical}

In this section, we test Theorem \ref{thm:main long} numerically. Since tracking the locations of all particles in a BRW is time-consuming and memory-intensive, we introduce a path-purging algorithm that removes particles that are highly unlikely to have descendants that realize the first passage time.\footnote{An efficient algorithm that simulates the tip of one-dimensional BBM has been proposed in \citep{brunet2020generate}. Unfortunately, the arguments therein do not seem to carry over to general BRWs in higher dimensions.} A preliminary version of the algorithm for spherically symmetric jumps can be found in Section 2.3 of \citep{zhang2024modeling}.

Recall that $I$ denotes the large deviation rate function of the jump distribution $\bxi$. In the path-purging algorithm, we purge along the normal vector 
\begin{align}
    \mathbf{n}=\nabla I(\widehat{c}_1,\z).\label{eq:n}
\end{align}
That is, we retain particles that have the $q_c/2$ largest inner product with $\mathbf{{n}}$, where $q_c$ is a prefixed large number. To intuitively justify this, suppose that a particle is found at $\bx'=(x',\z)$ where $x'<x$. The expected time for the BRW initiated at $\bx'$ to arrive at $B_x$ is around $(x-x')/\widehat{c}_1$. In order for another particle located at $\bx'+\bv$ to reach $B_x$ within roughly the same amount of time, we need
$$e^{-(x-x')I(\widehat{c}_1,\z)/\widehat{c}_1}=e^{-(x-x')I(\widehat{c}_1(\bx-\bx'-\bv)/(x-x'))/\widehat{c}_1}.$$
For small $\bv$, this holds if and only if $\bv\cdot \nabla I(\widehat{c}_1,\z)=0$, independent of the choice of $x'$.

If $\bxi$ is spherically symmetric, or satisfies condition (i) of Theorem \ref{thm:main_asymmetric_better}, the normal vector $\mathbf{n}$ is then parallel to the direction towards the termination site $\bx=(x,\z)$.


  In our implementation, we take $q_c=9000$ and consider BRW where particles only branch into three (i.e., $\{i>1:p_i>0\}=\{3\}$) in dimension $d=3$. At times, we also consider the delayed branching model from Section \ref{sec:newbrwproof}. We compute the first passage times for offspring distributions with $p_3\in \{0.05\ell:1\leq\ell\leq 19\}$ and $p_1=1-p_3$ (meaning $p_0=0$, unless otherwise stated) while varying $x$ and the distribution of $\bxi$. We showcase the following two cases:
  \begin{itemize}
      \item In Section \ref{sec:numerical brw}, we discuss the case where $\bxi$ is uniformly distributed on $\S^2$.
      \item In Section \ref{sec:numerical gbrw}, we deal with Gaussian distributions with a general covariance matrix.
  \end{itemize}

Finally, we note that the main result of \citep{gantert2011asymptotics} shows that there rarely exists an infinite path in a one-dimensional BRW that always stays above the curve $n\mapsto (c_1-\ee)n$ for $\ee>0$ small. While this suggests that our path purging algorithm cannot be backed up by theory for a fixed $q_c$ as $x\to\infty$, the simulations work effectively well in the regime where $q_c$ is quite large compared to $x$.


 \subsection{Uniform distribution on the sphere}\label{sec:numerical brw}

 Consider the case when $\bxi$ is uniformly distributed in $\S^2$, i.e., we are in the spherically symmetric case. 
 The statistic that we compare is the coefficient of the linear term in the estimation of $\tau_x$ in \eqref{eq:tauxasymp}. For any given $p_3\in \{0.05\ell:1\leq\ell\leq 19\}$, in the numerical BRW simulation, we calculate the FPT distribution at different offset distances, $x \in \{0.25L_x,\, 0.5L_x,\, 0.75L_x,\, L_x\}$, where $L_x=65.5$. The mean FPT $\E[\tau_x]$ is then fit to a function of the form
\begin{equation}
    \E[\tau_x] =  \frac{x}{c_1} + B \log x + C. \label{eq:fit_tau_approx}
\end{equation}

%
\begin{figure}[t]
    \centering
        \subfigure[]{\includegraphics[width=0.45\textwidth]{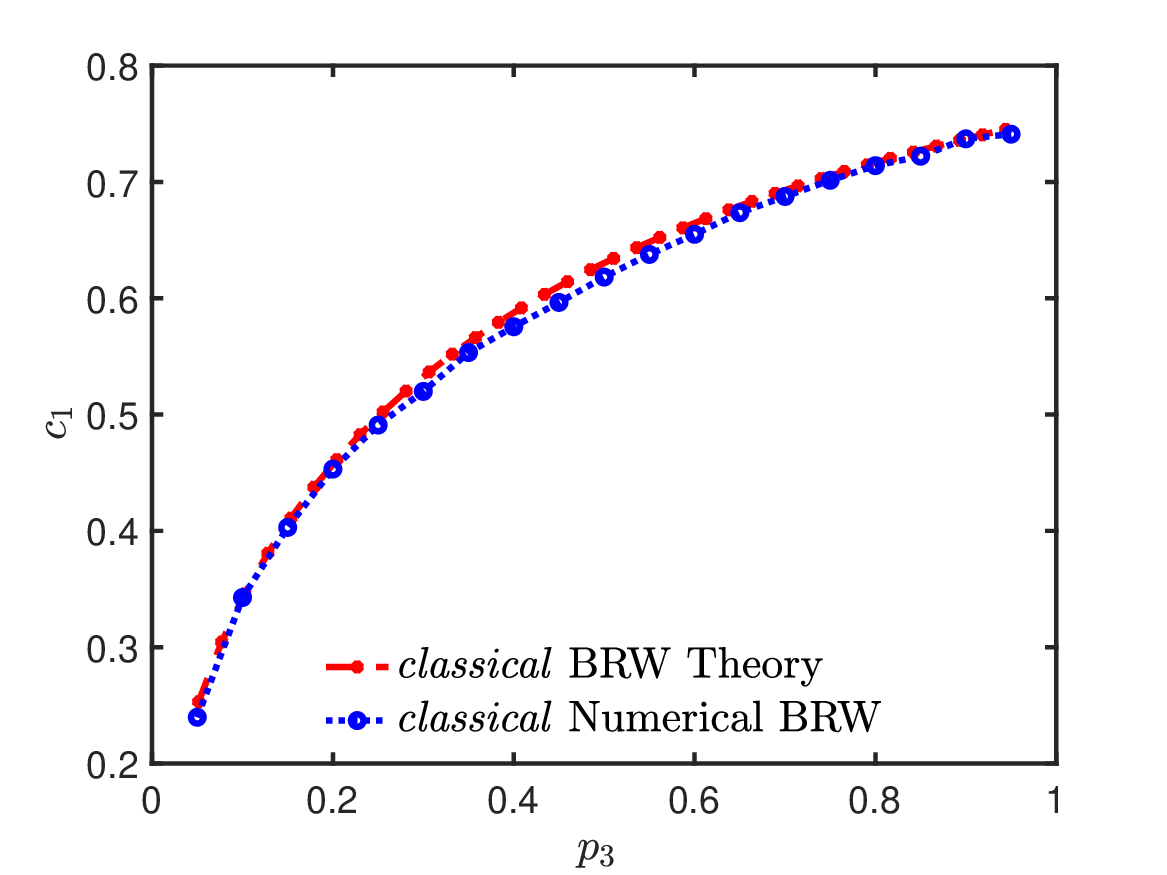}}
         \subfigure[]{\includegraphics[width=0.45\textwidth]{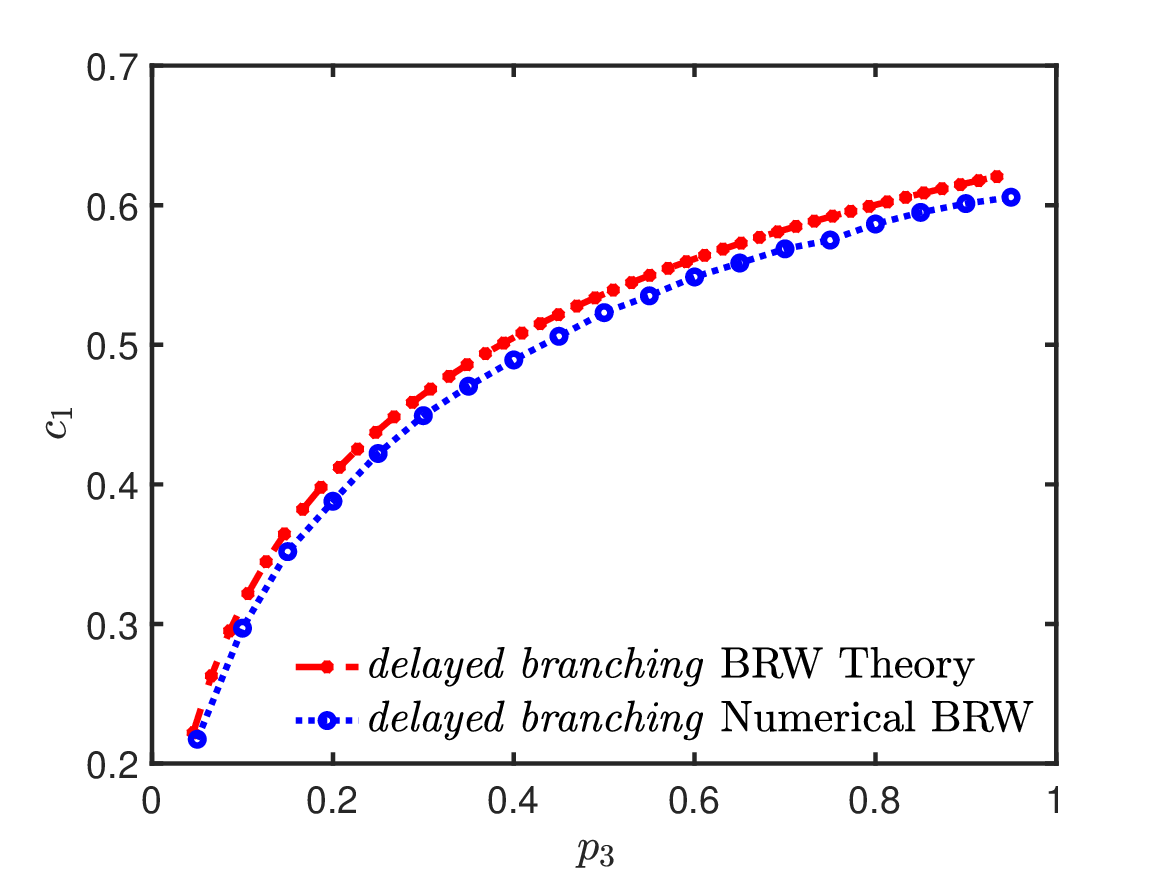}}
    \caption{The numerical 3D BRW obtained $c_1$ compared against the reference calculations across different $p_3$ values for the  (a) \rrm~and (b) {delayed branching} BRW with $p_0=0.004$. The mean of the FPT at each $x$ is computed from $2000$ samples.
    }
    \label{fig:c1_rho_diff_models}
\end{figure}
The numerical implementation validates the theoretical predictions to the classical and delayed branching BRW models, as shown in Figure~\ref{fig:c1_rho_diff_models}.
%
This also indicates that the numerical approximation of {path purging} does not affect the linear coefficient of the scaling behavior of the FPT distribution, and can be used as a suitable approximation for carrying out the numerical BRW simulations. The error between the theoretical and numerical implementation can be attributed to the limited range of $x$ over which the fit of \eqref{eq:fit_tau_approx} is executed. Figure~\ref{fig:c1_rho_diff_models} suggests that \emph{classical} model requires a smaller range of $x$ for the FPT evaluation for a higher fidelity estimation of $c_1$ in comparison to the \emph{delayed branching} model.
\begin{figure}[ht!]
    \centering
        \subfigure[]{\includegraphics[width=0.45\textwidth]{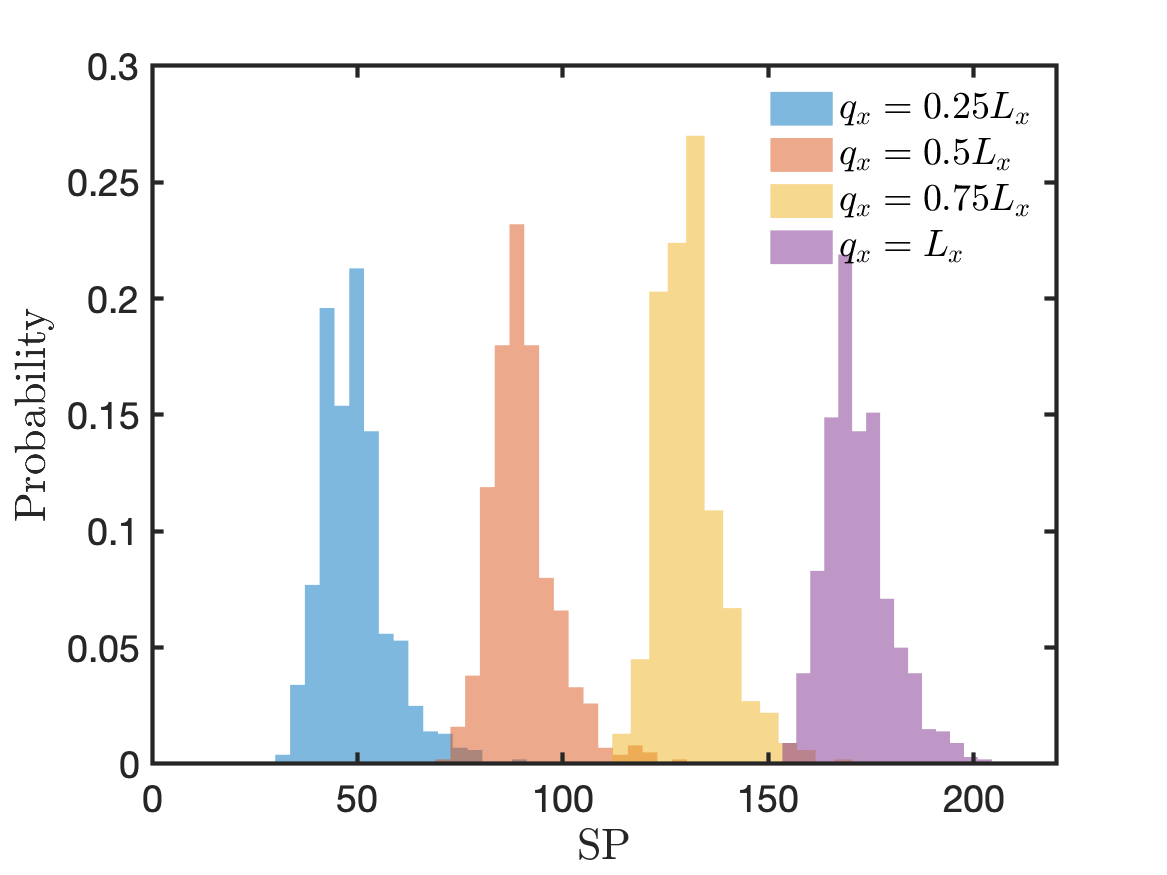}\label{fig:sig_valid_a}}
        \subfigure[]{\includegraphics[width=0.45\textwidth]{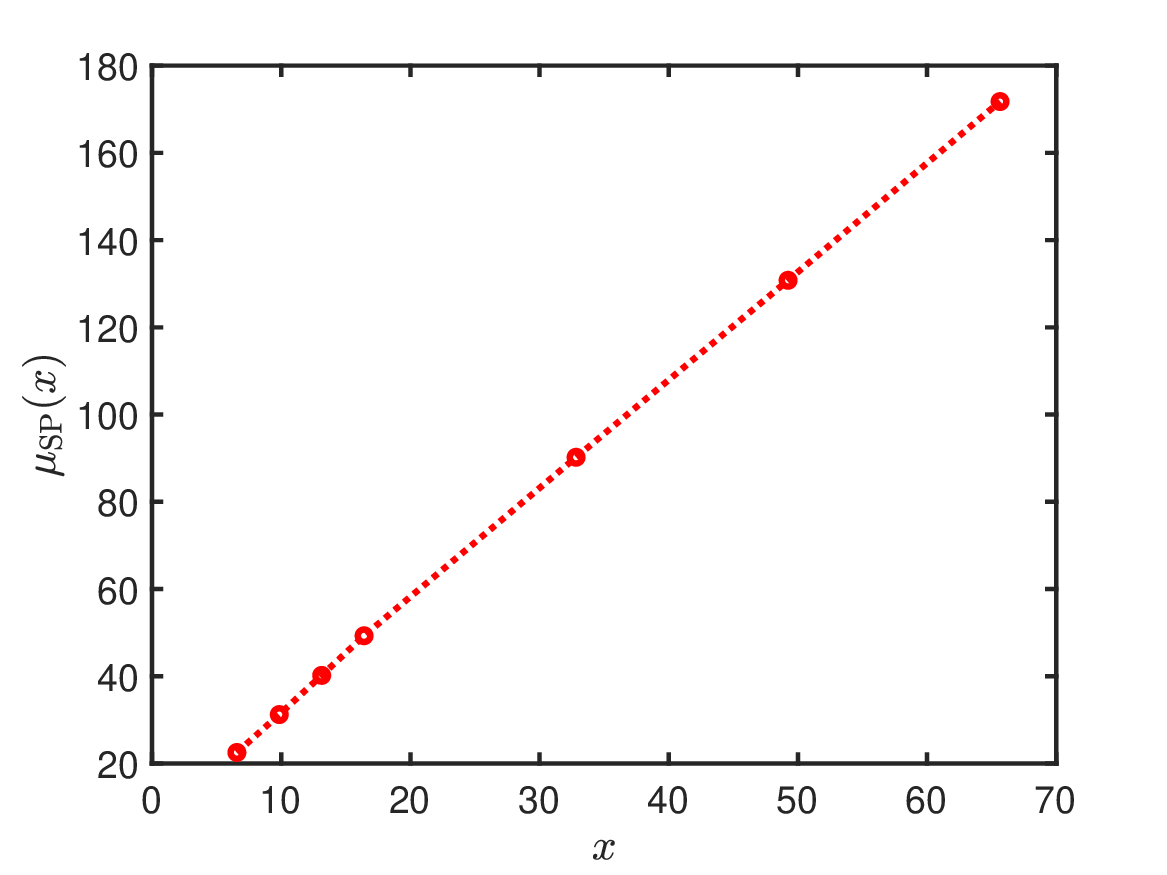}\label{fig:cgmd linear}}
        \subfigure[]{\includegraphics[width=0.45\textwidth]{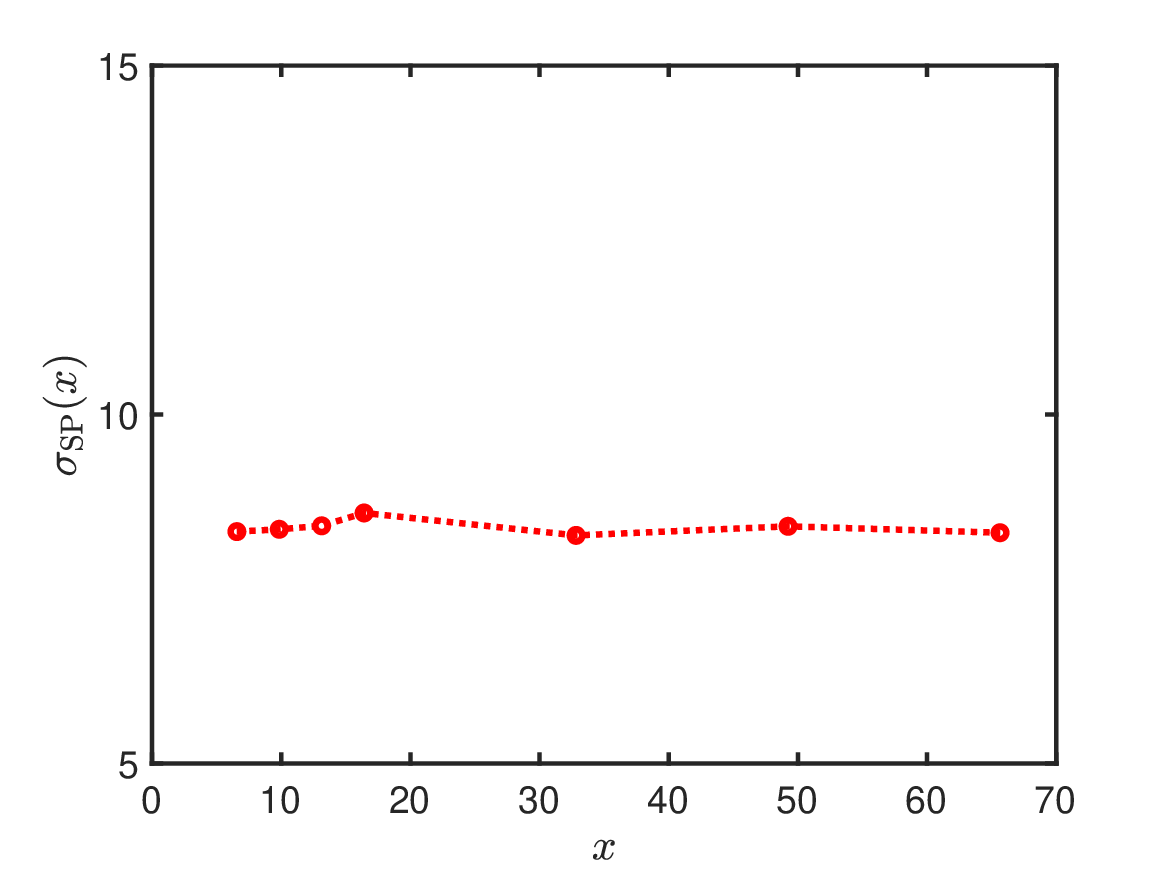}\label{fig:sig_valid_b}}
       \hspace{-1cm}
    \caption{(a) The FPT distributions at different values of $x$ for the \rrm. The (b) mean and (c) standard deviation of the SP as a function of $x$, respectively. The data points correspond to $x = 0.1L_x,\,0.15L_x,\,0.2L_x,\,0.25L_x,\,0.5L_x,\,0.75L_x,\,L_x$ and $p_3=0.0856$ for the \rrm, where $L_x=65.5$. The mean of the FPT at each $x$ is computed from $500$ samples.
    }
    \label{fig:sig_valid}
\end{figure}
Figure~\ref{fig:sig_valid_a} shows the FPT distribution at different $x$. Figure~\ref{fig:cgmd linear} shows the linear increase in the mean FPT (negligible logarithmic effect), whereas Figure~\ref{fig:sig_valid_b} validates the asymptotic in \eqref{eq:tauxasymp2} by showing a negligible change in the standard deviation of the FPT distribution with increasing offset distance, $x$.

\subsection{Gaussian distributions}
\label{sec:numerical gbrw}
Suppose now that the jump $\bxi$ is centered Gaussian with a positive definite covariance $\Sigma$. We have $\Lambda(\bla)=\bla^\top \Sigma\bla/2$ and $I(\bx)=\bx^\top \Sigma^{-1}\bx/2$. In particular, $\bc_2=\nabla I(x/n,\z)=((\Sigma^{-1})_{11}x/n,\dots,(\Sigma^{-1})_{1d}x/n)$ and $\Lambda(\bc_2)=(x/n)^2(\Sigma^{-1})_{11}/2$. 
By \eqref{eq:n}, the normal vector $\bn$ is a constant multiple of $\Sigma^{-1}\mathbf{e}_1$, where $\mathbf{e}_1=(1,\z)\in\R^d$. If $\Sigma$ is diagonal (independent jumps), the normal vector $\bn$ is then parallel to the direction of the termination site $\bx$. This is the case in our previous work~\citep{zhang2024modeling} as opposed to the general case of non-spherically symmetric and dependent jumps that we will discuss in this section.


 
 We look at three specific examples: (i) symmetric and independent (S.I.) jumps (presented earlier~\citep{zhang2024modeling}), (ii) non-symmetric and independent (N-S.I.) jumps, and (iii) non-symmetric and dependent (N-S.D.) jumps. The corresponding covariances for the three cases are given by
 \begin{equation*}
     \Sigma_{\rm S.I.}=
     \begin{bmatrix}
         1 &0 &0 \\
         0&1 &0 \\
         0&0&1
     \end{bmatrix}
,\,\,\,\,\,\,
     \Sigma_{\rm N\textrm{-}S.I.}=
     \begin{bmatrix}
         1&0 &0 \\
         0&1.5 &0 \\
         0&0&0.5
     \end{bmatrix}
,\,\,\,\,\,\,
     \Sigma_{\rm N\textrm{-}S.D.}=
     \begin{bmatrix}
         1&0.5 &0.25 \\
         0.5&1.5 &0.5 \\
         0.25&0.5&0.5
     \end{bmatrix}
     .
 \end{equation*}
We follow the same procedures as Section \ref{sec:numerical brw} that estimate the statistic $\widehat{c}_1$ in \eqref{eq:tauxasymp long} through fitting \eqref{eq:fit_tau_approx}. Note that at any given $p_3$, the $\widehat{c}_1$ is a function of $(\Sigma^{-1})_{11}$. As a result, we expect the $c_1$ of the Gaussian BRW with S.I.~and N-S.I.~jumps to be identical, as confirmed by our numerical results in Figures~\ref{fig:c1_rho_gbrw_models_a} and~\ref{fig:c1_rho_gbrw_models_b}. We use $p_0=0$ for the comparison between the numerical and theoretical results in this section.
\begin{figure}[t]
    \centering
        \subfigure[]{\includegraphics[width=0.45\textwidth]{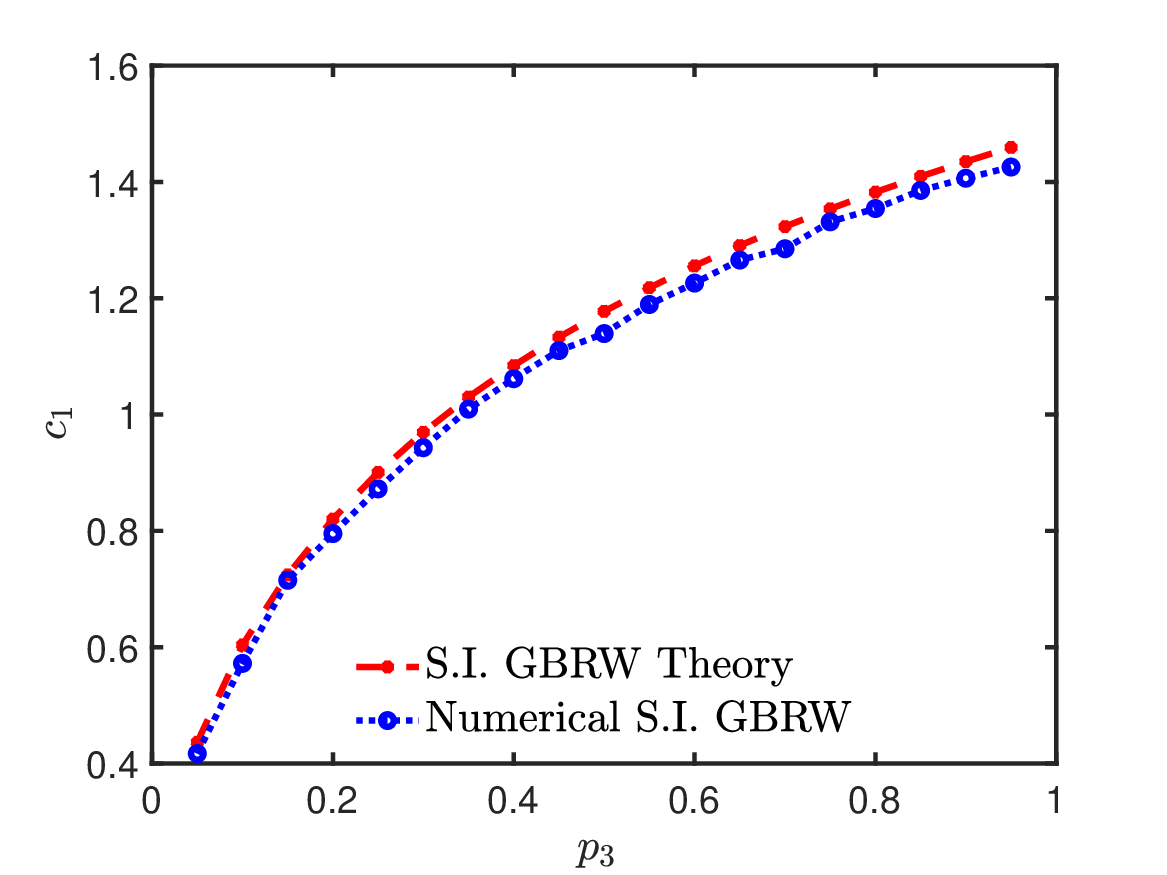}\label{fig:c1_rho_gbrw_models_a}}
        \subfigure[]{\includegraphics[width=0.45\textwidth]{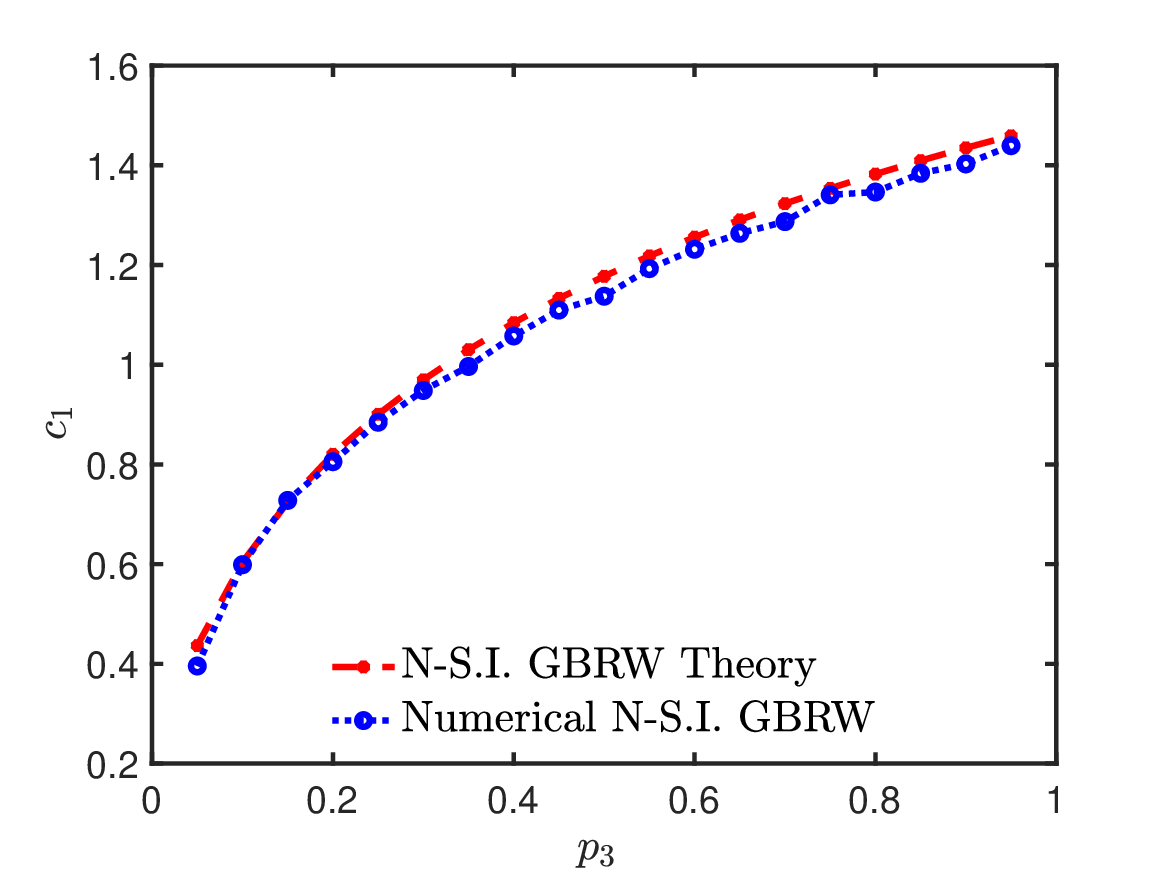}\label{fig:c1_rho_gbrw_models_b}}
        \subfigure[]{\includegraphics[width=0.45\textwidth]{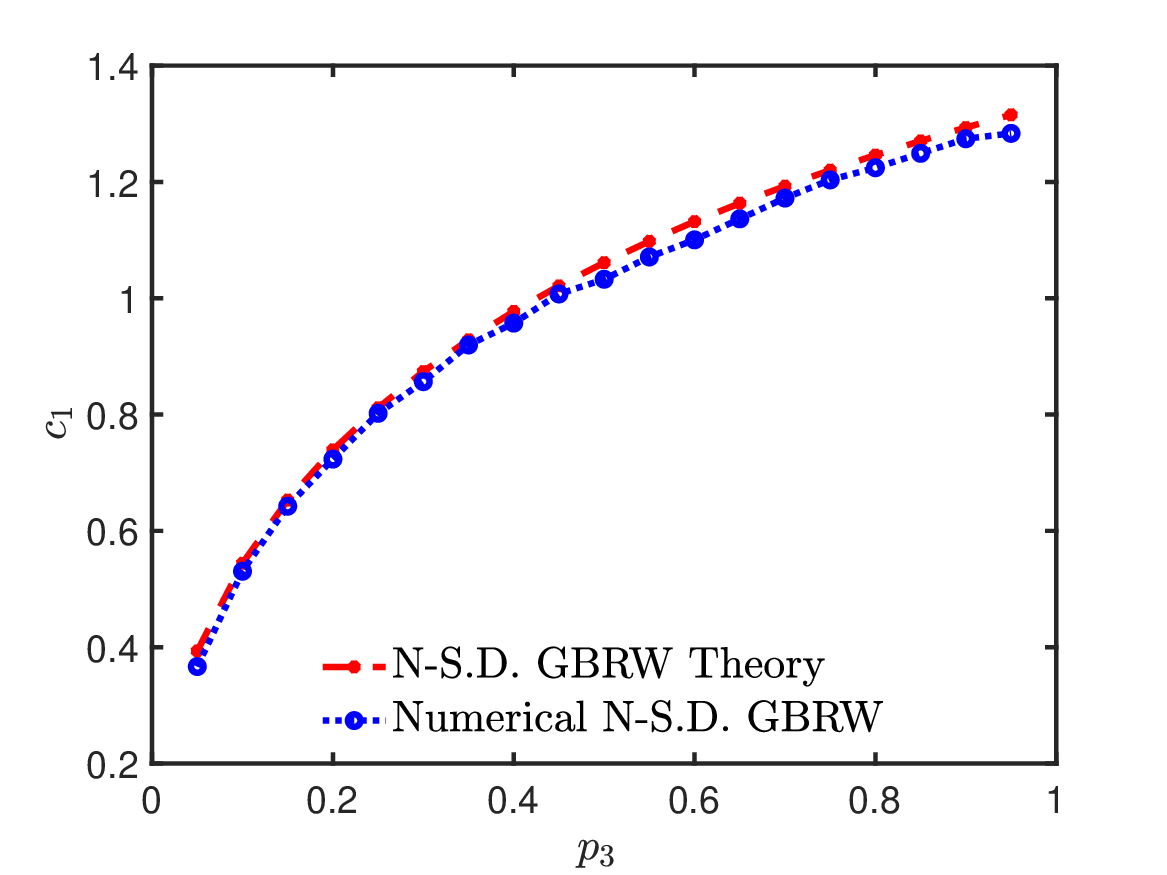}\label{fig:c1_rho_gbrw_models_c}}
    \caption{The numerically obtained $\widehat{c}_1$ compared against the theoretical $\widehat{c}_1$ for Gaussian BRW with  (a) symmetric and independent (S.I.) jumps, (b) non-symmetric and independent (N-S.I.) jumps, and (c) non-symmetric and dependent (N-S.D.) jumps. The mean of the FPT at each $x$ is computed from $10000$ samples.
    }
    \label{fig:c1_rho_gbrw_models}
\end{figure}
In the case of the N-S.D. jumps, the $(\Sigma^{-1})_{11}$ is larger and we expect the $\widehat{c}_1$ to be smaller. Figure~\ref{fig:c1_rho_gbrw_models_c} shows the agreement of the numerical implementation in being able to capture the lower $c_1$ at every $p_3$. The agreement of the numerical results could be improved by incorporating a larger number of paths for the computation of the mean FPT or by increasing the range of $x$ over which the mean FPT is computed to execute the fit in \eqref{eq:fit_tau_approx}. To ascertain that 500 samples are sufficient to capture the mean, we present the mean FPT $\E[\tau_x]$ at different values of $p_3$ at $x=65.5$ for the Gaussian BRW with N-S.I.~and N-S.D.~jumps.
\begin{figure}[t]
    \centering
        \subfigure[]{\includegraphics[width=0.45\textwidth]{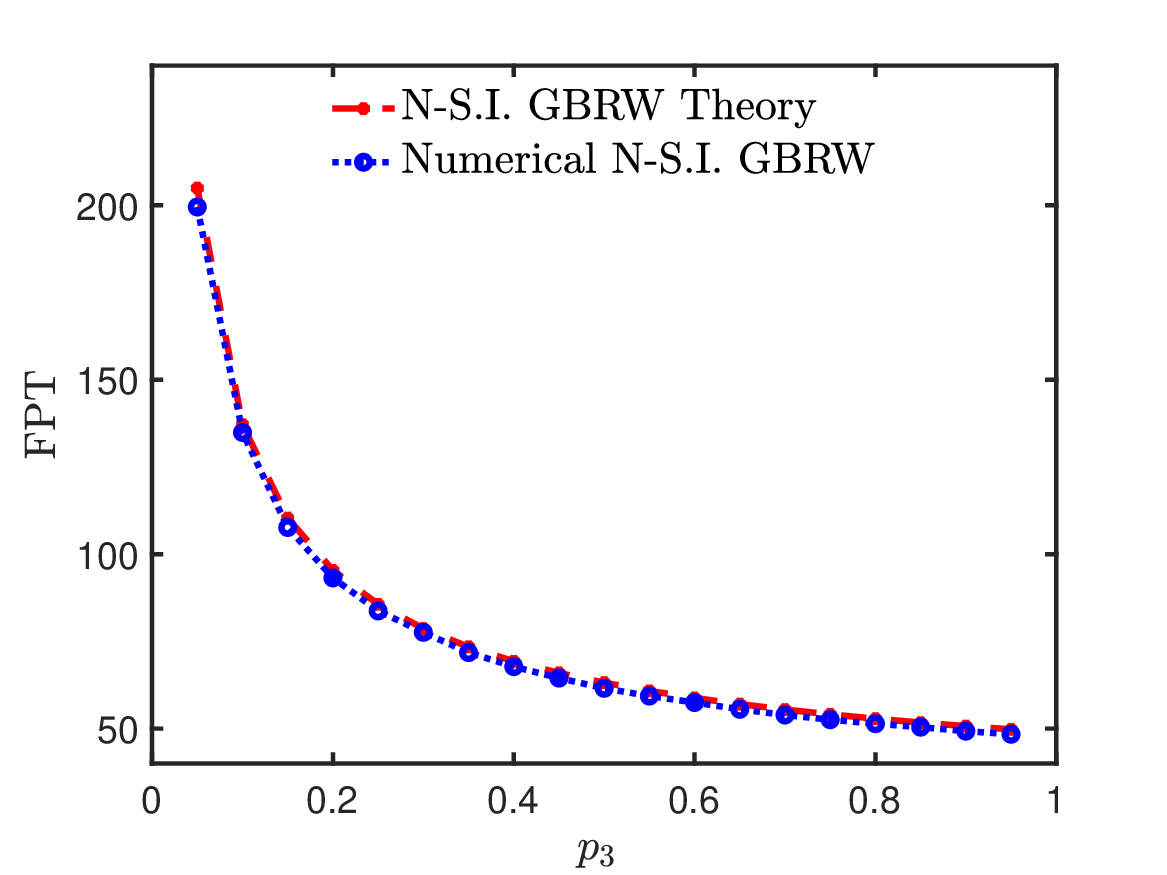}\label{fig:fpt_gbrw_models_a}}
        \subfigure[]{\includegraphics[width=0.45\textwidth]{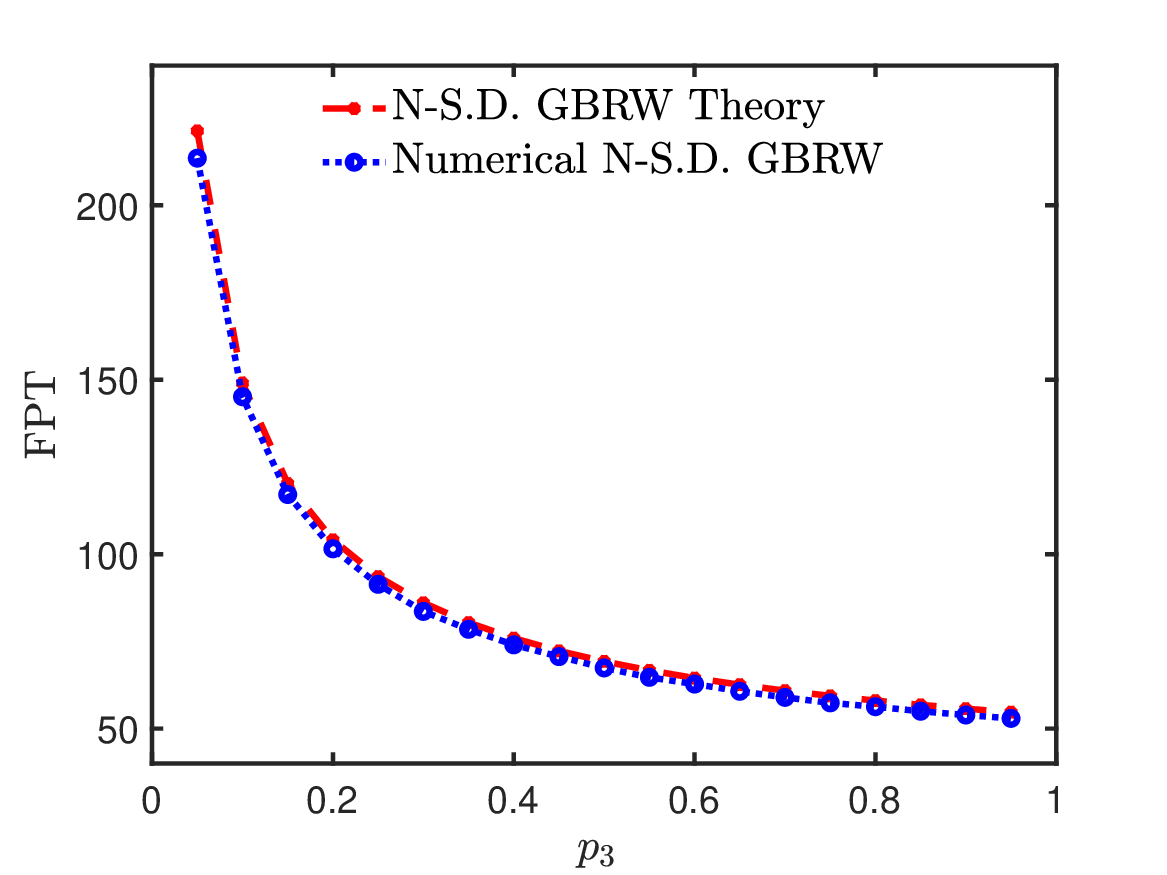}\label{fig:fpt_gbrw_models_b}}
    \caption{The numerically obtained expected FPT compared against the theoretical asymptotic \eqref{eq:tauxasymp long} for Gaussian BRW with (a) non-symmetric and independent (N-S.I.) jumps and (b) non-symmetric and dependent (N-S.D.) jumps. The mean of the FPT is calculated at $x=L_x$, where $L_x=65.5$, from $10000$ samples.
    }
    \label{fig:fpt_gbrw_models}
\end{figure}
Figure~\ref{fig:fpt_gbrw_models} confirms that the statistics from 500 independent samples are sufficient to capture the mean of the distribution. This confirms that the approximation of $c_1$ can be improved by fitting the FPT values over a larger range of $x$. 

\section{Concluding remarks}\label{sec:conclusion}

In this paper, we study the first passage times of an $\R^d$-valued branching random walk (BRW) to a shifted unit ball. The first passage time asymptotics consist of a linear term and a logarithmic correction term, as a function of the Euclidean distance of the target ball from the origin. We discuss extensions of this result to the delayed branching BRW model. As an immediate application, we obtain theoretical predictions of the shortest-path statistics for polymer networks consisting of long chains and random cross-links. We conclude with a conjecture on finer asymptotics of the first passage times. 

The first passage times of the branching Brownian motion, the continuous-time sibling of the BRW, possess better tractability due to its connection to the Fisher-KPP equation. It is proven in Section 4.1 of \citep{zhang2024modeling} the finer asymptotic for the first passage times of a standard branching Brownian motion that
\begin{align}
    \tau_x=\frac{x}{\sqrt{2}}+\frac{d+2}{4}\,\log x+O_\p(1),\label{eq:bbm}
\end{align}
where the $O_\p(1)$ is tight. We refer the readers to \citep{zhang2024modeling} for the proof and a detailed discussion.
In view also of the tightness result of Theorem \ref{theorem:concentration}, it is natural to postulate the following conjecture for BRW. Recall the constant $\widehat{c}_1$ from Section \ref{sec:main results}.

\begin{conjecture}\label{conjecture}
    Assume (A1)--(A4), or (A1), (A2),  (A5), and (A6). Conditioned upon survival, the first passage time for BRW in dimension $d\geq 2$ to $B_x$ is given by
$$\tau_x=A(x)+O_\p(1)=\frac{x}{\widehat{c}_1}+\frac{d+2}{2\,\widehat{c}_1\partial_{x_1}I(\widehat{c}_1,\z)}\,\log x+O_\p(1),$$where the $O_\p(1)$ is tight. 
\end{conjecture}
Conjecture \ref{conjecture} is also supported by our numerical simulations; see Figure \ref{fig:fpt_gbrw_models}.

The main technical obstruction to improving upon $o_\p(\log x)$ lies in the application of ballot theorems---the cone $\K$ in Sections \ref{sec:ub of thm main long} and \ref{sec:lb of thm main long} cannot be simply replaced by a half-space as certain sums would not converge (if this was the case, Conjecture \ref{conjecture} would hold). One possible approach is to develop bounds for random walks in time-dependent (and possibly non-circular) cones, where $\K=\K_n$ may depend on $n$, the number of steps in the random walk. The $o_\p(\log x)$ may be improved if one applies such estimates in the case where $\K_n\to \K_{\pi/2}(-\bc_2)$ with a certain rate. 

 \begin{remark}[Status of Conjecture~\ref{conjecture}]\label{rem:conj}
After the first version of this paper was posted, Conjecture~\ref{conjecture} was proved in \citep{blanchet2024tightness}. 
We retain the conjecture here for historical context and refer the readers to Theorems 1 and 2 of \citep{blanchet2024tightness}. Nevertheless, a stronger form of Conjecture \ref{conjecture} --- that the $O_\p(1)$ term converges in law --- remains unsolved.
\end{remark}

\begin{acks}
    We thank Amir Dembo, Haotian Gu, and Alexandra Stavrianidi for helpful discussions.  We also thank two referees for their detailed comments that significantly improved the presentation of this paper.
\end{acks}

\begin{funding}
    The material in this paper is based upon work supported by the Air Force Office of Scientific Research under award number FA9550-20-1-0397. Additional support is gratefully acknowledged from NSF 1915967, 2118199, 2229012, 2312204.
\end{funding}

\section*{Data Availability}The open source code for the branching random walk calculations can be accessed from \href{https://gitlab.com/micronano_public/PolyBranchX}{PolyBranchX}.

\begin{appendix}
\section*{Large deviation estimates}
In this appendix, we collect the deferred proofs of the large deviation estimates (Proposition~\ref{prop:ulb gaussian approx} in Section~\ref{sec:rw}, Lemma~\ref{lemma:weak uniform LD} in Section~\ref{sec:concentration}, and Proposition~\ref{prop:ulb gaussian approx long} in Section~\ref{242}). We first quote a result on a quantitative version of the central limit theorem for non-lattice distributions.

\begin{lemma}\label{lemma:edgeworthok}
Consider $d\geq 2$ and a compact set $K\subseteq \R^d$.    Suppose that the radial component $R$ of the spherically symmetric random variable $\bxi$ has exponential moments and $\p(R=0)<1$. 
    Let $\bla\in\R^d$ be such that $\E_\p[e^{\bla\cdot\bX}]$ exists and define a probability measure $\q$ by $\d\q/\d\p(\bx):=e^{\bla\cdot\bx}/\E_{\p}[e^{\bla\cdot\bX}]$ where $\p$ is the law of $\bX$.  
    Then the law of $\bxi$ under $\q$ satisfies the strong non-lattice property
    \begin{align}
        \limsup_{|\bt|\to\infty}|\E_\q[e^{i\bt\cdot \bxi}]|<1.\label{eq:nonlattice}
    \end{align}
    Moreover, consider an i.i.d.~sequence $\{\bX_n\}_{n\geq 1}$ with the same distribution as $\bxi$ under $\q$ with a positive definite covariance matrix $V$ and suppose that $\E_\q[\n{\bxi}^{d+2}]<\infty$. Then
    \begin{align}
        \sup_{B}\bigg|\p\Big(\frac{1}{\sqrt{n}}\sum_{j=1}^n(\bX_j-\E_\q[\bX_j])\in B\Big)-h(B)\bigg|=o(n^{-d/2}),\label{eq:edgeworth}
    \end{align}
    where the supremum is taken over all boxes $B\subseteq K$, $h(B)=(1+o(1))\mu_{\z,V}(B)$ as $n\to\infty$, and $\mu_{\z,V}$ is the Gaussian measure on $\R^d$ with mean $\z$ and covariance $V$.
\end{lemma}

\begin{proof}
Let us first assume that $\p(R=0)=0$. 
Denote by $\bXi=(\Xi_1,\dots,\Xi_d)$ the uniform distribution on $\mathbb S^{d-1}$. Suppose that $\bxi_1=R_1\bXi^{(1)}$ and $\bxi_2=R_2\bXi^{(2)}$ are i.i.d.~copies of $\bxi$, and $\bzeta=\bxi_1+\bxi_2$; since each $\bxi_i$ is radially symmetric, $R_i$ and $\bXi^{(i)}$ are independent. We first find the conditional density $z_{R_1,R_2}(r,\bth)$ of $\bzeta$ under $\p$ given $R_1$ and $R_2$, where $r$ and $\bth$ are the radial and polar components. Observe that the conditional distribution of $\bzeta$ given $R_1$ and $R_2$ is spherically symmetric. We have conditionally,
\begin{align*}
    \p(\n{\bzeta}\geq r\mid R_1,R_2)&=\p(\n{(R_1,\z)+R_2\bXi}\geq r)\\
    &=\p\left(R_1^2+2R_1R_2\Xi_1+R_2^2\geq r^2\right)=\p\left(\Xi_1\geq \frac{r^2-R_1^2-R_2^2}{2R_1R_2}\right).
\end{align*}
By Theorem 2.10 of \citep{fang2018symmetric}, for $y\in[-1,1]$ and some $C(d)\geq 0$, 
$$\p\left(\Xi_1\geq y\right)=C(d)\int_{-1}^y(1-t^2)^{(d-3)/2}\d t.$$
Altogether it follows that for some $C(d)>0$,
$$z_{R_1,R_2}(r,\bth)=C(d)r\left(1-\left( \frac{r^2-R_1^2-R_2^2}{2R_1R_2}\right)^2\right)^{\frac{d-3}{2}},~r\in[|R_1-R_2|,R_1+R_2],~\bth\in\mathbb S^{d-1}.$$
Therefore, the conditional law of  $\bzeta$ under $\p$, and hence also the unconditional law under $\q$, has a density in $\R^d$.
If $\p(R=0)>0$, the laws have densities except for an atom at zero, and the densities are non-trivial since $\p(R=0)<1$.  
By Lemma 4 in Section XV.4 of \citep{feller1991introduction}, \eqref{eq:nonlattice} holds with $\bzeta$ in place of $\bxi$, and hence \eqref{eq:nonlattice} holds since $\bzeta$ is a sum of two i.i.d.~copies of $\bxi$. The final claim follows from Theorem 2(b) of \citep{bhattacharya1978validity}, by noting that the sum in (1.14) therein is $o(1)$ as $n\to\infty$ uniformly on compact sets.
\end{proof}

\begin{proof}[Proof of Proposition \ref{prop:ulb gaussian approx}]
Let $\Lambda(\lambda)=\log\phi_{\xi^{(1)}}(\lambda)=\log\E[e^{\lambda{\xi^{(1)}}}]$. It is clear from (A2) that ${\xi^{(1)}}$ is non-lattice. The Bahadur--Rao theorem (Theorem 3.7.4 of \citep{dembo2009large}) then implies that uniformly in $c\in[-C\log n,C\log n]$,
\begin{align}
    \p(F_n\geq m_n+c)\asymp \frac{e^{-nI((m_n+c)/n)}}{\sqrt{n}}\asymp\frac{e^{-nI(c_1)-(c-3\log n/2)I'(c_1)}}{\sqrt{n}}=n\rho^{-n}e^{-\bl c}.\label{eq:BR}
\end{align}
By definition $\bl=I'(c_1)$, the constant  $\bl>0$ is such that the measure $\q$ defined by 
\begin{align}
    \frac{\d\q}{\d\p}(x):=e^{\bl x-\Lambda(\bl)}\label{eq:dq/dp}
\end{align}satisfies that under $\q$, $\{X_i\}_{i\in\N}$ are i.i.d.~with mean $c_1$. The random variable $(\widetilde{X}_i,\bY_i)$ is then centered under $\q$ where $\widetilde{X}_i=X_i-c_1$. Define $\widetilde{F}_n:=\widetilde{X}_1+\dots+\widetilde{X}_n$. It follows that 
\begin{align*}
    &\hspace{0.5cm}\p(F_n\geq m_n+c,\,\bJ_n\in B_\z(u(n)))\\
    &=e^{n\Lambda(\bl)}\E_{\q}\left[e^{-\bl F_n}\bone_{\{F_n\geq m_n+c,\,\bJ_n\in B_\z(u(n))\}}\right]\\
    &\asymp e^{n\Lambda(\bl)-m_n\bl}\sum_{j=c+1}^\infty e^{-\bl j}\q\left(\widetilde{F}_n\in\Big[ j-1-\frac{3}{2\bl}\,\log n,j-\frac{3}{2\bl}\,\log n\Big),\,\bJ_n\in B_\z(u(n))\right)\\
    &=\rho^{-n}n^{3/2}\sum_{j=c+1}^\infty e^{-\bl j}\q\left(\frac{(\widetilde{F}_n,\bJ_n)}{\sqrt{n}}\in\Big[ \frac{j-1-\frac{3}{2\bl}\,\log n}{\sqrt{n}},\frac{j-\frac{3}{2\bl}\,\log n}{\sqrt{n}}\Big)\times  B_\z\Big(\frac{u(n)}{\sqrt{n}}\Big)\right).
\end{align*}
Note that for $j\geq \sqrt{n}$, we may bound the probability from above by $1$.  For $j\le \sqrt{n}$, the sets
$$E_{j,n}:=\Big[ \frac{j-1-\frac{3}{2\bl}\,\log n}{\sqrt{n}},\frac{j-\frac{3}{2\bl}\,\log n}{\sqrt{n}}\Big)\times  B_\z\Big(\frac{u(n)}{\sqrt{n}}\Big)$$
are bounded as $n\to\infty$. Therefore, by Lemma \ref{lemma:edgeworthok}, with $\mu_{\z,V}$ denoting the centered  Gaussian measure with covariance matrix $V$ on $\R^d$, 
\begin{align}
    \q\left(\frac{(\widetilde{F}_n,\bJ_n)}{\sqrt{n}}\in E_{j,n}\right)=(1+o(1))\mu_{\z,V}\left(E_{j,n}\right)\asymp u(n)^{d-1}n^{-d/2},\label{eq:ejn step}
\end{align}
where the last step follows from the multivariate Gaussian density formula and that the positive definite matrix $V$ depends only on the law of $\bxi$. Altogether this yields
\begin{align*}
    \p(\bJ_n\in B_\z(u(n))\mid F_n\geq m_n)&\ll (n\rho^{-n}e^{-\bl c})^{-1} \rho^{-n}n^{3/2}\sum_{j=c+1}^\infty e^{-\bl j}u(n)^{d-1}n^{-d/2}\\
    &\ll u(n)^{d-1}n^{-(d-1)/2}.
\end{align*}
    This establishes \eqref{eq:ub}.

Using a similar approach, we see that
$$\p(F_n\in[ m_n+c,m_n+c+a(n)])\asymp  n\rho^{-n}e^{-\bl c}$$
and that
\begin{align*}
    &\hspace{0.5cm}\p(\bJ_n\in B_\by(1),~ F_n\in[ m_n+c,m_n+c+a(n)])\\
    &\asymp \rho^{-n}n^{3/2}\sum_{j=c+1}^{c+a(n)} e^{-\bl j}\q\left(\frac{(\widetilde{F}_n,\bJ_n)}{\sqrt{n}}\in\Big[ \frac{j-1-\frac{3}{2\bl}\,\log n}{\sqrt{n}},\frac{j-\frac{3}{2\bl}\,\log n}{\sqrt{n}}\Big)\times  B_{\frac{\by}{\sqrt{n}}}\Big(\frac{1}{\sqrt{n}}\Big)\right)\\
    &\gg \rho^{-n}n^{-(d-3)/2}e^{-\bl c}.
\end{align*}
This proves \eqref{eq:lb}.\footnote{We have implicitly used that $a(n)\geq 1$, but the general case $a(n)\gg 1$ proceeds with the same proof by decomposing $[ m_n+c,m_n+c+a(n)]$ into sub-intervals of smaller length than one.} The same computation also yields
\begin{align*}
    \p(\bJ_n\in B_\z(K\sqrt{n}),~ F_n\geq m_n+c)&\ll \rho^{-n}n^{3/2}n^{-1/2}=n\rho^{-n},
\end{align*}where the asymptotic constant does not depend on $K$. Together with \eqref{eq:BR} completes the proof of \eqref{eq:ub2}. 
\end{proof}

 \begin{proof}[Proof of Lemma \ref{lemma:weak uniform LD}]
We use a similar change of measure argument as in the proof of Proposition \ref{prop:ulb gaussian approx}. Denote by $I$ the large deviation rate function for the $\R^d$-valued random variable $\bxi$ (defined in \eqref{eq:I long}). For $\widehat{\bc}_2:=\nabla I(\bz/(Ck))$, we define the tilted measure $\q$ by $\d \q/\d\p(\bx):=e^{\widehat{\bc}_2\cdot\bx-\Lambda(\widehat{\bc}_2)}$, where we recall $\Lambda=\log\phi_\bxi$ is the log-moment generating function. Under $\q$, each $\bxi_i$ is i.i.d.~with mean $\bz/(Ck)$. It follows that
\begin{align}
    \p((F_{Ck},\bJ_{Ck})\in B_{\bz})&=e^{Ck\Lambda(\widehat{\bc}_2)}\E_\q[e^{-\widehat{\bc}_2\cdot(F_{Ck},\bJ_{Ck})}\bone_{\{(F_{Ck},\bJ_{Ck})\in B_{\bz}\}}]\nonumber\\
    &\gg e^{Ck\Lambda(\widehat{\bc}_2)-\widehat{\bc}_2\cdot\bz}\q\left(\left(\frac{F_{Ck}}{\sqrt{Ck}},\frac{\bJ_{Ck}}{\sqrt{Ck}}\right)\in \frac{B_{\bz}}{\sqrt{Ck}}\right)\nonumber\\
    &= e^{Ck\Lambda(\widehat{\bc}_2)-\widehat{\bc}_2\cdot\bz}\q\left(\left(\frac{\widetilde{F}_{Ck}}{\sqrt{Ck}},\frac{\widetilde{\bJ}_{Ck}}{\sqrt{Ck}}\right)\in  \frac{B_{\z}}{\sqrt{Ck}}\right),\label{eq:tilt2}
\end{align}
where each $(\widetilde{X_i},\widetilde{\bY}_i)=(X_i,\bY_i)-\bz/(Ck)$ is a centered random variable and $(\widetilde{F}_{Ck},\widetilde{\bJ}_{Ck})=\sum_{j=1}^{Ck}(\widetilde{X_i},\widetilde{\bY}_i)$. Using definition of $\widehat{\bc}_2$, we have 
\begin{align}
\Lambda(\widehat{\bc}_2)=\nabla I\Big(\frac{\bz}{Ck}\Big)\cdot\frac{\bz}{Ck}-I\Big(\frac{\bz}{Ck}\Big).\label{eq:Lambda}
\end{align}
In addition, Lemma \ref{lemma:edgeworthok} yields\footnote{Lemma \ref{lemma:edgeworthok} is stated for radially symmetric distributions, but the part that derives \eqref{eq:edgeworth} from \eqref{eq:nonlattice} does not depend on the symmetry and applies to general non-lattice distributions. Note that the non-lattice condition (A5) is preserved under exponential tilts.}
\begin{align}
   \q\left(\left(\frac{\widetilde{F}_{Ck}}{\sqrt{Ck}},\frac{\widetilde{\bJ}_{Ck}}{\sqrt{Ck}}\right)\in  \frac{B_{\z}}{\sqrt{Ck}}\right)\gg (Ck)^{-d/2}.\label{eq:edge2}
\end{align}
Combining \eqref{eq:tilt2}, \eqref{eq:Lambda}, and \eqref{eq:edge2} yields that
$$\p((F_{Ck},\bJ_{Ck})\in B_{\bz})\gg e^{-CkI(\bz/(Ck))}(Ck)^{-d/2}.$$
Recalling that $I$ is convex, $C^1$ on its domain, and $I(\z)=\nabla I(\z)=0$ (by Jensen's inequality, $\E[e^{\bla\cdot\bxi}]$ is minimized at $\bla=\z$ when $\bxi$ is centered), we have uniformly for $\bz$ with $\n{\bz}\leq C_1k$,
$$CI\Big(\frac{\bz}{Ck}\Big)\leq C\sup_{\n{\bx}_2=C_1}I\Big(\frac{\bx}{C}\Big)\leq C_1\sup_{\n{\bx}_2=C_1}\n{\nabla I\Big(\frac{\bx}{C}\Big)}_2\leq \frac{c_3}{2}$$
for $C\geq C_2$ as we pick $C_2$ large enough. 
This concludes the proof. 
 \end{proof}

\begin{proof}[Proof of Proposition \ref{prop:ulb gaussian approx long}]
The Bahadur--Rao theorem (Theorem 3.7.4 of \citep{dembo2009large}) yields that uniformly in $c\in[-C\log n,C\log n]$,
\begin{align}
    \p(F_n\geq x+c)\asymp n^{-1/2}e^{-nI(x/n)-cI'(x/n)}.\label{eq:BR long}
\end{align}
Mimicking \eqref{eq:dq/dp}, define the probability measure $\q$ by
\begin{align}
    \frac{\d\q}{\d\p}(\bx):=e^{\bc_2\cdot\bx-\Lambda(\bc_2)}.\label{eq:dqdp long}
\end{align}
It follows that $\E_\q[\bxi]=(x/n,\z)$. That is, the random variable $(\widetilde{X}_i,\bY_i)$ is then centered under $\q$ where $\widetilde{X}_i=X_i-x/n$. Define $\widetilde{F}_n:=\widetilde{X}_1+\dots+\widetilde{X}_n$. It follows that
\begin{align*}
    &\hspace{0.5cm}\p(F_n\geq x+c,\,\bJ_n\in B_\z(1))\\
    &=e^{n\Lambda(\bc_2)}\E_{\q}\left[e^{-\bc_2 \cdot(F_n,\bJ_n)}\bone_{\{F_n\geq x+c,\,\bJ_n\in B_\z(1)\}}\right]\\
    &\asymp e^{n\Lambda(\bc_2)-\bc_2\cdot (x,\z)}\sum_{j=c+1}^\infty e^{-\bc_2\cdot (j,\z)}\q\left(\widetilde{F}_n\in[ j-1,j),\,\bJ_n\in B_\z(1)\right)\\
    &=e^{-nI(x/n,\z)}\sum_{j=c+1}^\infty e^{-\bc_2\cdot (j,\z)}\q\left(\frac{(\widetilde{F}_n,\bJ_n)}{\sqrt{n}}\in\Big[ \frac{j-1}{\sqrt{n}},\frac{j}{\sqrt{n}}\Big)\times  B_\z\Big(\frac{1}{\sqrt{n}}\Big)\right)\\
    &\asymp e^{-nI(x/n,\z)-\bc_2\cdot (c,\z)}n^{-d/2},
\end{align*}
where in the last step we apply the same argument that leads to \eqref{eq:ejn step} and use that the first coordinate of $\bc_2$ is positive by our assumption $x/n>\underline{c}>0$. Altogether this yields \eqref{eq:ub long}.

Using a similar approach, we see that
$$\p(F_n\in[ x+c,x+c+a(n)])\asymp  n^{-1/2}e^{-nI(x/n)-cI'(x/n)}$$
and that
\begin{align*}
    &\hspace{0.5cm}\p(\bJ_n\in B_\by(1),~ F_n\in[ x+c,x+c+a(n)])\\
    &\asymp e^{n\Lambda(\bc_2)-\bc_2\cdot(x,\by)}
    \sum_{j=c+1}^{c+a(n)} e^{-\bc_2 \cdot(j,\z)}\q\left(\frac{(\widetilde{F}_n,\bJ_n)}{\sqrt{n}}\in\Big[ \frac{j-1}{\sqrt{n}},\frac{j}{\sqrt{n}}\Big)\times  B_{\frac{\by}{\sqrt{n}}}\Big(\frac{1}{\sqrt{n}}\Big)\right)\\
    &\gg e^{-nI(x/n,\z)-\bc_2\cdot (c,\by)}n^{-d/2}.
\end{align*}
This proves \eqref{eq:lb long}.  
\end{proof}

\end{appendix}

\bibliographystyle{imsart-number} 
\bibliography{reference}

\end{document}